\documentclass[]{siamart190516}

\usepackage{algorithm}
\usepackage{algpseudocode}
\usepackage{amsmath}
\usepackage{mathtools}
\usepackage{amssymb}
\usepackage{bm}
\usepackage{array}
\usepackage{bbm}
\usepackage{booktabs}
\usepackage{dsfont}
\usepackage{float}
\usepackage{graphicx}
\usepackage{longtable}
\usepackage{multirow}
\usepackage[caption=false]{subfig}
\usepackage{tikz}
\usepackage[normalem]{ulem}
\usepackage{morefloats}
\usepackage{hyperref}
\usepackage[sort]{cite}
\usepackage{rotating}

\ifpdf
  \DeclareGraphicsExtensions{.eps,.pdf,.png,.jpg}
\else
  \DeclareGraphicsExtensions{.eps}
\fi

\renewcommand{\Re}{\mathbb{R}}

\newcommand{\Id}{\mathbf{I}}

\newcommand{\J}{\mathbf{J}}

\newcommand{\mbf}[1]{\mathbf{#1}}

\newcommand{\A}{\mathbf{A}}

\newsiamremark{remark}{Remark}
\newsiamremark{assumption}{Assumption}



\newcolumntype{P}[1]{>{\centering\arraybackslash\vspace{-1ex}}m{#1}<{\vspace*{-1ex}}}

\DeclareMathOperator*{\argmin}{argmin}

\newif\iflong
\longtrue

\headers{Linearly Implicit Multistep Methods}{R. Glandon, M. Narayanamurthi, and A. Sandu}

\title{Linearly Implicit Multistep Methods for Time Integration \thanks{Submitted to the editors on 5/11/2020. \funding{This work was funded by awards NSF CCF--1613905, NSF ACI--1709727, AFOSR DDDAS FA9550-17-1-0015, and by the Computational Science Laboratory at Virginia Tech.}}}

\author{
Ross Glandon \thanks{Computational Science Laboratory, Department of Computer Science, Virginia Tech.  Blacksburg, Virginia 24060 (\email{rossg42@vt.edu})}
\and
Mahesh Narayanamurthi \thanks{Computational Science Laboratory, Department of Computer Science, Virginia Tech.  Blacksburg, Virginia 24060 (\email{maheshnm@vt.edu})}
\and
Adrian Sandu \thanks{Computational Science Laboratory, Department of Computer Science, Virginia Tech.  Blacksburg, Virginia 24060 (\email{sandu@cs.vt.edu})}
}

\newif\ifnosup
\nosuptrue

\ifnosup
\else
\fi

\usepackage{csl}

\begin{document}

\iftrue
\cslauthor{Ross Glandon, Mahesh Narayanamurthi, and Adrian Sandu}
\cslyear{19}
\cslreportnumber{13}
\cslemail{rossg42@vt.edu, maheshnm@vt.edu, sandu@cs.vt.edu}
\csltitle{Linearly Implicit Multistep Methods for Time Integration}
\csltitlepage
\fi

\maketitle

\begin{abstract}
Time integration methods for solving initial value problems are an important component of many scientific and engineering simulations. Implicit time integrators are desirable for their  stability properties, significantly relaxing restrictions on timestep size. However, implicit methods require solutions to one or more systems of nonlinear equations at each timestep, which for large simulations can be prohibitively expensive. This paper introduces a new family of linearly implicit multistep methods ({\sc Limm}), which only requires the solution of one linear system per timestep. Order conditions and stability theory for these methods are presented, as well as design and implementation considerations. Practical methods of order up to five are developed that have similar error coefficients, but improved stability regions, when compared to the widely used BDF methods. Numerical testing of a self-starting variable stepsize and variable order implementation of the new {\sc Limm} methods shows measurable performance improvement over a similar BDF implementation.
\end{abstract}



\section{Introduction}

In this paper we are concerned with the numerical solution of initial value problems (IVP):
\begin{equation}
\label{eqn:ode}
\frac{dy}{dt} = f(t, y), \quad t_0 \leq t \leq t_F, \quad y(t_0) = y_0; \quad y(t) \in \Re^N,~~ f : \Re \times \Re^N \rightarrow \Re^N.
\end{equation}
Systems of ordinary differential equations (ODEs) of this form appear in a wide variety of scientific and engineering simulations. Some types of simulations, such as those of chemical kinetics or aerosol dynamics \cite{Sandu_2003_KPPSEN2, Sandu_2003_aerosolFramework}, can be directly modeled by systems of densely coupled ODEs. Others, such as those of fluid dynamics \cite{Munson_2013_aa}, arise from the semi-discretization in space of partial differential equations (PDEs) via the method of lines, resulting in a large sparse system of ODEs.

Two time integration families of methods are used to discretize \eqref{eqn:ode}: explicit or implicit. Explicit schemes advance the solution to a new timestep using only information from previous steps, and are simple in structure with very low computational cost per timestep. However, they have stability limitations that result in problem-dependent bounds on the largest allowable stepsizes. For stiff problems, implicit time integration methods that determine solutions at a new timestep using both past and future information, are preferable. These schemes avoid stability-bound stepsize limitations at the cost of solving one or more nonlinear systems per timestep.

As solving large nonlinear systems can be very expensive, many types of linearly implicit methods have been developed that only require solutions of linear systems at each step. Rosenbrock methods \cite{Rosenbrock_1963}
\iflong
(and their many extensions \cite{steihaug1979, Rang2005, Rahunanthan2010, Schmitt1995, Wensch2005, Sandu_2014_ROK, Sandu_2020_BOROK}) 
\fi
are linearized implicit Runge-Kutta methods. Implicit-explicit (IMEX) methods \cite{Sandu_2010_extrapolatedIMEX, Sandu_2014_IMEX-GLM, Sandu_2016_highOrderIMEX-GLM, Akrivis_1999_IMEX, akrivis2015} couple an implicit scheme for the stiff component with an explicit scheme for the non-stiff component of a split problem.  A common splitting treats the nonlinear part of the problem explicitly and the linear part implicitly, therefore avoiding the need for nonlinear solves. Exponential integrators \cite{Hochbruck_1998_exp,Tokman_2011_EPIRK, Sandu_2014_expK, Sandu_2020_stiff_PEXP,Sandu_2019_EPIRKW} also effectively treat a linear-nonlinear problem splitting, with the linear portion solved via an exponential integrating factor. It is also common to linearize implicit methods which would normally require nonlinear solves by taking only a single Newton iteration \cite{yang2015, song2017}, or to replace parts of the implicit scheme with an extrapolation of past values \cite{forti2015, yao2017,Sandu_2014_IMEX-RK,Sandu_2014_IMEX_GLM_Extrap}, at the possible cost of order reduction and/or reduced stability.

In this paper we construct  linearly implicit multistep methods ({\sc Limm}) in much the same way that Rosenbrock methods are obtained from implicit Runge-Kutta methods. We determine order conditions which account for the linearization \cite{Butcher_2016_book, Hairer_book_I}, and solve for a family of $k$-step order $k$ methods for $k = 1,\dots,5$. {\sc Limm} methods have more free coefficients than traditional linear multistep methods with the same number of steps, and this additional freedom enables the optimization of accuracy and stability properties.  The new schemes designed herein have linear stability regions larger than, and error constants comparable to, the widely used BDF \cite{curtiss1952} family of methods.

The remainder of the paper is organized as follows. Section \ref{sec:limm} describes the construction of the {\sc Limm} general form. Section \ref{sec:order_conditions} builds order conditions for $k$-step {\sc Limm} methods of order $p$. Section \ref{sec:stability} considers the linear stability of {\sc Limm} methods, and Section \ref{sec:convergence} their convergence for step sizes not limited by the stiffness of the system. Section \ref{sec:methodderiv} discusses the design of new methods with optimal stability properties and error coefficients, and develops a family of $k$-step order $k$ methods with $k = 1,\dots,5$. Section \ref{sec:impdetails} provides the details necessary for an efficient implementation of {\sc Limm} methods with variable stepsize and variable order. Section \ref{sec:results} reports numerical results comparing {\sc Limm} to BDF methods. Finally, Section \ref{sec:conclusions} draws concluding remarks.

\section{Linearly Implicit Multistep Methods}
\label{sec:limm}

A linear $k$-step method computes the solution of \eqref{eqn:ode} as follows \cite[Chapter III.2]{Hairer_book_I}:
\begin{equation}
\label{eqn:lmm}
\sum_{i=-1}^{k-1} \alpha_i\, y_{n-i} = h_n \sum_{i=-1}^{k-1} \beta_i\, f(t_{n-i}, y_{n-i}),
\end{equation}
where the numerical solution $y_{n-i} \approx y(t_{n-i})$ and the step size is $h_n = t_{n+1}-t_{n}$.
The simplest way to obtain a linearly implicit multistep method is to linearize the implicit evaluation of $f(t_{n+1}, y_{n+1})$ in \eqref{eqn:lmm} about $\left(t_n, y_n\right)$, to obtain:
\begin{equation}
\begin{split}
\label{eqn:lmm-linearized}
\sum_{i=-1}^{k-1} \alpha_i\, y_{n-i} &= h_n \, \sum_{i=0}^{k-1}\, \beta_i'\, f(t_{n-i}, y_{n-i}) \\
 &+ h_n \beta_{-1}\, f_y(t_n, y_n) \, (y_{n+1} - y_n) + h_n^2\beta_{-1}\frac{\partial f}{\partial t}(t_n, y_n),
\end{split}
\end{equation}
where $\beta_0' = \beta_0 + \beta_{-1}$, $\beta_i' = \beta_i$ for $i=1,\dots,k-1$.
\begin{remark}[Autonomous and non-autonomous forms]
In order to simplify the notation in the remainder of the paper, we will make use of the autonomous form of \eqref{eqn:ode}, with $\frac{dy}{dt} = f(y)$ and the notation
\[
f_{n} \coloneqq f(y_{n}), \quad \J_{n} \coloneqq f_y(y_{n}).
\]
As the original non-autonomous form can be recovered by stacking $t$ with the vector $y$ and making corresponding changes to $f$ and $\J$, the methods developed for autonomous problems are also applicable to non-autonomous problems.
\end{remark}
Now, because the direct linearization approach in \eqref{eqn:lmm-linearized} maintains the same number of degrees of freedom, one can expect a degradation of accuracy and stability properties of \eqref{eqn:lmm-linearized} when compared to the standard nonlinear scheme \eqref{eqn:lmm} with the same number of steps. In order to increase the number of degrees of freedom one can generalize the approach by retaining in the formulation past linearized steps scaled by new $\mu$ coefficients:
\begin{equation}
\label{eqn:LIMM_v0}
\sum_{i=-1}^{k-1} \alpha_i\, y_{n-i} = h_n \,\sum_{i=0}^{k-1} \beta_i\, f_{n-i} + h_n \,\sum_{i=-1}^{k-1} \mu_i \,\J_{n-i-1} \,y_{n-i}.
\end{equation}
However, method \eqref{eqn:LIMM_v0} requires the storage of multiple past Jacobian-vector products in addition to past solutions and function values, which is not desirable. 
For this reason we consider a hybrid of the two approaches \eqref{eqn:lmm-linearized} and \eqref{eqn:LIMM_v0}, and, drawing inspiration from Rosenbrock methods \cite{Rosenbrock_1963}, \cite[Section IV.7]{Hairer_book_II}, define the following computational process.
\begin{definition}[{\sc Limm} methods]
A linearly implicit multistep method ({\sc Limm}) advances the numerical solution of  \eqref{eqn:ode} over one step $[t_n,t_{n+1}]$ as follows:
\begin{equation}
\label{eqn:LIMM_definition}
\sum_{i=-1}^{k-1} \alpha_i\, y_{n-i} = h\, \sum_{i=0}^{k-1} \beta_i\, f_{n-i} + h_n \, \J_n\, \left( \sum_{i=-1}^{k-1} \mu_i\, y_{n-i} + h_n \, \sum_{i=0}^{k-1} \nu_i\, f_{n-i}  \right),
\end{equation}
where, without loss of generality, $\alpha_{-1}=1$.
 
A linearly implicit multistep method of W-type ({\sc Limm-w}) advances the numerical solution using \eqref{eqn:LIMM_definition}, where the exact Jacobian $\J_n$ is replaced by an arbitrary matrix $\A_n$. This  matrix is chosen to ensure the numerical stability of the scheme, but without impacting the order of accuracy of the method.

The non-autonomous form of \eqref{eqn:LIMM_definition} adds the following term to the right side of equation \eqref{eqn:LIMM_definition}
\[
 h_n\, \frac{\partial f}{\partial t}(t_n,y_n)\, \left(\sum_{i=-1}^{k-1} \mu_i t_{n-i} + h_n\sum_{i=0}^{k-1} \nu_i\right).
\]
\end{definition}
We note that $k$-step {\sc Limm} schemes \eqref{eqn:LIMM_definition} have $4k+1$ free coefficients (compared to $2k+1$ for classical nonlinear methods \eqref{eqn:lmm}), and require storing only $k$ past solution and $k$ past function values at each step (the same as nonlinear methods \eqref{eqn:lmm}). Evaluation of the  {\sc Limm} numerical solution at each step  \eqref{eqn:LIMM_definition} requires to only solve a linear system of equations, expressed in a computationally efficient form as:
\begin{equation*}
\begin{split}
\left( \mathbf{I} - h_n\, \mu_{-1}\, \J_n\right)\cdot z &= \sum_{i=0}^{k-1} \left( \mu_i/\mu_{-1} -\alpha_i\right)\, y_{n-i} + h_n \, \sum_{i=0}^{k-1} \left(\nu_i/\mu_{-1}+\beta_i\right)\, f_{n-i}, \\
y_{n+1} &=  z - \sum_{i=0}^{k-1} (\mu_i/\mu_{-1})\, y_{n-i} - h_n \, \sum_{i=0}^{k-1} (\nu_i/\mu_{-1})\, f_{n-i}. 
\end{split}
\end{equation*}
Rosenbrock-W methods  \cite[Section IV.7]{Hairer_book_II} seek to increase the computational efficiency of Rosenbrock methods by replacing the exact Jacobian $\J_n$ with an arbitrary matrix $\mathbf{A}_n$, such that the corresponding linear system at each step can be solved more easily. This concept is extended to linearly implicit multistep methods by considering {\sc Limm-w} schemes.
%

\section{Order Conditions}
\label{sec:order_conditions}

Derivation of the order conditions amounts to equating the Taylor series coefficients of the numerical and exact solutions about the current time $t_n$, up to a specific predetermined order. This is the approach taken for building the classical order condition theory for linear multistep methods \cite[Chapter III.2]{Hairer_book_I}. We discuss a direct Taylor series approach to obtain order conditions for {\sc Limm-w} schemes in Section \ref{sec:LIMW-order}. However, constructing Taylor series by repeatedly differentiating the numerical solution becomes increasingly difficult for high-order schemes, and we use a B-series approach \cite{Hairer_book_I,Butcher_2016_book} to derive {\sc Limm} order conditions in Section \ref{sec:LIMM-order}.

We assume that the linearly implicit scheme \eqref{eqn:LIMM_definition} computes the solution at a non-uniform time grid $t_{n-i}$, and express the distances between the solution points as fractions of the current step $h_n = t_{n+1}-t_{n}$:
\begin{subequations}
\label{eqn:abscissae}
\begin{equation}
t_{n-i} = t_n - c_i\, h_n, \quad i=-1,\dots,k; \qquad c_{-1} = -1, \quad c_{0} = 0.
\end{equation}
Many adaptive step analyses consider the ratios of consecutive step sizes $\omega_n$. One can express $c_i$ as a function of these step size ratios as follows:
\begin{equation}
\label{eqn:step-size-ratios}
\begin{split}
\iflong
& t_{n-\ell+1} - t_{n-\ell} = (t_{n+1}-t_{n}) \prod_{j=0}^{\ell+1} \frac{t_{n-j}-t_{n-j-1}}{t_{n-j+1}-t_{n-j}} = h\, \prod_{j=0}^{\ell+1} \omega_{n-j}^{-1}, \\
& t_{n} = t_{n-i}  + \sum_{\ell=1}^{i} (t_{n-\ell+1} - t_{n-\ell}) = t_{n-k}  + h\, \sum_{\ell=1}^{i} \prod_{j=0}^{\ell+1} \omega_{n-j}^{-1}, \\
\fi
& 
\omega_n = \frac{t_{n+1}-t_{n}}{t_{n}-t_{n-1}}, \qquad
c_i = \sum_{\ell=1}^{i} \prod_{j=0}^{\ell+1} \omega_{n-j}^{-1}, \qquad i \ge 1.
\end{split}
\end{equation}
\end{subequations}

\begin{remark}[Method coefficients notation]
When operating on this non-uniform time grid, all the coefficients, $\alpha_i$, $\beta_i$, $\mu_i$, and $\nu_i$, $i \in [-1,k-1]$, of method \eqref{eqn:LIMM_definition} are, in general, functions of the stepsize fractions $c_j$, $j \in [1,k]$ (or analogously, of the stepsize ratios $\omega_n$). In order to simplify our notation, we will leave out the function notation unless it is needed for clarity, and will explicitly call out when we are only considering fixed-stepsize methods (as in the latter half of Section \ref{sec:methodderiv}).
\end{remark}

\subsection{Order conditions for {\sc Limm-w} methods}
\label{sec:LIMW-order}

The order conditions for {\sc Limm-w}  methods \eqref{eqn:LIMM_definition} can be obtained by the standard Taylor series approach, leading to the following result.
\begin{theorem}[{\sc Limm-w} order conditions]
The {\sc Limm-w}  method \eqref{eqn:LIMM_definition} has order of consistency $p$ if and only if its coefficients satisfy:
\begin{subequations}
\label{eqn:LIMW-order-conditions}
\begin{align}
\label{eqn:LIMW-order-conditions-0-alpha}
 \sum_{i=-1}^{k-1} \alpha_i &= 0,  \\
\label{eqn:LIMW-order-conditions-0-mu}
\sum_{i=-1}^{k-1} \mu_i &= 0, \\
\label{eqn:LIMW-order-conditions-traditional}
\sum_{i=-1}^{k-1}  \alpha_i\,c_i^\ell + \ell\, \sum_{i=0}^{k-1} \beta_i\, c_i^{\ell-1} &= 0,  \qquad \ell = 1,\dots,p, \\
\label{eqn:LIMW-order-conditions-mu}
\sum_{i=-1}^{k-1} \mu_i\, c_i^{\ell-1}  - (\ell-1)\, \sum_{i=0}^{k-1} \nu_i\, c_i^{\ell-2} &= 0, \qquad \ell = 2,\dots,p. 
\end{align}
\end{subequations}
\end{theorem}
\begin{proof}
We show that the local truncation error  is of order $p+1$. To this end we apply the method \eqref{eqn:LIMM_definition} starting with exact past solution values:
\begin{equation}
\label{eqn:LIMM_lte}
\begin{split}
y_{n+1} + \sum_{i=0}^{k-1} \alpha_i\, y(t_{n-i}) &= h_n \, \sum_{i=0}^{k-1} \beta_i\, y'(t_{n-i}) + h_n \, \A_n\, \sum_{i=0}^{k-1} \mu_i\, y(t_{n-i}) \\
&\quad + h_n\,\mu_{-1}\, \A_n\, y_{n+1} + h_n\, \A_n\, \sum_{i=0}^{k-1} \nu_i\, h_n\, f(t_{n-i},y(t_{n-i})),
\end{split}
\end{equation}
and show that $y_{n+1} - y(t_{n+1}) \sim \mathcal{O}(h_n^{p+1})$.
We insert the exact solution in \eqref{eqn:LIMM_lte}:
\begin{equation}
\label{eqn:LIMM_exact}
\begin{split}
\sum_{i=-1}^{k-1} \alpha_i\, y(t_{n-i}) & = h_n\, \sum_{i=0}^{k-1} \beta_i\, y'(t_{n-i}) + h_n\, \A_n\, \sum_{i=-1}^{k-1} \mu_i\, y(t_{n-i})  \\
&\quad + h_n^2\, \A_n\, \sum_{i=0}^{k-1} \nu_i\, y'(t_{n-i}) + r_n, 
\end{split}
\end{equation}
where $r_n$ is the local residual, and expand in Taylor series about the current time $t_n$
\begin{equation*}
\begin{split}
& \sum_{\ell \ge 0}  \sum_{i=-1}^{k-1} \alpha_i\,  \frac{(-c_i)^\ell\,h_n^\ell}{\ell!} y^{(\ell)}(t_n) 
 = \sum_{\ell \ge 1} \sum_{i=0}^{k-1} \beta_i\, \frac{(-c_i)^{\ell-1}\,h_n^{\ell}}{(\ell-1)!} y^{(\ell)}(t_n) \\
&+ \A_n\,\sum_{\ell \ge 1}  \sum_{i=-1}^{k-1} \mu_i\, \frac{(-c_i)^{\ell-1}\,h_n^\ell}{(\ell-1)!} y^{(\ell-1)}(t_n) 
 + \A_n\,\sum_{\ell \ge 2} \sum_{i=0}^{k-1} \nu_i\, \frac{(-c_i)^{\ell-2}\,h_n^{\ell}}{(\ell-2)!} y^{(\ell-1)}(t_n) + r_n.
\end{split}
\end{equation*}
Equating powers of $h_n$ on both sides of the equality up to power $p$ yields \eqref{eqn:LIMW-order-conditions}, and $r_n \sim \mathcal{O}(h_n^{p+1})$.
\iflong
For $h^0$ we obtain \eqref{eqn:LIMW-order-conditions-0-alpha}.
For $h^\ell$, $1 \le \ell \le p$ we have two equations: one for the free terms, which leads to is \eqref{eqn:LIMW-order-conditions-traditional}, and another one for the terms multiplied by $\A_n$, which leads to \eqref{eqn:LIMW-order-conditions-0-mu} for $\ell=1$ and to \eqref{eqn:LIMW-order-conditions-mu} for $\ell \ge 2$.
\fi
Subtracting \eqref{eqn:LIMM_exact} from \eqref{eqn:LIMM_exact-initialization} leads to the local truncation error:
\begin{equation}
\label{eqn:lte-residual}
y_{n+1} - y(t_{n+1}) = -\left(\mathbf{I} - h_n \,\A_n\right)^{-1}\,r_n,
\end{equation}
and therefore $y_{n+1} - y(t_{n+1})  \sim \mathcal{O}(h_n^{p+1})$ in the asymptotic case $h_n \to 0$ \cite[Chapter III.2]{Hairer_book_I}. Consequently the method has order $p$. 
\end{proof}

\begin{remark}[Stiff case]
From \eqref{eqn:LIMM_exact} we see that the residual has the form $r_n = q_n + h_n\A_n\, s_n$, where imposing 
\eqref{eqn:LIMW-order-conditions-0-alpha}, \eqref{eqn:LIMW-order-conditions-traditional} leads to $q_n \sim \mathcal{O}(h_n^{p+1})$, and imposing 
\eqref{eqn:LIMW-order-conditions-0-mu}, \eqref{eqn:LIMW-order-conditions-mu} leads to $s_n \sim \mathcal{O}(h_n^{p})$. From \eqref{eqn:lte-residual} the local truncation error is
\begin{equation}
\label{eqn:lte-residual-stiff}
y_{n+1} - y(t_{n+1}) = -\left(\mathbf{I} - h_n \,\A_n\right)^{-1}\,(q_n + s_n) + s_n.
\end{equation}
In the asymptotic case where $h_n \,\A_n \to 0$ the $s_n$ components from the two terms cancel, leaving $y_{n+1} - y(t_{n+1}) = q_n + \mathcal{O}(h_n)\,s_n$, which recovers Theorem \ref{sec:LIMW-order}.
Consider now the stiff case where $h_n \to 0$ but $\Vert h_n \,\A_n \Vert \not\to 0$, and the above cancellation does not happen.
Assume that the matrix $\A_n$ in \eqref{eqn:LIMM_definition} has simple eigenvalues with non-positive real parts, which implies that  $\Vert (\mathbf{I} - h_n \A_n)^{-1} \Vert \le C < \infty$ for all step sizes $h_n > 0$.  The local error \eqref{eqn:lte-residual-stiff} is dominated by $s_n \sim \mathcal{O}(h_n^{p})$. A simple way to recover the full order is to impose conditions \eqref{eqn:LIMW-order-conditions-mu} up to order $p+1$, which gives $s_n \sim \mathcal{O}(h_n^{p+1})$.
\end{remark}

\begin{remark}[Connection with traditional LMM]
Equations \eqref{eqn:LIMW-order-conditions-0-alpha} and \eqref{eqn:LIMW-order-conditions-traditional} represent the order conditions for a traditional linear multistep method \cite[Chapter III.2]{Hairer_book_I}. Consequently, the coefficients $\{\alpha_i,\beta_i\}$ correspond to an order $p$ explicit $k$-step method. The coefficients $\mu_i$ and $\nu_i$ are selected such that they only contribute $\mathcal{O}(h_n^{p+1})$ to the local truncation error. 
\end{remark}

\begin{remark}[Coefficients for variable steps]
\label{rem:variable-step-dependency}
The order conditions \eqref{eqn:LIMW-order-conditions} are linear in the unknown method coefficients, with the system matrix depending on stepsize fractions $c_j$, (or analogously, of the stepsize ratios $\omega_j$). Due to the Vandermonde structure the system has a unique solution for any sequence of variable steps with $j - e \le c_j \le j +e$ with $0 \le e < 0.5$ (including the case of  fixed step sizes where $c_j = j$). Moreover, in this case the method coefficients  $\alpha_i$, $\beta_i$, $\mu_i$, $\nu_i$ depend continuously on the stepsize fractions $c_j$ (and analogously, on the stepsize ratios $\omega_j$).  
\end{remark}

Equation \eqref{eqn:LIMW-order-conditions} pose $2p+1$ constraints for the $4k+1$ free coefficients of the method. Simpler order conditions are possible when $\A_n$ is the exact Jacobian $f_y(t_n,y_n)$, or a well specified approximation of it. In this case the direct Taylor series  approach becomes difficult to handle, and  we employ the Butcher series machinery.

In practice we may consider only linearly implicit schemes \eqref{eqn:LIMM_definition} with $\nu_i=0$, $i=0,\dots,k-1$, since the remaining $3k+1$ free coefficients offer sufficient degrees of freedom for good method design. Sections \ref{sec:methodderiv}, \ref{sec:impdetails}, and \ref{sec:results} describe construction and testing of methods that make this simplification, and Section \ref{sec:convergence} considers convergence.

\subsection{Order conditions for {\sc Limm} methods with exact Jacobian}
\label{sec:LIMM-order}

We employ the Butcher series (B-series) formalism to derive order conditions for the {\sc Limm} methods \eqref{eqn:LIMM_definition}.
B-series \cite{Hairer_book_I} offer a representation of Taylor series expansions of numerical and exact solutions as expansions over a set of elementary differentials, represented graphically using rooted-trees. 

We consider the family of trees $\mathcal{T}_1 = \mathcal{T} \cup \{\emptyset\} \cup \{ \tau_\circ\}$, where
$\mathcal{T}$ is the set of Butcher T-trees \cite{Hairer_book_I},  $\emptyset$  denotes the empty tree, $\tau \in \mathcal{T}$ denotes the tree with a single node,  and $\tau_\circ$ denotes a special tree with a single fat node (a color different than that of the nodes of $\mathcal{T}$). If $\mathfrak{t}_1, \dots, \mathfrak{t}_L \in \mathcal{T}$ then $[\mathfrak{t}_1 \dots \mathfrak{t}_L] \in \mathcal{T}$ denotes the tree obtained by joining all $L$ subtrees to a single root. 
Each tree in $\mathcal{T}$ corresponds to a traditional elementary differential \cite{Hairer_book_I}. In addition,  $\mathcal{F}(\emptyset)(y) = y$ and $\mathcal{F}(\tau_\circ)(y) = \J_n \, y$. The latter elementary differential appears in the numerical solution, but not in the exact solution.

A B-series expansion over the set of rooted-trees $\mathcal{T}_1$ is \cite{Hairer_book_I}:
\begin{equation}
	B(\mathsf{a}, y) = \sum_{\mathfrak{t} \in \mathcal{T}_1} \mathsf{a}(\mathfrak{t})\cdot \frac{h^{|\mathfrak{t}|}}{\sigma(\mathfrak{t})}\,\mathcal{F}(\mathfrak{t})(y)
\end{equation}
where
$\mathsf{a}: \mathcal{T}_1  \mapsto \mathbb{R}$ is a function mapping trees to real values (the coefficients of the B-series),
$|\mathfrak{t}|$ the order (or number of nodes) of the tree,
$\sigma(\mathfrak{t})$ is the symmetry of the tree (the number of equivalent rearrangements of $\mathfrak{t}$) \cite{Hairer_book_I}, 
and $\mathcal{F}(\mathfrak{t})(y)$ is the elementary differential corresponding to a tree $\mathfrak{t}$, evaluated at $y$. 
B-series are the gold standard approach to construct the order conditions for one-step methods such as Runge-Kutta \cite{Hairer_book_I,Butcher_2016_book,Sandu_2019_MRI-GARK,Sandu_2020_MR-GARK_Implicit}, Rosenbrock \cite{Sandu_2014_ROK}, and exponential \cite{Sandu_2014_expK,Sandu_2019_EPIRKW,Sandu_2020_stiff_PEXP} schemes.  

To study the {\sc Limm} local truncation error we consider the method \eqref{eqn:LIMM_definition} initialized with the exact solution evaluated at a series of $k$ previous time-steps, $\{t_n, \dots, t_{n - k + 1}\}$:
\begin{equation}
\label{eqn:LIMM_exact-initialization}
\begin{split}
& y_{n+1} + \sum_{i=0}^{k-1} \alpha_i\, y(t_{n-i}) = h_n\, \sum_{i=0}^{k-1} \beta_i\, y'(t_{n-i}) + h_n\, f_y\big(t_n,y(t_n)\big)\, \sum_{i=0}^{k-1} \mu_i\, y(t_{n-i}) \\
&\qquad + h_n\,\mu_{-1}\, f_y\bigl(t_n,y(t_n)\bigr)\, y_{n+1}  + h_n\, f_y\big(t_n,y(t_n)\big)\, \sum_{i=0}^{k-1} \nu_i\, \left( h_n\,y'(t_{n-i}) \right).
\end{split}
\end{equation}
The time-shifted exact solution \eqref{eqn:ode} has the following B-series expansion over $\mathcal{T}_1$:
\begin{subequations}
\begin{equation}
\label{eqn:exact_y_b_series}
y(t_n+c\,h_n) \sim B(a_{(c)},y(t_n)), \qquad a_{(c)}(\mathfrak{t}) 
= \begin{cases}
0, & \mathfrak{t} = \tau_\circ, \\  
1, & \mathfrak{t} = \emptyset, \\  
\frac{c^{|\mathfrak{t}|}}{\gamma(\mathfrak{t})}, & |\mathfrak{t}| \ge 1. 
\end{cases}
\end{equation}
where $\mathfrak{t} \in \mathcal{T}$ is a T-tree, and  $\gamma(\mathfrak{t})$ is defined as the product of order of the tree $\mathfrak{t}$ and of all its subtrees \cite{Hairer_book_I}. Similarly, the time-shifted exact solution derivative of  \eqref{eqn:ode} has the following B-series expansion over $\mathcal{T}_1$:
\begin{equation}
\label{eqn:exact_f_b_series}
 h_n\,y'(t_n+c\,h_n) \sim B(\mathrm{D}a_{(c)},y(t_n)), \qquad (\mathrm{D}a_{(c)})(\mathfrak{t}) = 
\begin{cases}
0, & \mathfrak{t} = \tau_\circ, \\  
0, & \mathfrak{t} = \emptyset, \\  
\frac{|\mathfrak{t}|\,c^{|\mathfrak{t}|-1}}{\gamma(\mathfrak{t})}, & |\mathfrak{t}| \ge 1. 
\end{cases}
\end{equation}
where $\tau \in \mathcal{T}$ is the T-tree with a single node, and $[\mathfrak{u}_1,\dots,\mathfrak{u}_L]$ denotes the tree where the root has $L$ children, each the root of a subtree $\mathfrak{u}_i$.

%
The Jacobian matrix times a B-series is another B-series over $\mathcal{T}_1$ \cite{Butcher_2010_exp-order}:
\begin{equation}
\label{eqn:Ja}
\begin{split}
& h_n\, f_y(t_n,y(t_{n})) \cdot \textnormal{B}(a,y(t_{n}) \sim \textnormal{B}\bigl( \mathrm{J} a, y(t_{n})\bigr), \\
& 
(\mathrm{J} a)(\mathfrak{t}) = \begin{cases}
a(\emptyset), & \mathfrak{t} = \tau_\circ, \\  
0 & \mathfrak{t}=\emptyset,~\mathfrak{t}=\tau, \\
 a(\mathfrak{u}) & \textnormal{for}~\mathfrak{t}=[\mathfrak{u}], ~ \mathfrak{u} \ne \emptyset, \\
 0 & \textnormal{otherwise}.
 \end{cases}
\end{split}
\end{equation}
\end{subequations}
The numerical solution of \eqref{eqn:LIMM_exact-initialization} is a B-series over $\mathcal{T}_1$:
\begin{equation}
\label{eqn:numerical_y_b_series}
y_{n+1} \sim B(\theta,y(t_n)).
\end{equation}
Next, in \eqref{eqn:LIMM_exact-initialization} replace each quantity by the corresponding B-series. Use the exact solutions \eqref{eqn:exact_y_b_series} and their derivatives \eqref{eqn:exact_f_b_series} at the current and past times $t_{n-k},\dots,t_{n}$, and the numerical solution  \eqref{eqn:numerical_y_b_series} at $t_{n+1}$, to obtain:
\begin{equation*}
\begin{split}
\theta = \mu_{-1}\, (\mathrm{J} \theta) - \sum_{i=0}^{k-1} \alpha_i\, a_{(-c_i)} + \sum_{i=0}^{k-1} \beta_i\, (\mathrm{D}a_{(-c_i)}) \\
+ \sum_{i=0}^{k-1} \mu_i\, (\mathrm{J} a_{(-c_i)}) 
+ \sum_{i=0}^{k-1} \nu_i\, (\mathrm{J}(\mathrm{D}a_{(-c_i)})). 
\end{split}
\end{equation*}
This leads to the following recursive definition of $\theta$ over the set $\mathcal{T}_1$ of trees:
\begin{equation}
\label{eqn:LIMM_gen-bseries}
\theta(\mathfrak{t}) = 
\begin{cases}
- \sum_{i=0}^{k-1} \alpha_i, & \mathfrak{t} = \emptyset, \\[6pt]
 \sum_{i=0}^{k-1} \mu_i + \mu_{-1}\, (- \sum_{i=0}^{k-1} \alpha_i), & \mathfrak{t} = \tau_\circ, \\[6pt]
- \sum_{i=0}^{k-1} \alpha_i\, (-c_i) + \sum_{i=0}^{k-1} \beta_i, & \mathfrak{t} = \tau, \\[6pt]
\mu_{-1}\, \theta(\mathfrak{u}_1) - \frac{1}{\gamma(\mathfrak{t})}\,\sum_{i=0}^{k-1} \alpha_i\, (-c_i)^{|\mathfrak{t}|} & \\
\quad + \frac{|\mathfrak{t}|}{\gamma(\mathfrak{t})}\,\sum_{i=0}^{k-1} (\beta_i + \mu_i)\, (-c_i)^{|\mathfrak{t}|-1}  \\
\quad + \frac{|\mathfrak{t}|\,(|\mathfrak{t}|-1)}{\gamma(\mathfrak{t})}\,\sum_{i=0}^{k-1} \nu_i\,(-c_i)^{|\mathfrak{t}|-2},
 &  \mathfrak{t} = [\mathfrak{u}_1], ~ |\mathfrak{u}_1| \ge 1, \\[6pt]
- \frac{1}{\gamma(\mathfrak{t})}\,\sum_{i=0}^{k-1} \alpha_i\, (-c_i)^{|\mathfrak{t}|} + \frac{|\mathfrak{t}|}{\gamma(\mathfrak{t})}\,\sum_{i=0}^{k-1} \beta_i\, (-c_i)^{|\mathfrak{t}|-1}, &  \mathfrak{t} = [\mathfrak{u}_1,\dots,\mathfrak{u}_L],~L\ge 2. 
\end{cases}
\end{equation}
To obtain the order conditions we equate the B-series coefficients of the numerical method \eqref{eqn:LIMM_gen-bseries} with those of the exact solution up to order $p$:
\[
\theta(\mathfrak{t}) = \frac{1}{\gamma(\mathfrak{t})},\quad \forall\, \mathfrak{t} \in \mathcal{T}_1 ~\textnormal{with}~ |\mathfrak{t}| \le p.
\]
For trees with $|\mathfrak{t}| \le 1$ we have:
\begin{equation}
\label{eqn:LIMM_gen-bseries-2}
\begin{split}
\mathfrak{t} = \emptyset:\quad & -\sum_{i=0}^{k-1} \alpha_i = 1 \quad \Leftrightarrow \quad \sum_{i=-1}^{k-1} \alpha_i = 0 \quad (\textnormal{using } \alpha_{-1}=1),  \\[3pt]
\mathfrak{t} = \tau_\circ:\quad & \sum_{i=-1}^k \mu_i = 0  \quad (\textnormal{using condition for } \mathfrak{t} = \emptyset),  \\[3pt]
\mathfrak{t} = \tau:\quad & - \sum_{i=0}^{k-1} \alpha_i\, (-c_i) + \sum_{i=0}^{k-1} \beta_i = 1.
\end{split}
\end{equation}
For trees with $|\mathfrak{t}| \ge 2$ with a multiply branched root ($\mathfrak{t} = [\mathfrak{u}_1,\dots,\mathfrak{u}_L]$, $L\ge 2$) the order condition reads:
\begin{equation}
\label{eqn:LIMM_gen-bseries-3}
 - \sum_{i=-1}^{k-1} \alpha_i\, (-c_i)^{|\mathfrak{t}|} + |\mathfrak{t}|\,\sum_{i=0}^{k-1} \beta_i\, (-c_i)^{|\mathfrak{t}|-1} = 0,
\end{equation}
where we used the relations $\alpha_{-1}=1$ and $c_{-1}=-1$.
For trees with $|\mathfrak{t}| \ge 2$ with a singly branched root ($\mathfrak{t} = [\mathfrak{u}_1]$, $|\mathfrak{u}_1| \ge 1$) the order condition reads:
\begin{equation}
\label{eqn:LIMM_gen-bseries-4}
\begin{split}
- \sum_{i=-1}^{k-1} \alpha_i\, (-c_i)^{|\mathfrak{t}|}  + |\mathfrak{t}|\,\sum_{i=-1}^{k-1} \left( (\beta_i + \mu_i)\, (-c_i)^{|\mathfrak{t}|-1} 
 + (|\mathfrak{t}|-1)\,\nu_i\,(-c_i)^{|\mathfrak{t}|-2} \right) = 0,
\end{split}
\end{equation}
where we formally set $\beta_{-1}=0$, and use the lower order condition $\theta(\mathfrak{u}_1) = 1/\gamma(\mathfrak{u}_1) = |\mathfrak{t}|/\gamma(\mathfrak{t})$. We make the following observations:
\begin{itemize}
\item T-trees of order two, $|\mathfrak{t}|=2$, have a singly branched root and only condition \eqref{eqn:LIMM_gen-bseries-4} is applied. 

\item Equations \eqref{eqn:LIMM_gen-bseries-3} and \eqref{eqn:LIMM_gen-bseries-4} depend only on the order of the tree $|\mathfrak{t}|$, but not on the tree topology. For higher orders $|\mathfrak{t}| \ge 3$ there are T-trees with both singly-branched and multiply-branched roots, therefore both order conditions  \eqref{eqn:LIMM_gen-bseries-3} and \eqref{eqn:LIMM_gen-bseries-4} apply. The only way both equations \eqref{eqn:LIMM_gen-bseries-3} and \eqref{eqn:LIMM_gen-bseries-4} are satisfied is to require that:
\begin{equation*}
\sum_{i=-1}^{k-1} \left( \mu_i\, (-c_i)^{|\mathfrak{t}|-1} 
 + (|\mathfrak{t}|-1)\,\nu_i\,(-c_i)^{|\mathfrak{t}|-2} \right) = 0 \quad \textnormal{for } |\mathfrak{t}| \ge 3.
\end{equation*}
\end{itemize}

Using $c_{-1}=-1$ \eqref{eqn:abscissae},  $\alpha_{-1}=1$, and $\beta_{-1}=0$, we have the following result.
\begin{theorem}[{\sc Limm} order conditions]
The {\sc Limm} method \eqref{eqn:LIMM_definition} has order of consistency $p \ge 1$ if and only if the coefficients satisfy:
\begin{subequations}
\label{eqn:LIMM_order-conditions}
\begin{align}
\label{eqn:LIMM_order-conditions-0-alpha}
\sum_{i=-1}^{k-1} \alpha_i & = 0, \quad \textnormal{(consistency)}\\
\label{eqn:LIMM_order-conditions-0-mu}
\sum_{i=-1}^k \mu_i &= 0, \quad \textnormal{(consistency)} \\
\label{eqn:LIMM_order-conditions-1}
\sum_{i=-1}^{k-1} \alpha_i\, c_i + \sum_{i=0}^{k-1} \beta_i &= 0, \quad \textnormal{(order one)} \\
\label{eqn:LIMM_order-conditions-2}
\sum_{i=-1}^{k-1} \alpha_i\, c_i^{2}  + 2\,\sum_{i=-1}^{k-1} \left( (\beta_i + \mu_i)\, c_i 
 - \nu_i \right) &= 0,  \quad \textnormal{(order two)} \\
\label{eqn:LIMM_order-conditions-traditional}
\sum_{i=-1}^{k-1} \alpha_i\, c_i^{\ell} + \ell\,\sum_{i=0}^{k-1} \beta_i\, c_i^{\ell-1} & = 0, \quad \textnormal{(order } \ell = 3, \dots, p\textnormal{)} \\
\label{eqn:LIMM_order-conditions-mu}
\sum_{i=-1}^{k-1} \mu_i\, c_i^{\ell-1} 
 - (\ell-1)\,\sum_{i=0}^{k-1} \nu_i\,c_i^{\ell-2} & = 0, \quad \textnormal{(order } \ell = 3, \dots, p\textnormal{)}.
\end{align}
\end{subequations}
\end{theorem}

\begin{remark}
A comparison of {\sc Limm} order conditions \eqref{eqn:LIMM_order-conditions} with {\sc Limm-w} order conditions \eqref{eqn:LIMW-order-conditions} reveals that they are the same, except for the second order condition. Specifically, {\sc Limm} order condition \eqref{eqn:LIMM_order-conditions-2} is the sum of {\sc Limm-w} conditions \eqref{eqn:LIMW-order-conditions-traditional} and \eqref{eqn:LIMW-order-conditions-mu} for $\ell=2$. There are $2p$ {\sc Limm} conditions for order $p$, compared to $2p+1$ {\sc Limm-w} conditions.
\end{remark}


\section{Linear Stability Analysis}
\label{sec:stability}

To study linear stability we apply the {\sc Limm} method \eqref{eqn:LIMM_definition} to the Dahlquist test equation
\[
y' = \lambda\, y, \quad y(t_0)=1,
\]
to obtain the numerical solution
\begin{equation}
\sum_{i=-1}^{k-1} \left(\alpha_i - z\, \left(\beta_i + \mu_i\right) - z^2\,\nu_i \right)\, y_{n-i} = 0, \quad  \text{where  } z = h\,\lambda,~ \beta_{-1} = \nu_{-1} =0.
\end{equation}
To solve this homogeneous linear difference equation we substitute $\zeta^{k-i-1}$ for $y_{n-i}$. This yields the relation:
\begin{equation}
\label{eqn:LIMM_stability}
\begin{split}
&\varrho(\zeta) - z\,\sigma(\zeta) - z^2\,\upsilon(\zeta) =  0, \qquad \textnormal{where:}\\
&\varrho(\zeta) \coloneqq \sum_{i=-1}^{k-1} \alpha_i \zeta^{k-i-1}, \quad
\sigma(\zeta) \coloneqq  \sum_{i=-1}^{k-1} \left(\beta_i + \mu_i\right) \zeta^{k-i-1}, \quad
\upsilon(\zeta) \coloneqq  \sum_{i=0}^{k-1} \nu_i\, \zeta^{k-i-1}.
\end{split}
\end{equation}
The linearly implicit multistep method \eqref{eqn:LIMM_definition} is zero-stable if all the roots of the polynomial $\varrho(\zeta)$ lie on or inside the unit circle, with only simple roots on the unit circle \cite[Section III, Def. 3.2]{Hairer_book_I}.

The stability region of the linearly implicit multistep method \eqref{eqn:LIMM_definition} is a subset of the complex plane defined as \cite[Section V, Def. 1.1]{Hairer_book_II}:
\begin{equation}
\label{eqn:stability-region}
\mathcal{S} \coloneqq \left\{z \in \mathbf{C}:\, \begin{array}{l}
  \text{all roots } \zeta_j(z) \text{ of eqn. \eqref{eqn:LIMM_stability} satisfy } \left|\zeta_j(z)\right| \leq 1, \\
  \text{and multiple roots satisfy } \left|\zeta_j(z)\right| < 1
  \end{array} \right\}.
\end{equation}
In order to visualize this stability region we make use of the reverse map
\begin{equation}
\label{eqn:mu}
z(\zeta,\bm{\alpha},\bm{\beta},\bm{\mu},\bm{\nu}) = \begin{cases} 
\frac{\varrho(\zeta)}{\sigma(\zeta)}, & \nu_0 = \cdots = \nu_{k-1} = 0, \\
\frac{-\sigma(\zeta) \pm \sqrt{\sigma(\zeta)^2 + 4\,\upsilon(\zeta)\,\varrho(\zeta)}}{2\,\upsilon(\zeta)}, & \textnormal{otherwise},
\end{cases}
\end{equation}
and evaluate it for $\zeta = e^{i\theta}$,  $0 \leq \theta \leq 2\pi$ to produce the root locus curve. 
Figure \ref{fig:limmwfstabreg} shows stability region plots for a set of methods constructed via the optimization approach described in Section \ref{sec:methodderiv}.

\begin{remark}[Variable step sizes]
\label{rem:variable-step-sizes}
For variable step size the method coefficients, and therefore the polynomials \eqref{eqn:LIMM_stability}, depend on the ratios $\omega_j$ of consecutive step sizes \eqref{eqn:step-size-ratios}  for the last $k$ steps: $\varrho(\omega_{i-k:i-1},\zeta)$, $\sigma(\omega_{i-k:i-1},\zeta)$,  $\upsilon(\omega_{i-k:i-1},\zeta)$. Consequently, the stability region \eqref{eqn:stability-region} and the inverse map \eqref{eqn:mu} also depend on the step size ratios.
\end{remark}

\begin{remark}[Stability matrix]
When all $\nu_i = 0$ the stability region \eqref{eqn:stability-region} is that of an implicit LMM method, and equals the region where the following matrix has all eigenvalues inside the unit disk, with only simple eigenvalues on the unit circle \cite[Section V]{Hairer_book_II}:
\begin{equation}
\label{eqn:limm-stability-matrix}
\mathbf{M}(\omega_{i-k:i-1},z) \coloneqq \begin{bmatrix} \frac{-\alpha_{0:k-2}^T+z\,(\beta_{0:k-2} + \mu_{0:k-2})^T}{1-z\mu_{-1}} &  \frac{-\alpha_{k-1}+z\,(\beta_{k-1} + \mu_{k-1})}{1-z\mu_{-1}} \\ \mathbf{I}_{(k-1)\times(k-1)} &  \boldsymbol{0}_{(k-1) \times 1} \end{bmatrix} \in \Re^{k \times k}.
\end{equation}
\end{remark}

\section{Convergence}
\label{sec:convergence}
%
Here we consider {\sc Limm-w} methods \eqref{eqn:LIMM_definition} with $\nu_i=0$, $i=0,\dots,k-1$, satisfying \eqref{eqn:LIMW-order-conditions} up to order $p$. We write the {\sc Limm-w} method \eqref{eqn:LIMM_definition} as an IMEX LMM scheme applied to a linear-nonlinear partitioning of the system:
\begin{equation}
\label{eqn:LIMM=IMEX}
\begin{split}
\sum_{i=-1}^{k-1} \alpha_i\, y_{n-i} & =  h_n\, \sum_{i=0}^{k-1} \beta_i\, (f_{n-i} - \J_n\,y_{n-i}) + h_n\, \sum_{i=-1}^{k-1} (\mu_i + \beta_i)\, \J_n\,y_{n-i}.
\end{split}
\end{equation}
One sees that the order conditions for the IMEX LMM scheme \eqref{eqn:LIMM=IMEX} are equivalent to the {\sc Limm-w} order conditions \eqref{eqn:LIMW-order-conditions}.

For brevity the following discussion considers the autonomous case, however it can immediately be extended to the non-autonomous case. Assume that, for any $\tau \in [t_0,t_F]$, there is an interval $[\tau - \varepsilon, \tau+\varepsilon]$ such that the function can be locally decomposed into a linear part and a nonlinear remainder:
\begin{equation}
\label{eqn:linear-nonlinear-split}
f(y(t)) = \J_\tau\,y(t) + \mathbf{r}_\tau(y(t)) \qquad \forall t \in (\tau - \varepsilon, \tau+\varepsilon),
\end{equation}
where the matrix $\J_\tau$ is diagonalizable, and all its eigenvalues have non-positive real parts.
$\J_\tau$ captures all the stiffness of the system in a vicinity of the exact trajectory, i.e., the remaining nonlinear parts $\mathbf{r}_\tau(y)$ are non-stiff, and Lipschitz-continuous with moderate Lipschitz constants in a vicinity of the exact solution:
\begin{equation}
\label{eqn:residual-lipschitz}
\Vert \mathbf{r}_\tau(y) - \mathbf{r}_\tau(z) \Vert \le \mathrm{L}_\tau\, \Vert y - z \Vert.
\end{equation}
Assumptions \eqref{eqn:linear-nonlinear-split}, \eqref{eqn:residual-lipschitz} mean that the stiffness is due to linear dynamics only, and that the stiff directions of the system evolution do not change too rapidly along a trajectory \cite{Sandu_2020_IMEX-GLM-Theory}.
Choose a finite number of $\tau_i$ values such that the corresponding intervals $(\tau_i-\varepsilon_i,\tau_i+\varepsilon_i)$ cover the entire compact integration time interval $[t_0,t_F]$. Consequently, we select a finite number of subintervals and on each we have the corresponding decomposition \eqref{eqn:linear-nonlinear-split}. Since we have a finite number of decompositions, without loss of generality, we will carry out the convergence analysis on a single subinterval and a single decomposition \eqref{eqn:linear-nonlinear-split}, with matrix $\J_\ast$; the subscripts $\tau$ will be dropped. We note that the Jacobian of the nonlinear remainder is $\mathbf{r}_{y}(y) = \J(y) - \J_\ast$, and, from \eqref{eqn:residual-lipschitz}, that $\Vert \mathbf{r}_{y}(y) \Vert \le  \mathrm{L}$ in a vicinity of the solution.

\begin{theorem}[Convergence]
\label{thm:convergence}
Apply an order $p$ {\sc Limm-w} scheme \eqref{eqn:LIMM=IMEX} to solve the system \eqref{eqn:ode}. Assume that the eigenvalues of $h\,\J_\ast$ have non-positive real parts, and fall inside the stability region \eqref{eqn:stability-region} (computed for constant step size coefficients) for any $h > 0$. Perform integration using a sequence of steps such that ratios of consecutive step sizes are uniformly bounded below and above $\omega_{\rm \min} \le \omega_j \le \omega_{\rm max}$. The bounds $\omega_{\rm \min} \le 1$ and $\omega_{\rm \max} \ge 1$ are chosen such as to ensure that the method is linearly stable \eqref{eqn:limm-stability-matrix} when applied to integrate the stiff component $\J_\ast y$:
\begin{equation}
\label{eqn:matrix-stability}
\textstyle
\Vert \prod_{i=\ell_1}^{\ell_2} \mathbf{M}(\omega_{i-k:i-1},h_i \J_\ast) \Vert \le C_{\mathbf{M}}\quad \forall h_i : \sum_{i=\ell_1}^{\ell_2} h_i \le t_F-t_0, ~ \forall \ell_2 \ge \ell_1.
\end{equation}
By possibly further restricting the bounds $\omega_{\rm \min}$, $\omega_{\rm max}$, we ensure that \eqref{eqn:step-size-ratios} leads to step size fractions $j - e \le c_j \le j +e$ with $0 \le e < 0.5$. (Note that such bounds always exist, since for uniform steps $\omega_j=1$ the stability equation \eqref{eqn:matrix-stability} holds, and $c_j=j$.) 

Under these assumptions the numerical solution converges with order $p$ to the exact solution for sequences of sufficiently small step sizes $h_j \le h_\ast$, where the upper bound $h_\ast$ is independent of the stiffness of the system.
\end{theorem}

\begin{proof}

First, using Remark \ref{rem:variable-step-dependency} and the assumption on the step size ratio bounds, the method coefficients depend continuously on step size ratios and remain  uniformly bounded: $|\alpha_i(\omega_{n-k:n-1})|$, $|\beta_i(\omega_{n-k:n-1})|$, $|\mu_i(\omega_{n-k:n-1})| \le const$ for all  $\omega_{\rm \min} \le \omega_j \le \omega_{\rm max}$.
For notation brevity in the remaining part of the proof we omit the explicit dependency of method coefficients on step size ratios.

Since $\mu_{-1} > 0$   and all the eigenvalues of $\J_*$ have non-positive real parts the following matrix is uniformly bounded for any step size:
\[
\mathbf{T}_n \coloneqq \mathbf{I}_{N} - h_n\,\mu_{-1} \, \J_* , \quad
\Vert \mathbf{T}_n^{-1} \Vert \le C_{\mathbf{T}}~~\forall\, h_n. 
\]
Replace the exact solution into the method \eqref{eqn:LIMM=IMEX}, then subtract the numerical solution to obtain the following recurrence of global errors:
\begin{equation*}
\begin{split}
\Delta y_{n+1}  
&= h_n\,\boldsymbol{\delta}_n+ h_n\,\boldsymbol{\zeta}_n  + \boldsymbol{\theta}_n + \mathbf{T}_n^{-1}\,R_n, \\
\boldsymbol{\delta}_n &= \mathbf{T}_n^{-1}\, \sum_{i=0}^{k-1} \beta_i\, \Delta\mathbf{r}_n(y_{n-i}), 
\qquad \boldsymbol{\zeta}_n  = \mathbf{T}_n^{-1}\, (\mathbf{r}_y)_n\, \sum_{i=-1}^{k-1} \mu_i \,\Delta y_{n-i}, \\
\boldsymbol{\theta}_n &= \mathbf{T}_n^{-1}\,\sum_{i=0}^{k-1} \left(-\alpha_i + (\beta_i + \mu_i)\,h_n\J_\ast  \right)\, \Delta y_{n-i},
\end{split}
\end{equation*}
where the local truncation error is $R_n \sim \mathcal{O}(h_n^{p+1})$. Let $\mathbb{Y}_{n} = [ y_{n}^T \dots y_{n-k+1}^T ]^T$ be the solution history vector. Taking norms we have that:
\begin{equation*}
\begin{split}
& \Vert \boldsymbol{\delta}_n \Vert \le C_{\mathbf{T}}\, \mathrm{L} \, \sum_{i=0}^{k-1} |\beta_i| \, \max_{i=0,\dots,k-1}\Vert \Delta y_{n-i}\Vert  \le C_1\, \Vert \Delta\mathbb{Y}_{n} \Vert, \\
&\Vert \mathbf{T}_n^{-1}\,h_n\J_\ast  \Vert = \Vert (\mu_{-1})^{-1}\, \left(\mathbf{T}_n^{-1} -\mathbf{I} \right) \Vert \le |\mu_{-1}|^{-1}\, (1 + C_{\mathbf{T}}), \\
&\Vert \boldsymbol{\theta}_n \Vert \le \left( C_{\mathbf{T}}\,\sum_{i=0}^{k-1} |\alpha_i| + \frac{1 + C_{\mathbf{T}}}{|\mu_{-1}|}\,\sum_{i=0}^{k-1} |\beta_i + \mu_i| \right)\, \max_{i=0,\dots,k-1}\Vert \Delta y_{n-i}\Vert \le C_2\, \Vert \Delta\mathbb{Y}_{n} \Vert. 
\end{split}
\end{equation*}
Consider the step size bound $h_\ast$ that is independent of the stiffness of the system:
\[
h_\ast < \left( C_{\mathbf{T}}\,\mathrm{L}\,|\mu_{-1}| \right)^{-1} \quad  \Rightarrow \quad h_\ast \, C_{\mathbf{T}}\,\mathrm{L}\,|\mu_{-1}| = C_3 < 1.
\]
For any $h_n \le h_\ast$ the error in the solution is bounded by:
\begin{equation*}
\begin{split}
\Vert \Delta y_{n+1} \Vert &\le \frac{h_\ast\,C_1 + C_2}{1-C_3}\, \Vert \Delta\mathbb{Y}_{n} \Vert+ \frac{h_\ast\,C_{\mathbf{T}}\,\mathrm{L}}{1-C_3}\, \sum_{i=0}^{k-1} |\mu_i| \,\max_{i=0,\dots,k-1}\Vert \Delta y_{n-i}\Vert  +  \frac{C_{\mathbf{T}}\,\Vert R_n \Vert}{1-C_3} \\
& \Rightarrow  \quad \Vert \boldsymbol{\zeta}_n \Vert \le C_4 \, \Vert \Delta\mathbb{Y}_{n} \Vert + C_5\,\Vert R_n \Vert,
\end{split}
\end{equation*}
where all constants $C_1$ to $C_5$ are independent of the stiffness of the system.

%
The global error vector obeys the recurrence:
\begin{equation}
\label{eqn:solution-recurrence}
\begin{split}
& \Delta \mathbb{Y}_{n+1} =  \mathbf{M} (\omega_{n-k:n-1},h_n\J_\ast) \, \Delta\mathbb{Y}_{n} + \mathbf{e}_1 \otimes  \left( h_n\,(\boldsymbol{\delta}_n+\boldsymbol{\zeta}_n) +  \mathbf{T}_n^{-1}\,R_n \right) \\
&=   \prod_{i=k}^{n} \mathbf{M} (\omega_{i-k:i-1},h_i\J_\ast)  \, \Delta\mathbb{Y}_k  \\
&\quad +  \sum_{\ell=k+1}^{n} \prod_{i=n-\ell}^{n} \mathbf{M} (\omega_{i-k:i-1},h_i\J_\ast)  \cdot  \mathbf{e}_1 \otimes  \left( h_\ell\,(\boldsymbol{\delta}_\ell + \boldsymbol{\zeta}_\ell) +  \mathbf{T}_\ell^{-1}\,R_\ell \right), 
\end{split}
\end{equation}
where $e_1 \in \Re^k$ has the first entry equal to one and all other entries equal to zero, and $\otimes$ is the Kronecker product.
Taking norms leads to the following global error bounds:
\begin{eqnarray*}
\Vert \Delta\mathbb{Y}_{n+1} \Vert 
&\le & C_{\mathbf{M}}\, \Vert \Delta\mathbb{Y}_k \Vert + C_{\mathbf{M}}\,(C_1+C_4)\,\sum_{\ell=k+1}^{n} \left( h_{\rm max}\,\Vert \Delta \mathbb{Y}_{\ell} \Vert + \rho_\ell \right), ~~
\rho_\ell \sim \mathcal{O}\big( h_{\rm max}^{p+1} \big),
\end{eqnarray*}
where $h_{\rm max} \coloneqq \max_\ell h_\ell$.
Solving this recurrence inequality by standard techniques, and assuming the initial state is sufficiently accurate $\Delta\mathbb{Y}_k \sim \mathcal{O}(h_{\rm max}^{p})$, proves the result.
%
\end{proof}


\iflong
\section{Index-1 DAE Solution}
\label{sec:DAE}
Consider the index-1 differential-algebraic equation (DAE, \cite{Hairer_book_II})
\begin{equation}
\label{eqn:LIMM_dae-1}
y' = f(y,z), \quad 0 = g(y,z),   \quad \Rightarrow \quad z = G(y),
\end{equation}
where the sub-Jacobian $g_z$ is nonsingular and has a negative logarithmic norm in a neighborhood of the exact solution.
Here we consider {\sc Limm-w} methods \eqref{eqn:LIMM_definition} with $\nu_i=0$, $i=0,\dots,k-1$ satisfying \eqref{eqn:LIMW-order-conditions-0-mu} up to order $p$.
\iflong
\begin{equation}
\label{eqn:LIMM_spp}
\begin{split}
&\begin{bmatrix} \Id - h\, \mu_{-1}\, (f_{y})_{n} & - h\, \mu_{-1}\, (f_{z})_{n} \\  - h\, \varepsilon^{-1}\,\mu_{-1}\, (g_{y})_{n} & \Id - h\,\varepsilon^{-1}\, \mu_{-1}\, (g_{z})_{n} \end{bmatrix}\,
\begin{bmatrix} y_{n+1} \\ z_{n+1} \end{bmatrix} \\
& =
\begin{bmatrix} 
- \sum_{i=0}^{k-1} \alpha_i\, y_{n-i} + h\, \sum_{i=0}^{k-1} \beta_i\, f_{n-i} + h\, \sum_{i=0}^{k-1} \mu_i\, ((f_{y})_{n}\,y_{n-i} + (f_{z})_{n}\, z_{n-i}), \\
- \sum_{i=0}^{k-1} \alpha_i\, z_{n-i} + h\, \sum_{i=0}^{k-1} \beta_i\, \varepsilon^{-1}\,g_{n-i} + h\, \varepsilon^{-1}\, \sum_{i=0}^{k-1} \mu_i\, ((g_{y})_{n}\,y_{n-i} + (g_{z})_{n}\, z_{n-i})
\end{bmatrix}.
\end{split}
\end{equation}
We take the limit $\varepsilon \to 0$ to find the solution of the index-1 DAE.
\fi
Application to \eqref{eqn:LIMM_dae-1} gives:
\begin{subequations}
\label{eqn:LIMM_dae}
\begin{eqnarray}
\label{eqn:LIMM_dae-1_y}
\qquad 
 \sum_{i=-1}^{k-1} \alpha_i\, y_{n-i}  &=& h\, \sum_{i=0}^{k-1} \beta_i\, f_{n-i} + h\, \sum_{i=-1}^{k-1} \mu_i\, ((f_{y})_{n}\,y_{n-i} + (f_{z})_{n}\, z_{n-i}), \\
\label{eqn:LIMM_dae-1_z}
0 &=& \sum_{i=0}^{k-1} \beta_i\, g_{n-i} + \sum_{i=-1}^{k-1} \mu_i\, ((g_{y})_{n}\,y_{n-i} + (g_{z})_{n}\, z_{n-i}).
\end{eqnarray}
\end{subequations}
%
%
Substituting \eqref{eqn:LIMM_dae-1_z} into \eqref{eqn:LIMM_dae-1_y} gives the solution for $y$:
\begin{equation}
\label{eqn:LIMM_dae_y}
\begin{split}
\sum_{i=-1}^{k-1} \alpha_i\, y_{n-i} &= h\, \sum_{i=0}^{k-1} \beta_i\, (f_{n-i}- (f_{z}\,g_{z}^{-1})_n\,g_{n-i}) + h\,(f_{y}-f_{z}\,g_{z}^{-1}\,g_{y})_n \,\sum_{i=-1}^{k-1} \mu_i\, y_{n-i}.
\end{split}
\end{equation}
We make the assumption that one can fit smooth curves through $z_{n-i}$ and $y_{n-i}$ for $i=0,\dots,k$.  From \eqref{eqn:LIMM_dae-1_z} and the order conditions  
\eqref{eqn:LIMW-order-conditions-0-mu}--\eqref{eqn:LIMW-order-conditions-mu} we have $\sum_{i=0}^{k-1} \beta_i\, g_{n-i} \sim \mathcal{O}(h^p)$.
Equation \eqref{eqn:LIMM_dae_y} is the {\sc Limm} method applied to the reduced system, plus a perturbation $\mathcal{O}(h^{p+1})$. Consequently the slow variable has a local truncation error of order $p+1$. 
 
\begin{equation*}
\begin{split}
\delta z &= G(y)-z, \\
0=g(y,G(y)) &= g(y,z) + g_z(y,z)\,\delta z + \mathcal{O}(\|\delta z\|^2), \\
f(y,G(y)) &= f(y,z) + f_z(y,z)\,\delta z + \mathcal{O}(\|\delta z\|^2) = (f - f_z\,g_z^{-1}\,g)(y,z) + \mathcal{O}(\|\delta z\|^2).
\end{split}
\end{equation*}

From \eqref{eqn:LIMM_dae-1_z}:
\begin{equation}
\label{eqn:LIMM_dae_z}
\sum_{i=-1}^{k-1} (\beta_i + \mu_i)\, g_{n-i} =   \sum_{i=-1}^{k-1} \mu_i\, (  g_{n-i} -(g_{y})_{n}\,y_{n-i} - (g_{z})_{n}\, z_{n-i}).
\end{equation}
Let $\Delta y_{n} = y_{n}-y(t_{n})$ and $\Delta z_{n} = z_{n}-z(t_{n})$ be the global errors. Substituting the exact solutions in \eqref{eqn:LIMM_dae-1_z}, and subtracting from the numerical solution \eqref{eqn:LIMM_dae-1_z}, gives the following recurrence:
\begin{equation}
\label{eqn:LIMM_dae_deltaz}
\sum_{i=-1}^{k-1} (\beta_i + \mu_i)\, g_{n-i} =   \sum_{i=-1}^{k-1} \mu_i\, (  g_{n-i} - (g_{y})_{n}\,\Delta y_{n-i} - (g_{z})_{n}\, \Delta z_{n-i}) +\mathcal{O}(h^{p+1}). 
\end{equation}
Using $\Delta_{n-i} = \max( \Vert \Delta y_{n-i} \Vert, \Vert \Delta z_{n-i} \Vert )$ we have the relation:
\begin{equation*}
\begin{split}
0 &= g(z(t_{n-i}), y(t_{n-i})) = g(z_{n-i} - \Delta z_{n-i} , y_{n-i}- \Delta y_{n-i}) \\
&= g_{n-i} - (g_z)_{n-i}\, \Delta z_{n-i} - (g_y)_{n-i}\, \Delta y_{n-i} + \mathcal{O}(\Delta_{n-i}^2)  \\
&= g_{n-i} - \Bigl((g_z)_{n}+\mathcal{O}(h)\Bigr)\,\Delta z_{n-i} - \Bigl((g_y)_{n}+\mathcal{O}(h)\Bigr)\,\Delta y_{n-i} + \mathcal{O}(\Delta_{n-i}^2). \\
g_{n-i} &= g(z(t_{n-i})+ \Delta z_{n-i} , y(t_{n-i})+ \Delta y_{n-i})  \\
&= g_z(t_{n})\,(z(t_{n-i})-z(t_{n})+\Delta z_{n-i}) +  g_y(t_{n})\,(y(t_{n-i})-y(t_{n})+\Delta y_{n-i}) + \mathcal{O}(h\, \Delta_{n-i}) 
\end{split}
\end{equation*}
the global error recurrences \eqref{eqn:LIMM_dae_deltaz} become:
\begin{equation}
\label{eqn:LIMM_global_err}
\begin{split}
& \sum_{i=-1}^{k-1} \mu_i\, \Delta z_{n-i} = - (g_{z}^{-1}\,g_{y})_{n} \,\sum_{i=-1}^{k-1} \mu_i\, \Delta y_{n-i} - (g_{z}^{-1})_n\, \sum_{i=0}^{k-1} \beta_i\, g_{n-i} + \mathcal{O}(h^{p+1})  \\
& \sum_{i=-1}^{k-1} \Bigl( \mu_i + \beta_i + \mathcal{O}(h) \Bigr)\, \Bigl( \Delta z_{n-i} + (g_{z}^{-1}\,g_{y})_{n-1} \,\Delta y_{n-i}\Bigr) =  \mathcal{O}(h^{p+1}).  
\end{split}
\end{equation}


%
Since $\mu_i + \beta_i$ is a stable recurrence with roots strictly smaller than one, its perturbation remains stable for small enough step sizes. Consequently $\Delta z_{n-i} + (g_{z}^{-1}\,g_{y})_{n} \,\Delta y_{n-i} \sim \mathcal{O}(h^p)$ for all $n$. Therefore if $\Delta y_{n} \sim \mathcal{O}(h^p)$ then $\Delta y_{n} \sim \mathcal{O}(h^p)$ for all $n$.

\subsection{The GLM-ROS route}

In matrix form, a GLM-ROS method is represented as
\begin{equation}
\begin{split}
K &= h \, F \left( \mathbf{A} \otimes K  + \mathbf{U} \otimes \zeta^{[n-1]} \right) + h \, \J_n \, \boldsymbol{\Gamma} \otimes K + h \, \J_n \, \boldsymbol{\Psi} \otimes \zeta^{[n-1]}, \\
\zeta^{[n]} &= \mathbf{B} \otimes K + \mathbf{V} \otimes \zeta^{[n-1]}. 
\end{split}
\end{equation}

We write the LIMM method \eqref{eqn:LIMM_definition} as
\begin{equation}
\begin{split}
K_1 &\coloneqq y_{n+1}, \quad K_2 \coloneqq h\,f(t_{n+1},y_{n+1}), \\
K_1 &= - \sum_{i=0}^{k-1} \alpha_i\, y_{n-i} + h\, \sum_{i=0}^{k-1} \beta_i\, f_{n-i} + h\, \J_n\, \gamma_{-1}\, K_1 + h\, \J_n\, \sum_{i=0}^{k-1} \mu_i\, y_{n-i}, \\
 &= \sum_{i=1}^{k} (-\alpha_{i-1})\, \zeta^{[n-1]}_{i} + h\, \sum_{i=1}^{k} \beta_{i-1}\, \zeta^{[n-1]}_{i+k} + h\, \gamma_{-1}\, \J_n\, K_1
 + h\, \J_n\, \sum_{i=1}^{k} \mu_{i-1}\, \zeta^{[n-1]}_{i}, \\
K_2 &= h\,f\left( K_1\right), \\
\zeta^{[n]}_{1} &= y_{n+1} = K_1, \\
\zeta^{[n]}_{i} &= y_{n-i+2} = \zeta^{[n-1]}_{i-1}, \quad 2 \le i \le k \\
\zeta^{[n]}_{k+1} &= f_{n+1} = K_2 \\
\zeta^{[n]}_{k+i} &= f_{n-i+2} = \zeta^{[n-1]}_{k+i-1}, \quad 2 \le i \le k \\
\end{split}
\end{equation}
The Butcher tableau of a LIMM written as a GLM reads:
\[
\mathbf{A} = \begin{bmatrix} 0 & 0 \\ 1 & 0 \end{bmatrix}, \quad \boldsymbol{\Gamma} = \begin{bmatrix} 1 & 0 \\ 0 & 0 \end{bmatrix}, \quad
\widehat{\mathbf{A}} = \begin{bmatrix} 1 & 0 \\ 1 & 0 \end{bmatrix},
\]
\[
\boldsymbol{\Psi} = \begin{bmatrix} \boldsymbol{0}_{1 \times k} & \boldsymbol{0}_{1 \times k} \\ \boldsymbol{\mu}_{0:k-1}^T & \boldsymbol{\beta}_{0:k-1}^T \end{bmatrix},
\quad \mathbf{U} = 0,
\quad \mathbf{B} = [e_1,e_{k+1}],
\quad \mathbf{V} = \begin{bmatrix} \mathbf{S}_{k \times k} & \boldsymbol{0} \\  \boldsymbol{0}  & \mathbf{S}_{k \times k} \end{bmatrix}.
\]

\subsection{The IMEX route}
We write the {\sc Limm-w} method \eqref{eqn:LIMM_definition} as an IMEX LMM scheme applied to a linear-nonlinear partitioning of the system:
\begin{equation}
\label{eqn:LIMM=IMEX_idx1}
\begin{split}
\sum_{i=-1}^{k-1} \alpha_i\, y_{n-i} & =  h\, \sum_{i=0}^{k-1} \beta_i\, (f_{n-i} - \J_n\,y_{n-i}) + h\, \sum_{i=-1}^{k-1} (\mu_i + \beta_i)\, \J_n\,y_{n-i}.
\end{split}
\end{equation}
The IMEX LMM order conditions are equivalent to the {\sc limm-w} order conditions:
\begin{align}
 \sum_{i=-1}^{k-1} \alpha_i &= 0,  \\
\sum_{i=-1}^{k-1}  \alpha_i\,c_i^\ell + \ell\, \sum_{i=0}^{k-1} \beta_i\, c_i^{\ell-1} &= 0,  \qquad \ell = 1,\dots,p, \\
\sum_{i=-1}^{k-1}  \alpha_i\,c_i^\ell + \ell\, \sum_{i=-1}^{k-1} (\mu_i + \beta_i)\, c_i^{\ell-1} &= 0,  \qquad \ell = 1,\dots,p.
\end{align}

\begin{equation}
\begin{split}
\sum_{i=-1}^{k-1} \alpha_i\, y_{n-i} &= h\, \sum_{i=0}^{k-1} \beta_i\, (f_{n-i}-(f_{z}\,g_{z}^{-1})_{n}\,g_{n-i}-\J^R_n \, y_{n-i}) \\
&\quad + h\,\J^R_n \,\sum_{i=-1}^{k-1} (\beta_i+\mu_i)\, y_{n-i} \\
0 &= \sum_{i=0}^{k-1} \beta_i\, g_{n-i} + \sum_{i=-1}^{k-1} \mu_i\, ((g_{y})_{n}\,y_{n-i} + (g_{z})_{n}\, z_{n-i}).
\end{split}
\end{equation}

\fi

\section{Construction of Optimized {\sc Limm} Schemes}
\label{sec:methodderiv}

In order to construct practical {\sc Limm} schemes we seek to satisfy two requirements: have a small local truncation error and a large numerical stability region \eqref{eqn:stability-region}. For multistep methods, these two requirements are frequently at odds with each other; thus, it is vital that they are considered together when designing a method. To quantify the local truncation error of a method of order $p \ge 1$ we consider the (scaled) residuals of the $(p+1)$-st order conditions \eqref{eqn:LIMW-order-conditions-traditional} and \eqref{eqn:LIMW-order-conditions-mu}:
\begin{subequations}
\label{eqn:lerr_res}
\begin{align}
\label{eqn:lerr_res_a}
\rho_{p+1,a} &= \sum_{i=-1}^{k-1}  \alpha_i\,c_i^{p+1} + (p+1)\, \sum_{i=0}^{k-1} \beta_i\, c_i^{p}, \\
\label{eqn:lerr_res_b}
\rho_{p+1,b} &= (p+1)\,\sum_{i=-1}^{k-1} \mu_i\, c_i^{p}  - (p+1)\,p\, \sum_{i=0}^{k-1} \nu_i\, c_i^{p-1}. 
\end{align}
\end{subequations}
If $p=1$ and exact Jacobians are used then one only needs to consider the sum of the two residuals. 
To quantify stability observe that the {\sc Limm} method is $A(\phi)$ stable with a stability angle:
\begin{equation}
 \label{eqn:alphastab}
 \phi(\bm{\alpha},\bm{\beta},\bm{\mu},\bm{\nu}) = \argmin_{0 \leq\theta\leq 2\pi} \left|\arg\left(-z\left(e^{i\theta}, \bm{\alpha},\bm{\beta},\bm{\mu},\bm{\nu}\right)\right)\right|.
\end{equation}

We design practical linearly implicit multistep methods using a multiobjective genetic optimization algorithm to simultaneously maximize the $A(\phi)$ stability angle \eqref{eqn:alphastab}, and minimize the local error of the method.
Maximizing the stability angle $\phi$ is equivalent to minimizing the first objective function:
\begin{equation}
 \Phi_1\left(\bm{\alpha},\bm{\beta},\bm{\mu},\bm{\nu}\right) = \left( 1 + \phi(\bm{\alpha},\bm{\beta},\bm{\mu},\bm{\nu}) \right)^{-1}.
 \label{eqn:obj1}
\end{equation}
To maximize accuracy we minimize the second objective function:
\begin{equation}
 \Phi_2\left(\bm{\alpha},\bm{\beta},\bm{\mu},\bm{\nu}\right) = \rho_{p+1,a}^2 + \rho_{p+1,b}^2.
 \label{eqn:obj2}
\end{equation}
In order to simplify the search space, we choose to set $\nu_i = 0, \, i = 0, ..., k-1$. Then, we make use of the method order conditions to write the coefficients $\alpha_i$, $\beta_i$, and $\mu_i$ in terms of the $c_i$'s and a smaller subset of free parameters. Then, to get fixed stepsize coefficients, we substitute $c_i = i$ in the coefficient expressions and in the residuals $\rho_{p+1,a}$ and $\rho_{p+1,b}$ (we can later retrieve the variable stepsize coefficients by leaving the $c_i$'s and substituting only for the free parameters). We also make one additional simplification: by setting $\sigma(0) = 0$, we guarantee that as $z \rightarrow -\infty$, the stability function $\zeta(z) \rightarrow 0$ (which is a condition for L-stability \cite[Section IV.3]{Hairer_book_II}). These simplifications are all directly embedded in the computation of both \eqref{eqn:obj1} and  \eqref{eqn:obj2}.

Finally, the resulting method must also be zero-stable, so we apply a nonlinear inequality constraint to enforce that roots of $\varrho(\zeta)$, the first characteristic polynomial of the method, fall on or inside the unit circle depending on multiplicity.

Making use of Matlab's optimization toolbox, we use a genetic algorithm to produce a population of good candidate methods for each of orders 2-5. From the set of candidates we select a method of each order that has the appropriate balance between stability angle and error coefficient. To acquire exact coefficients, we rationalize the optimal free parameters and substitute into the original coefficient expressions. Fixed stepsize coefficients for the selected {\sc Limm-w} methods are presented in Table \ref{tbl:limmwfcoefs}, with error coefficients and stability angles compared with BDF in Table \ref{tbl:limmfstab}. Figure \ref{fig:limmwfstabreg} plots the stability regions of the selected {\sc Limm-w} methods in the complex plane.
\iflong
Coefficients for the fixed stepsize {\sc Limm} methods can be found in Table \ref{tbl:limmfcoefs}, and the corresponding stability regions in Figure \ref{fig:limmwfstabreg}.
\fi
\ifnosup
The full variable stepsize coefficient expressions can be found in Appendices \ref{sec:limmvarcoeff} and \ref{sec:limmwvarcoeff}.
\else
The full variable stepsize coefficient expressions can be found in the supplementary material in Sections \ref{sec:limmvarcoeff} and \ref{sec:limmwvarcoeff}.
\fi

\iflong
\begin{figure}[h]
\centering
\subfloat[Order one method.]{
\includegraphics[width=0.3\textwidth]{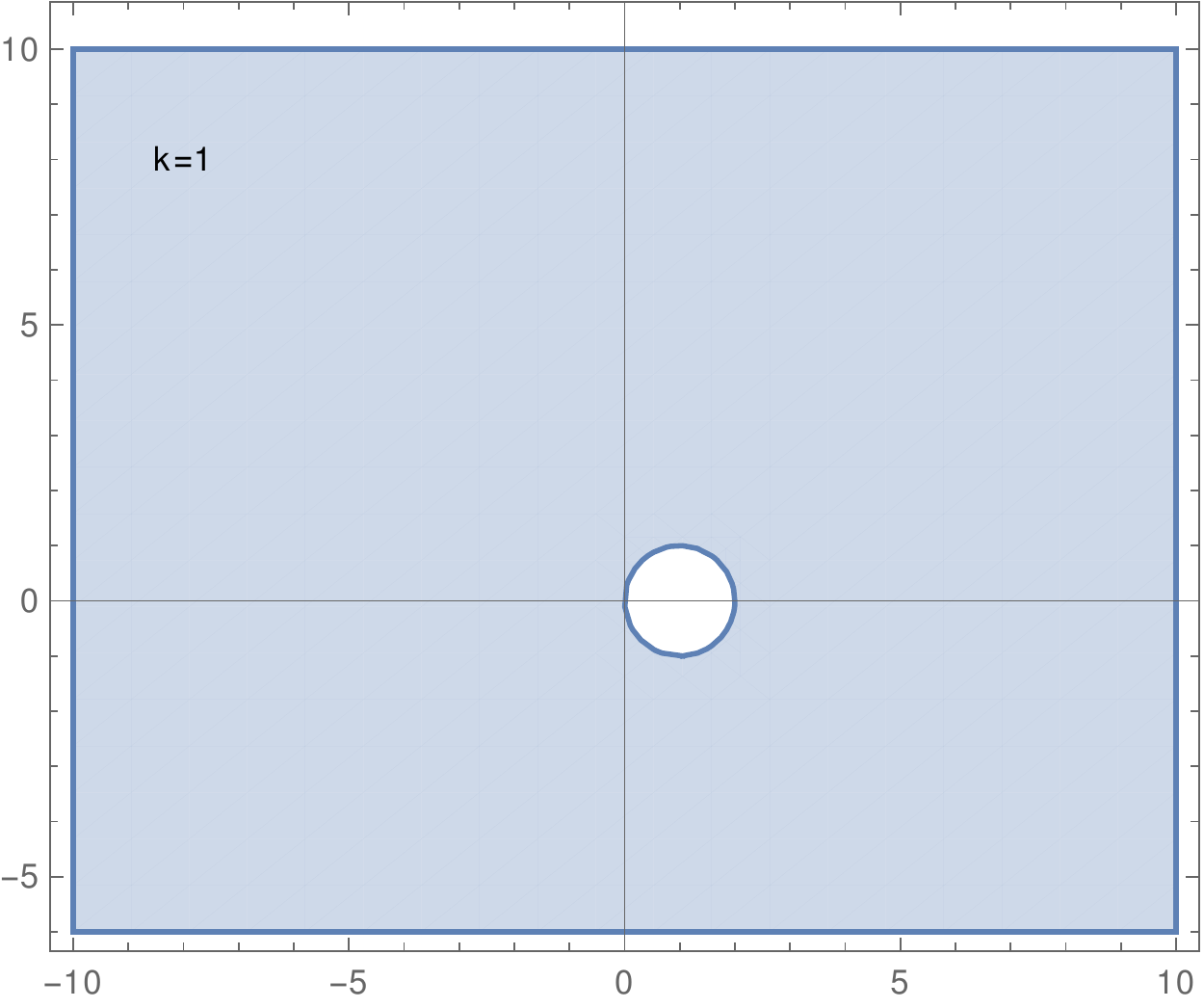}
\label{fig:limm1fstabreg}
}
\subfloat[Order two method.]{
\includegraphics[width=0.3\textwidth]{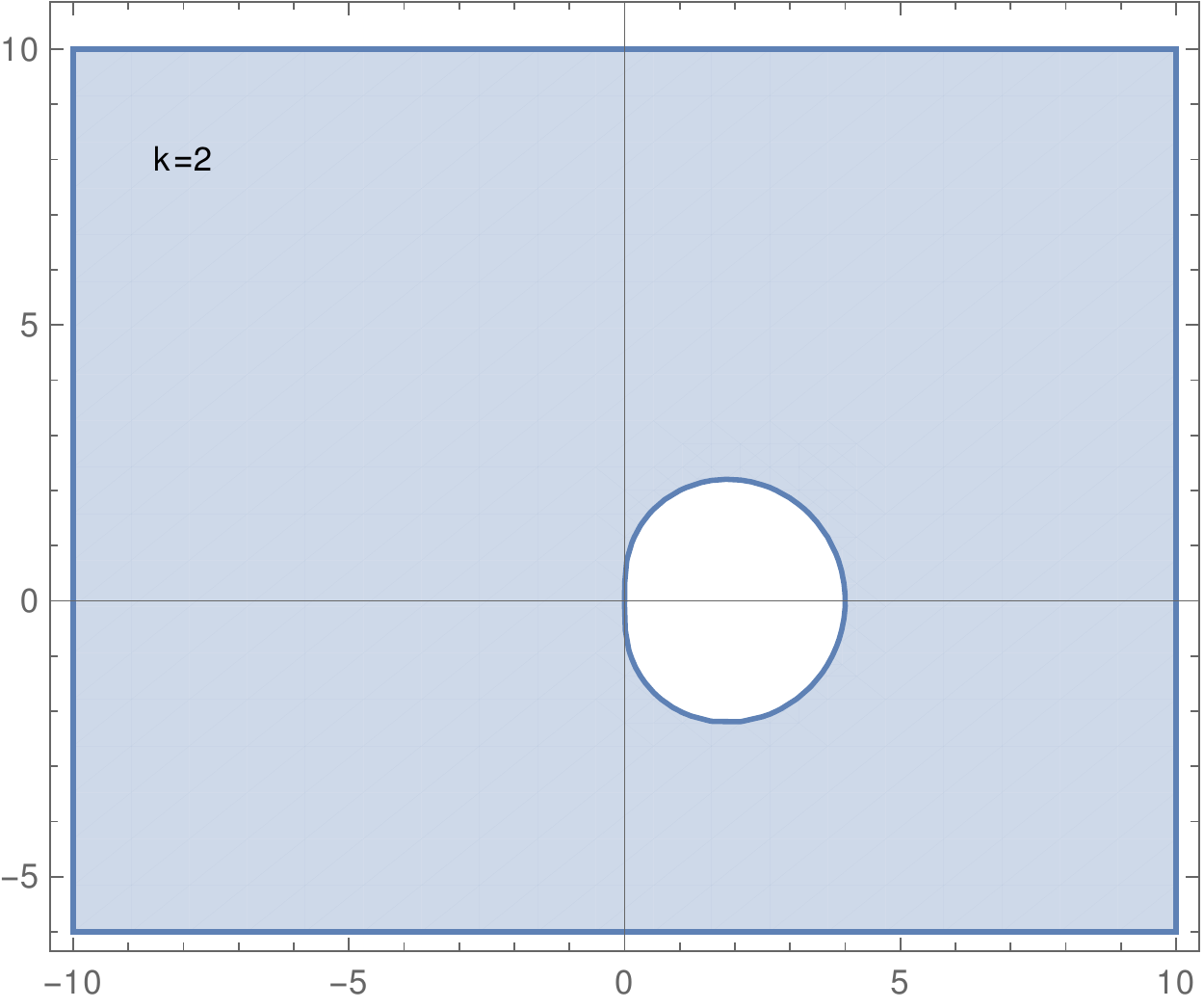}
\label{fig:limm2fstabreg}
}
\subfloat[Order three method.]{
\includegraphics[width=0.3\textwidth]{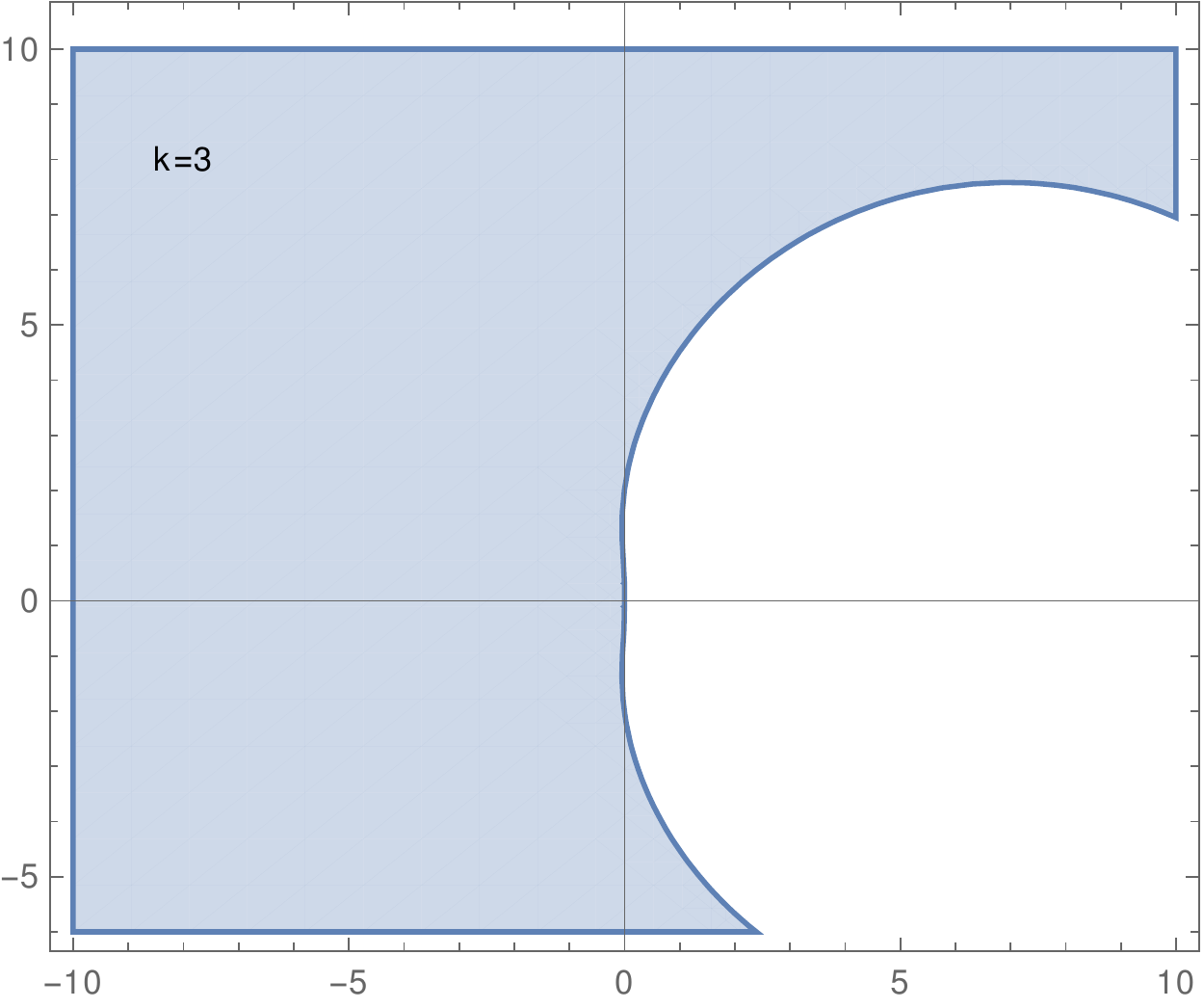}
\label{fig:limm3fstabreg}
}

\subfloat[Order four method.]{
\includegraphics[width=0.3\textwidth]{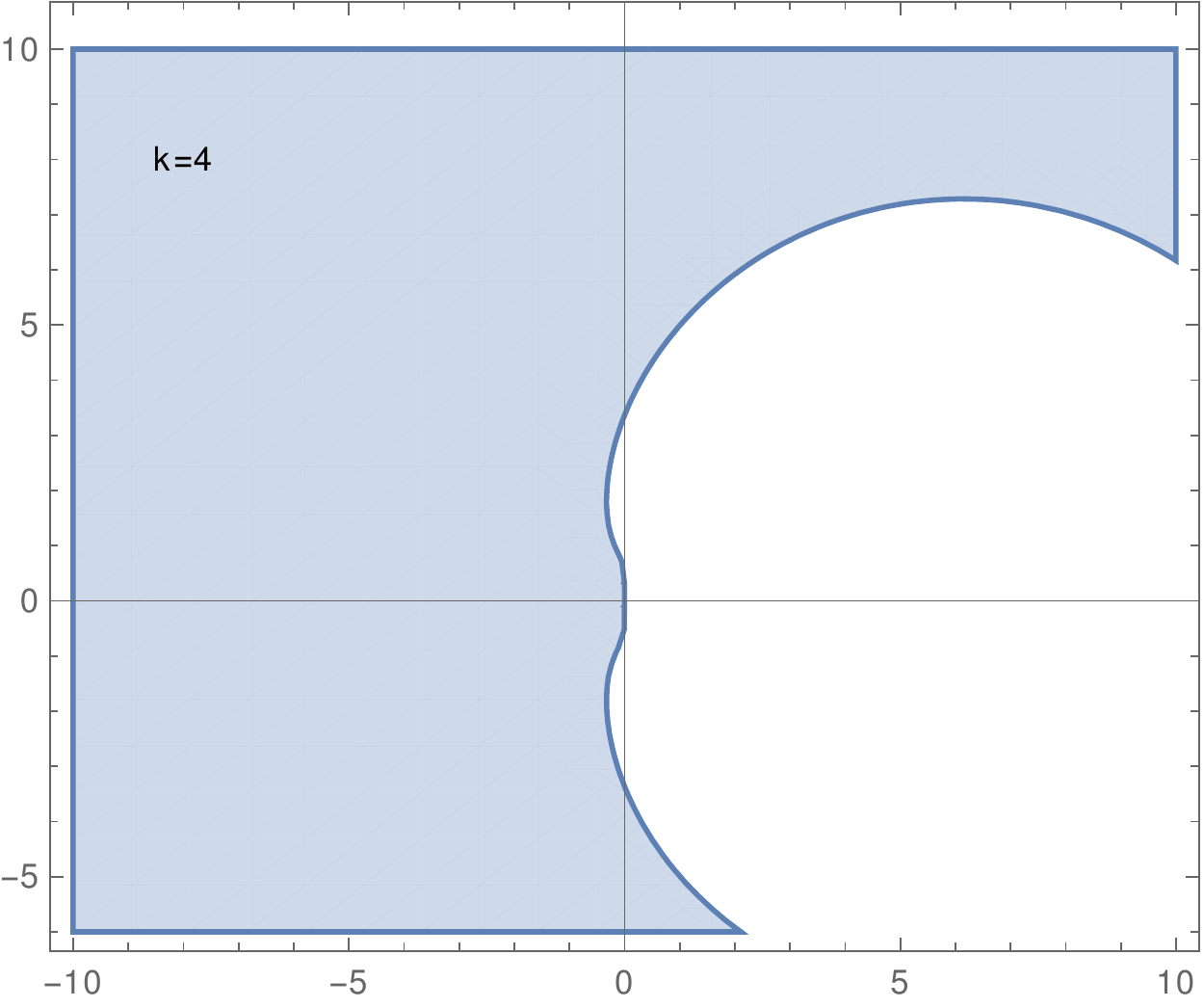}
\label{fig:limm4fstabreg}
}
\subfloat[Order five method.]{
\includegraphics[width=0.3\textwidth]{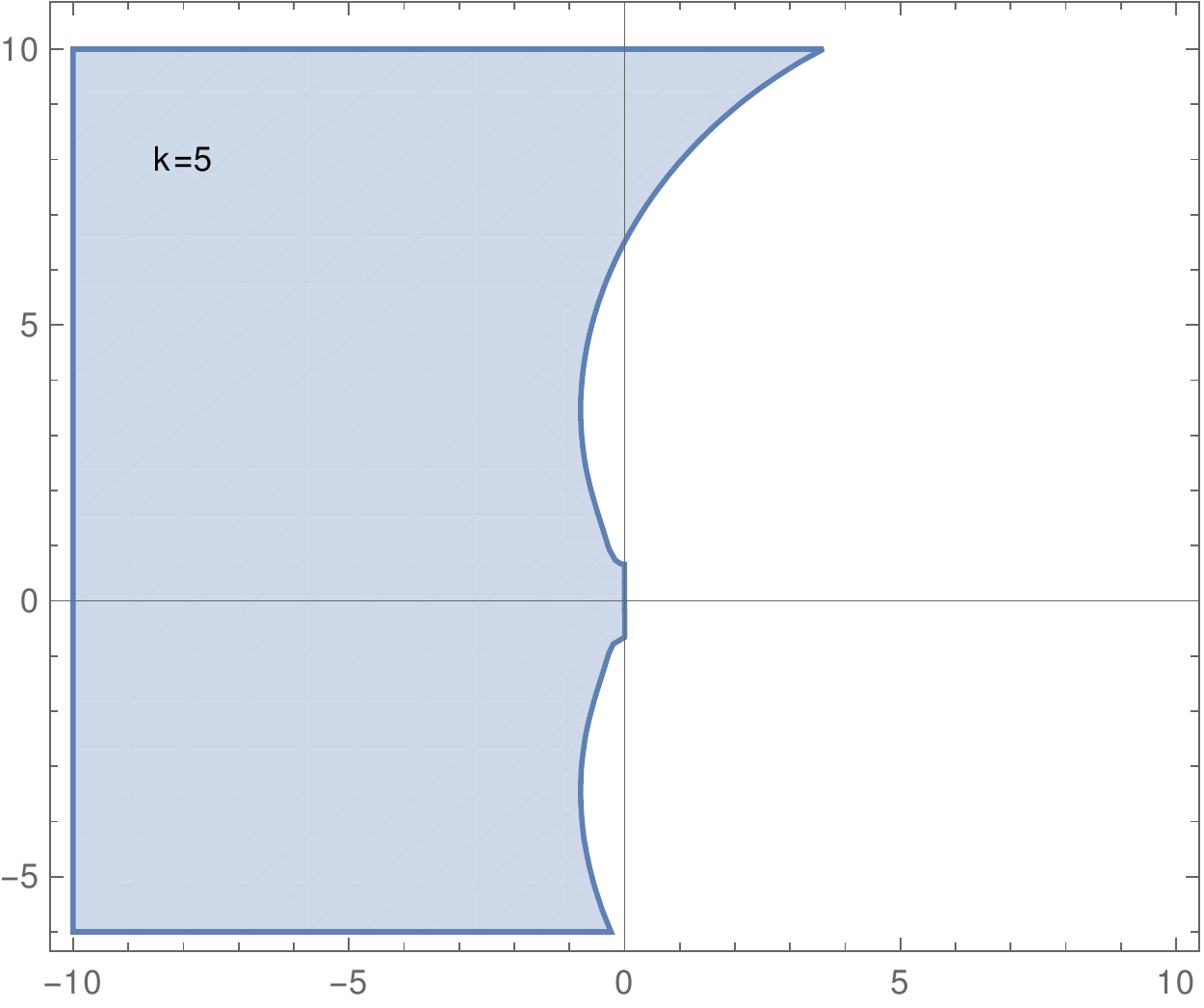}
\label{fig:limm5fstabreg}
}
\caption{Stability regions for fixed stepsize {\sc Limm} methods of orders one through five. Orders one and two are $A$-stable, and orders three to five are $A(\phi)$-stable with $\phi$'s listed in Table \ref{tbl:limmfstab}.}
\label{fig:limmfstabreg}
\end{figure}
\fi

\begin{figure}[h]
\centering
\subfloat[Order one method.]{
\includegraphics[width=0.3\textwidth]{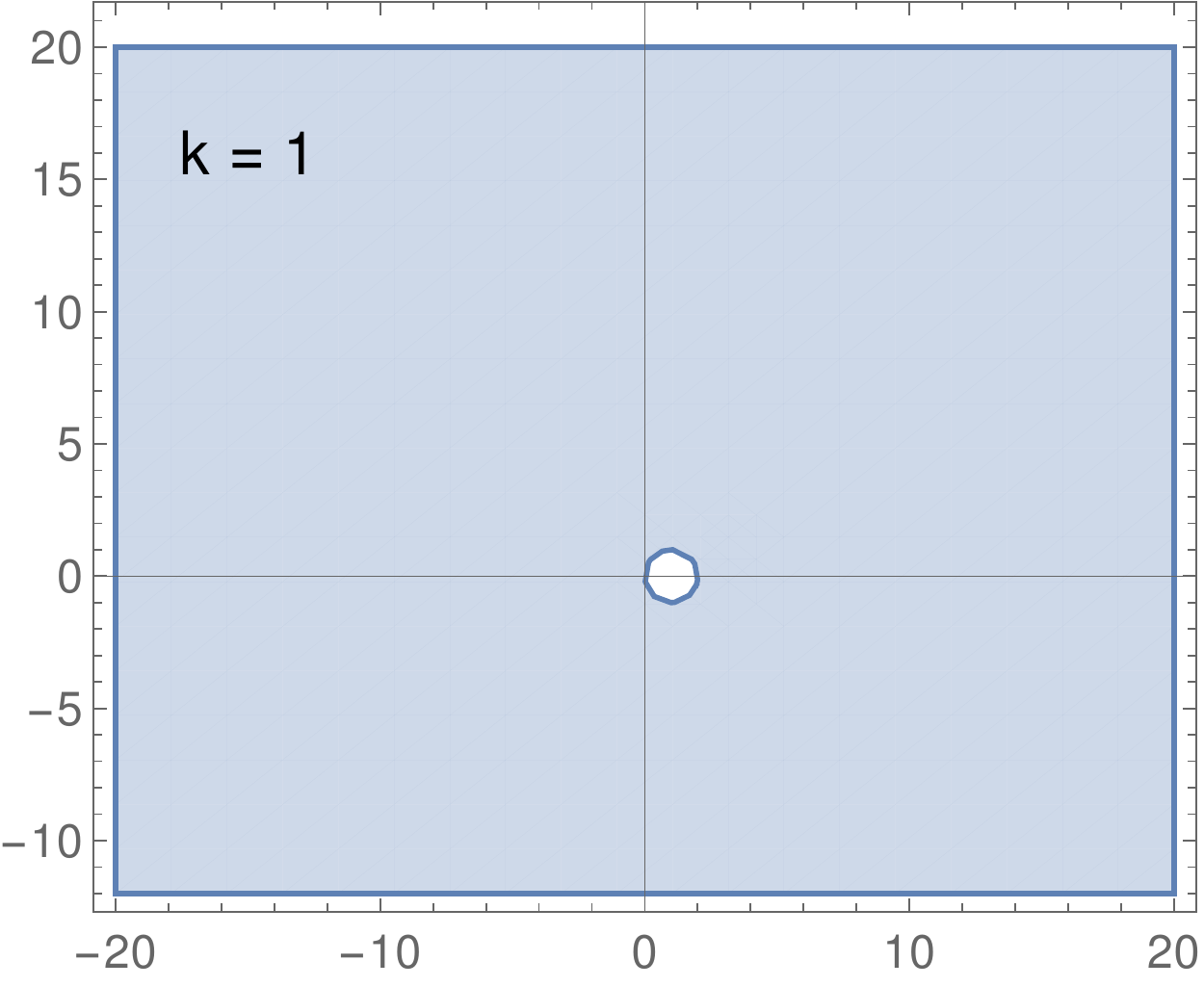}
\label{fig:limmw1fstabreg}
}
\subfloat[Order two method.]{
\includegraphics[width=0.3\textwidth]{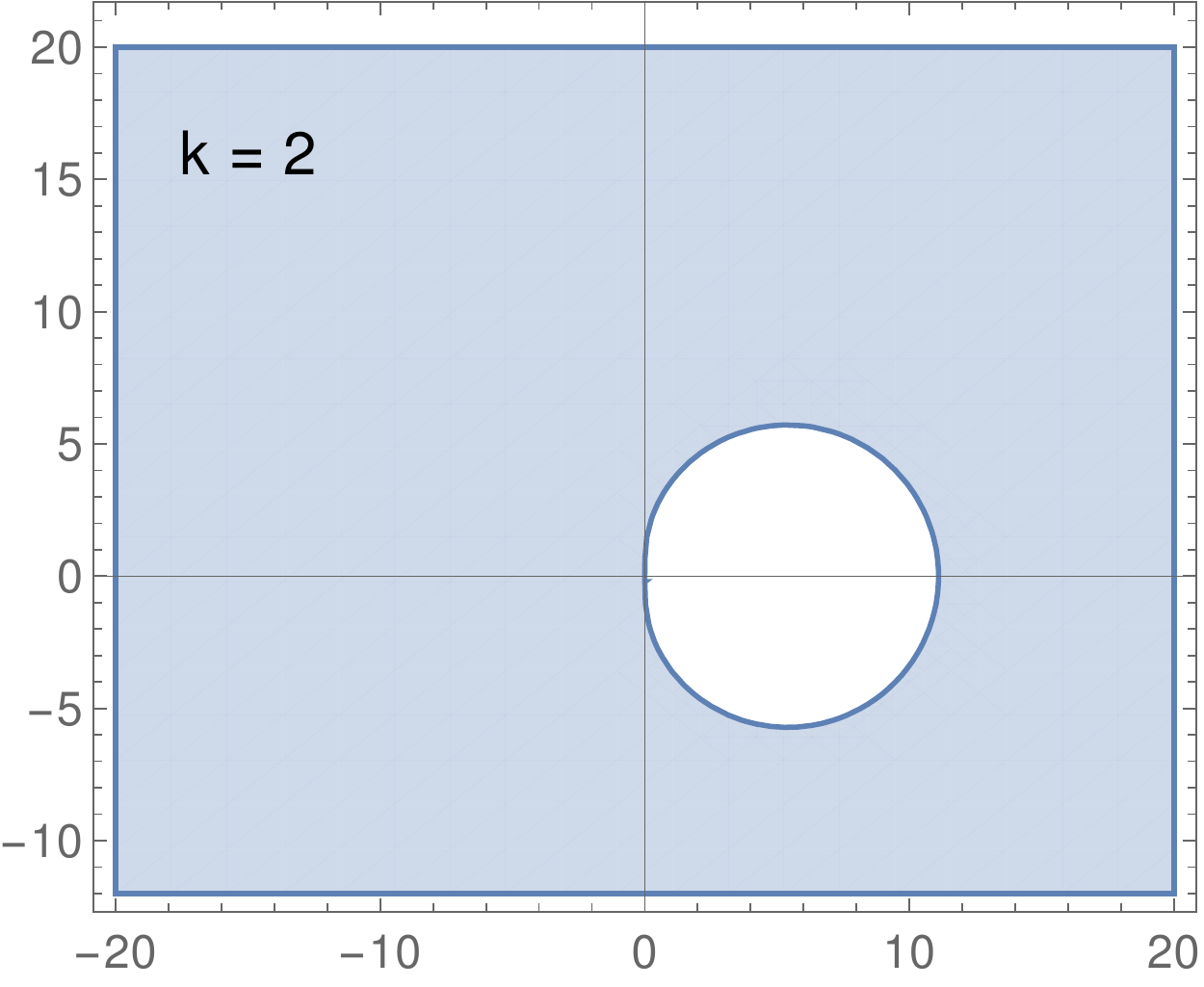}
\label{fig:limmw2fstabreg}
}
\subfloat[Order three method.]{
\includegraphics[width=0.3\textwidth]{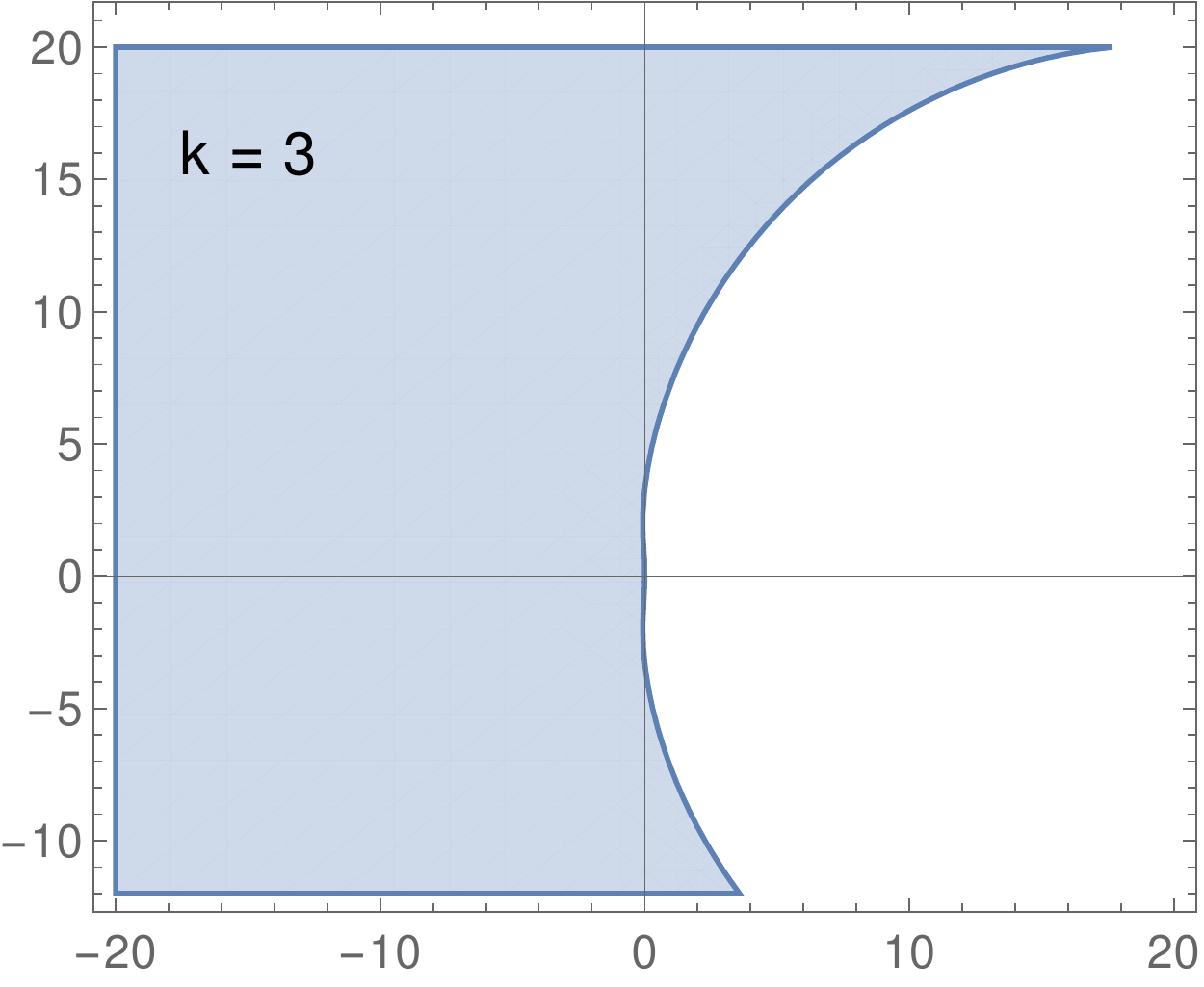}
\label{fig:limmw3fstabreg}
}

\subfloat[Order four method.]{
\includegraphics[width=0.3\textwidth]{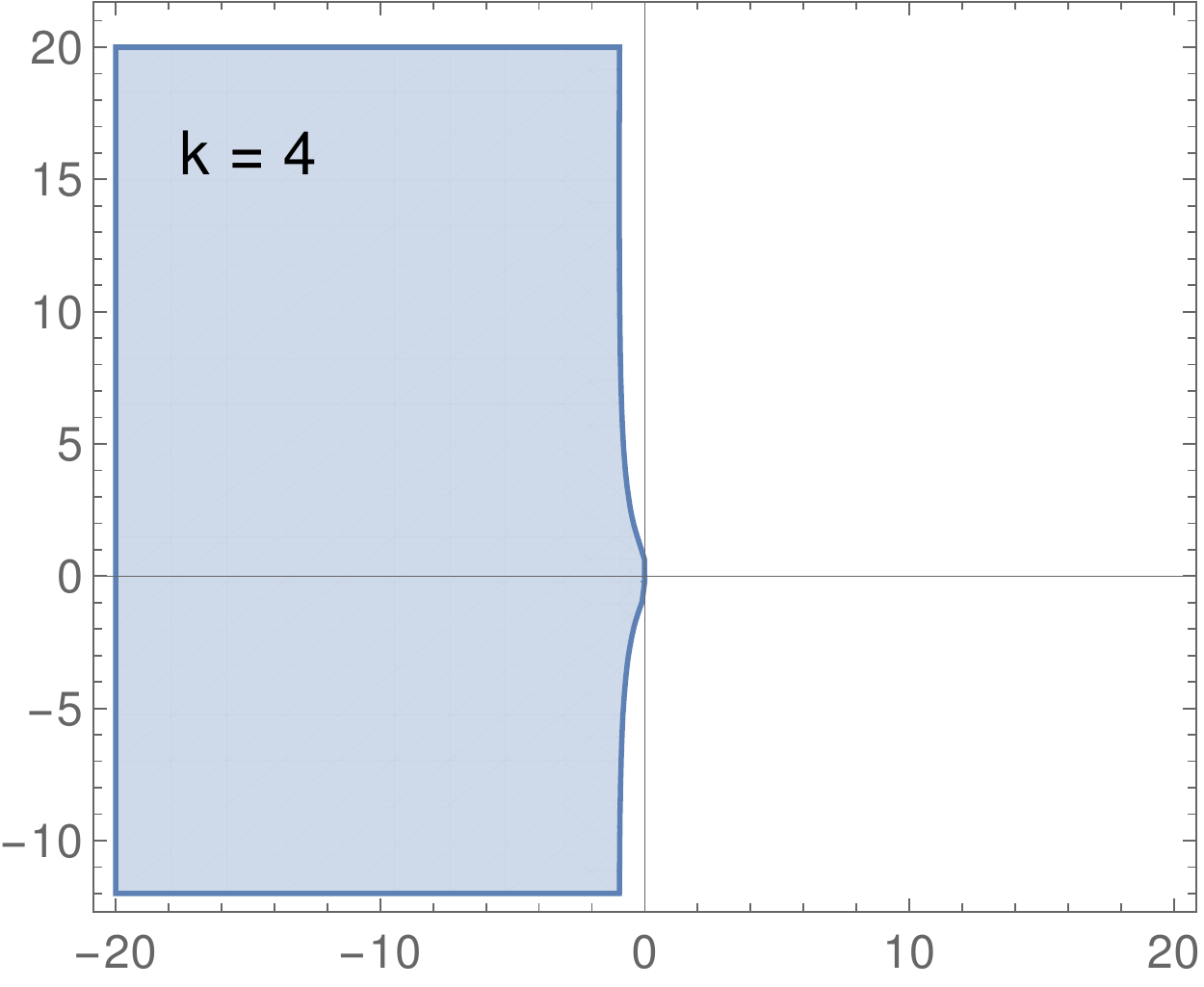}
\label{fig:limmw4fstabreg}
}
\subfloat[Order five method.]{
\includegraphics[width=0.3\textwidth]{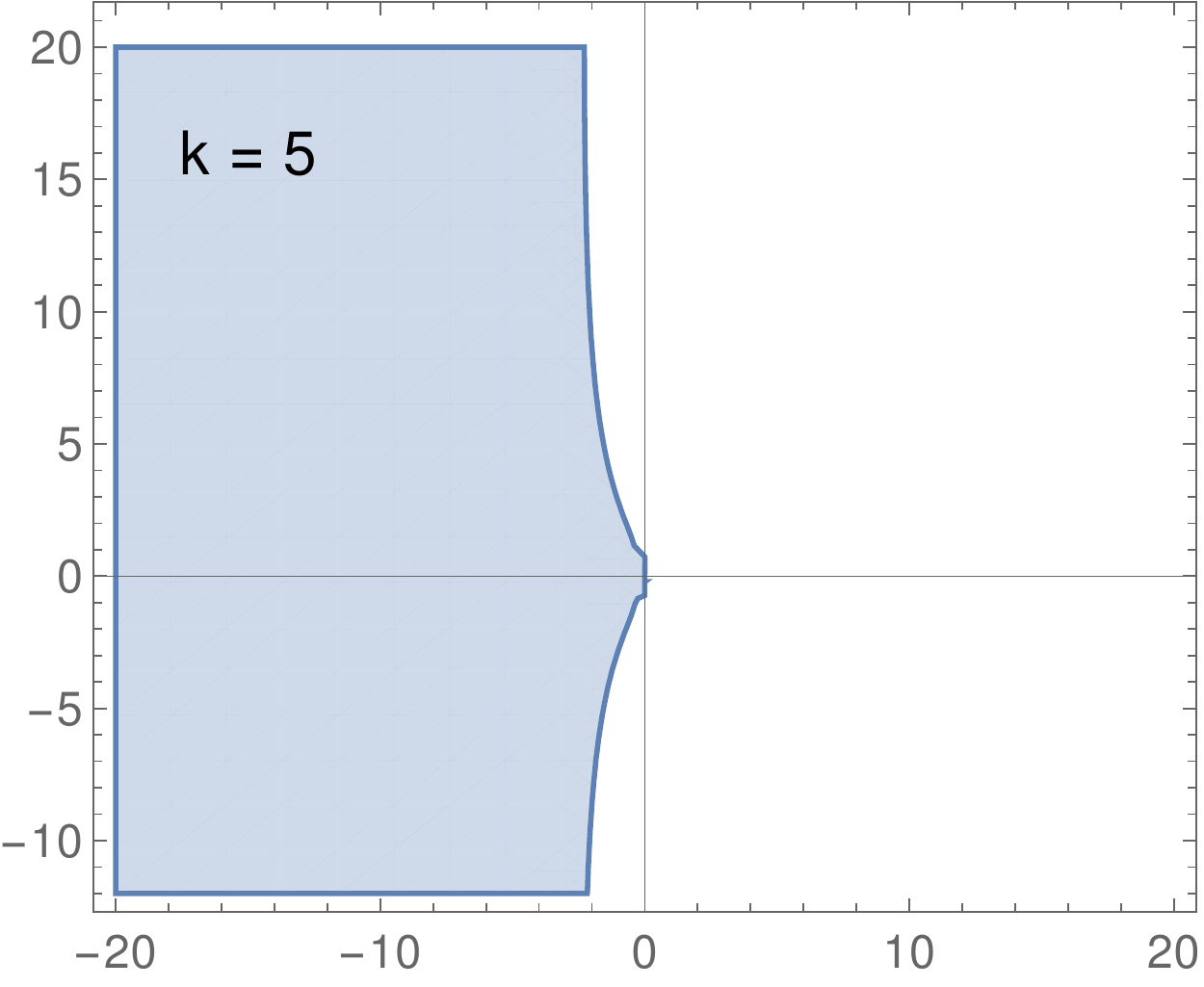}
\label{fig:limmw5fstabreg}
}
\caption{Stability regions for fixed stepsize {\sc Limm-w} methods of orders one through five. Orders one and two are $A$-stable, and orders three to five are $A(\phi)$-stable with $\phi$'s listed in Table \ref{tbl:limmfstab}.}
\label{fig:limmwfstabreg}
\end{figure}

\begin{table}[tbhp]
{\footnotesize
\caption{Characteristics of $k$-step fixed stepsize {\sc Limm}, {\sc Limm-w}, BDF, and explicit/implicit Adams methods for $k = 1,\dots,5$. Implicit Adams methods are of order $k+1$; all others are of order $k$. For {\sc Limm} and {\sc Limm-w}, stability angles are computed from \eqref{eqn:alphastab}, and error constants are computed as in \eqref{eqn:localerr-constant}.}
\begin{center}
$
\renewcommand{\arraystretch}{1.5}
\begin{array}{c|c|ccccc}
 & $k$ & 1 & 2 & 3 & 4 & 5 \\
\hline
\multirow{2}{*}{\text{explicit Adams}} & \text{A($\phi$)-stability angle} & \text{n/a} & \text{n/a} & \text{n/a} & \text{n/a} & \text{n/a} \\
 & \text{Error constant} & 0.5 & 0.416667 & 0.375 & 0.348611 & 0.329861 \\
\hline
\multirow{2}{*}{\text{implicit Adams}} & \text{A($\phi$)-stability angle} & 90. & \text{n/a} & \text{n/a} & \text{n/a} & \text{n/a} \\
 & \text{Error constant} & 0.083333 & 0.041666 & 0.026389 & 0.01875 & 0.014269 \\
 \hline
\multirow{2}{*}{\text{BDF}} & \text{A($\phi$)-stability angle} & 90. & 90. & 86.03 & 73.35 & 51.84 \\
 & \text{Error constant} & 0.5 & 0.333333 & 0.25 & 0.2 & 0.166667 \\
\hline
\multirow{2}{*}{\text{\sc Limm}} & \text{A($\phi$)-stability angle} & 90. & 90. & 87.7849 & 78.0742 & 72.9999 \\
 & \text{Error constant} & 0.5 & 0.222222 & 0.167344 & 0.204625 & 0.217405 \\
\hline
\multirow{2}{*}{\text{\sc Limm-w}} & \text{A($\phi$)-stability angle} & 90. & 90. & 87.3899 & 77.9101 & 70.3168 \\
 & \text{Error constant} & 0.5 & 0.424915 & 0.403238 & 0.380873 & 0.365325
\end{array}

$
\label{tbl:limmfstab}
\end{center}
}
\end{table}

\iflong
\begin{table}[h]
\centering
{\footnotesize
\caption{Exact coefficients for $k$-step fixed stepsize {\sc Limm} methods of order $k$, for $k=1 \dots 5$.}
\subfloat[{\sc Limm} 1-step order 1 coefficients.]{
\begin{minipage}{0.45\textwidth}
\centering
$
\renewcommand{\arraystretch}{1.5}
\begin{array}{c|cc}
 i & -1 & 0 \\
\hline
 \alpha_i & 1 & -1 \\
 \beta_i & 0 & 1 \\
 \mu_i & 1 & -1 \\
\end{array}
$
\end{minipage}
\label{tbl:limfcoef1}
}
\subfloat[{\sc Limm} 2-step order two coefficients.]{
\begin{minipage}{0.45\textwidth}
\centering
$
\renewcommand{\arraystretch}{1.5}
\begin{array}{c|ccc}
 i & -1 & 0 & 1 \\
\hline
 \alpha_i & 1 & -\frac{4}{3} & \frac{1}{3} \\
 \beta_i & 0 & \frac{2}{3} & 0 \\
 \mu_i & \frac{2}{3} & -\frac{2}{3} & 0 \\
\end{array}
$
\end{minipage}
\label{tbl:limfcoef2}
}

\subfloat[{\sc Limm} 3-step order three coefficients.]{
$
\renewcommand{\arraystretch}{1.5}
\begin{array}{c|cccc}
 i & -1 & 0 & 1 & 2 \\
\hline
 \alpha_i & 1 & -\frac{67569925}{40220258} & \frac{77233903}{99562899} & -\frac{383355371802341}{4004445485007942} \\
 \beta_i & 0 & \frac{6}{11} & -\frac{56091046951621340}{198220051507893129} & \frac{30378060674886581}{198220051507893129} \\
 \mu_i & \frac{3082752052157006}{6006668227511913} & -\frac{30378060674886581}{66073350502631043} & \frac{19781424978365126}{198220051507893129} & -\frac{30378060674886581}{198220051507893129} \\
\end{array}
$
\label{tbl:limfcoef3}
}

\subfloat[{\sc Limm} 4-step order 4 coefficients.]{
$
\renewcommand{\arraystretch}{1.5}
\begin{array}{c|cc}
 i & -1 & 0 \\
\hline
 \alpha_i & 1 & -\frac{60010656}{28439311} \\
 \beta_i & 0 & \frac{12}{25} \\
 \mu_i & \frac{6044411368232668137128215}{12447162528941684828131512} & -\frac{60023632933941523627586873}{103726354407847373567762600} \\
\hline
 i & 1 & 2 \\
\hline
 \alpha_i & \frac{71006953}{40099309} & -\frac{345107661}{454781887} \\
 \beta_i & -\frac{829829410576978812863115039}{1140989898486321109245388600} & \frac{133675753843217938307088979}{142623737310790138655673575} \\
 \mu_i & \frac{194551206099828504610038241}{285247474621580277311347150} & -\frac{2829520362862954765370488571}{3422969695458963327736165800} \\
\hline
 i & 3 & \text{} \\
\hline
 \alpha_i & \frac{50927106883029008210353}{518631772039236867838813} & \text{} \\
 \beta_i & -\frac{271157550073699750683379121}{1140989898486321109245388600} & \text{} \\
 \mu_i & \frac{271157550073699750683379121}{1140989898486321109245388600} & \text{} \\
\end{array}
$
\label{tbl:limfcoef4}
}

\subfloat[{\sc Limm-w} 5-step order 5 coefficients.]{
$
\renewcommand{\arraystretch}{1.5}
\begin{array}{c|cc}
 i & -1 & 0 \\
\hline
 \alpha_i & 1 & -\frac{104367911}{41202283} \\
 \beta_i & 0 & \frac{60}{137} \\
 \mu_i & \frac{322638273004961021870227746746423}{712722768713639590268860359964200} & -\frac{31175917409117421775097382197076197}{48821509656884311933416934657547700} \\
\hline
 i & 1 & 2 \\
\hline
 \alpha_i & \frac{59680231}{21017185} & -\frac{97736124}{57440479} \\
 \beta_i & -\frac{1740570722762351776400683674709186511}{1220537741422107798335423366438692500} & \frac{487813399545245689582675417708028617}{203422956903684633055903894406448750} \\
 \mu_i & \frac{1717451252646034545185780351980957211}{1220537741422107798335423366438692500} & -\frac{2669383545787015283771247804743841377}{1220537741422107798335423366438692500} \\
\hline
 i & 3 & 4 \\
\hline
 \alpha_i & \frac{19515650}{39801941} & -\frac{188732392210474496577705869057}{1979785468648998861857945444345} \\
 \beta_i & -\frac{25562879042079908014978668038159641}{21412942831966803479568830990152500} & \frac{157267484617875282653199076556264173}{610268870711053899167711683219346250} \\
 \mu_i & \frac{426670615738191742376152898428305157}{348725068977745085238692390411055000} & -\frac{157267484617875282653199076556264173}{610268870711053899167711683219346250} \\
\end{array}
$
\label{tbl:limfcoef5}
}
\label{tbl:limmfcoefs}
}
\end{table}
\fi

\begin{table}[h]
\centering
{\footnotesize
\caption{Exact coefficients for $k$-step fixed stepsize {\sc Limm-w} methods of order $k$, for $k=1 \dots 5$.}
\subfloat[{\sc Limm-w} 1-step order 1 coefficients.]{
\begin{minipage}{0.4\textwidth}
\centering
$
\renewcommand{\arraystretch}{1.5}
\begin{array}{c|cc}
 i & -1 & 0 \\
\hline
 \alpha_i & 1 & -1 \\
 \beta_i & 0 & 1 \\
 \mu_i & 1 & -1 \\
\end{array}

$
\end{minipage}
\label{tbl:limmwfcoef1}
}
\subfloat[{\sc Limm-w} 2-step order two coefficients.]{
\begin{minipage}{0.5\textwidth}
\centering
$
\renewcommand{\arraystretch}{1.5}
\begin{array}{c|ccc}
 i & -1 & 0 & 1 \\
\hline
 \alpha_i & 1 & -\frac{146619050}{133414177} & \frac{13204873}{133414177} \\
 \beta_i & 0 & \frac{193518829}{133414177} & -\frac{73309525}{133414177} \\
 \mu_i & \frac{73309525}{133414177} & -\frac{146619050}{133414177} & \frac{73309525}{133414177} \\
\end{array}

$
\end{minipage}
\label{tbl:limmwfcoef2}
}

\subfloat[{\sc Limm-w} 3-step order three coefficients.]{
$
\renewcommand{\arraystretch}{1.5}
\begin{array}{c|cccc}
 i & -1 & 0 & 1 & 2 \\
\hline
 \alpha_i & 1 & -\frac{192592391}{118869921} & \frac{41981416}{61945353} & -\frac{5229175002546}{90906657005273} \\
 \beta_i & 0 & \frac{16233524076078647}{9817918956569484} & -\frac{4193351041739980}{2454479739142371} & \frac{4833530710149845}{9817918956569484} \\
 \mu_i & \frac{4833530710149845}{9817918956569484} & -\frac{4833530710149845}{3272639652189828} & \frac{4833530710149845}{3272639652189828} & -\frac{4833530710149845}{9817918956569484} \\
\end{array}

$
\label{tbl:limmwfcoef3}
}

\subfloat[{\sc Limm-w} 4-step order 4 coefficients.]{
$
\renewcommand{\arraystretch}{1.5}
\begin{array}{c|cc}
 i & -1 & 0 \\
\hline
 \alpha_i & 1 & -\frac{68547635}{35752838} \\
 \beta_i & 0 & \frac{136586035293284691}{70863342514650928} \\
 \mu_i & \frac{719593273725529014067099}{1590107007076830596400464} & -\frac{719593273725529014067099}{397526751769207649100116} \\
\hline
 i & 1 & 2 \\
\hline
 \alpha_i & \frac{332147775}{246829693} & -\frac{120323842}{247754257} \\
 \beta_i & -\frac{4675749204985773774031537}{1590107007076830596400464} & \frac{3052167106160890365719135}{1590107007076830596400464} \\
 \mu_i & \frac{2158779821176587042201297}{795053503538415298200232} & -\frac{719593273725529014067099}{397526751769207649100116} \\
\hline
 i & 3 & \text{} \\
\hline
 \alpha_i & \frac{11382486133370227314625}{198763375884603824550058} & \text{} \\
 \beta_i & -\frac{719593273725529014067099}{1590107007076830596400464} & \text{} \\
 \mu_i & \frac{719593273725529014067099}{1590107007076830596400464} & \text{} \\
\end{array}

$
\label{tbl:limmwfcoef4}
}

\subfloat[{\sc Limm-w} 5-step order 5 coefficients.]{
$
\renewcommand{\arraystretch}{1.5}
\begin{array}{c|cc}
 i & -1 & 0 \\
\hline
 \alpha_i & 1 & -\frac{170476503}{75237041} \\
 \beta_i & 0 & \frac{3317715388830682274181888772466725}{1533160577078234002169550303186624} \\
 \mu_i & \frac{659152962863648794216719015147251}{1533160577078234002169550303186624} & -\frac{3295764814318243971083595075736255}{1533160577078234002169550303186624} \\
\hline
 i & 1 & 2 \\
\hline
 \alpha_i & \frac{124149029}{52265116} & -\frac{53697673}{39342191} \\
 \beta_i & -\frac{3387422206381293505203420155442595}{766580288539117001084775151593312} & \frac{294683351120793575703659865634035}{63881690711593083423731262632776} \\
 \mu_i & \frac{3295764814318243971083595075736255}{766580288539117001084775151593312} & -\frac{3295764814318243971083595075736255}{766580288539117001084775151593312} \\
\hline
 i & 3 & 4 \\
\hline
 \alpha_i & \frac{67073128}{206463953} & -\frac{2219582774479398588921363466455}{31940845355796541711865631316388} \\
 \beta_i & -\frac{1632980052046035774065588376123413}{766580288539117001084775151593312} & \frac{659152962863648794216719015147251}{1533160577078234002169550303186624} \\
 \mu_i & \frac{3295764814318243971083595075736255}{1533160577078234002169550303186624} & -\frac{659152962863648794216719015147251}{1533160577078234002169550303186624} \\
\end{array}

$
\label{tbl:limmwfcoef5}
}
\label{tbl:limmwfcoefs}
}
\end{table}

\section{Variable Stepsize and Variable Order Implementation}
\label{sec:impdetails}

In this section we discuss the details necessary to build an efficient self-starting variable stepsize and variable order implementation of {\sc Limm} methods. Following \cite[Chapter III.7]{Hairer_book_I}, a key component required for adapting stepsize and order is an estimate $e^{(k,h)}_{n+1}$ of the local truncation error when the solution  $y_{n+1}$ is computed with an order $k$ {\sc Limm} scheme and step size $h_n$. Based on this error estimate one can determine whether to accept or reject the current step, and can estimate the optimal stepsize for a method of order $k$ using: 
\begin{equation}
\label{eqn:hopt}
h_{\text{opt}}^{(k)} = h_n \, \Vert e^{(k,h)}_{n+1} \Vert^{-1/(k+1)}.
\end{equation}
To determine when a change of method order is needed, we further require the error estimates $e^{(k-1,h)}_{n+1}$ and $e^{(k+1,h)}_{n+1}$ for the order $k-1$ and $k+1$ methods, respectively. Then, we can select for the next step the method order which gives the best balance between a low estimated error and a large optimal timestep.


In traditional linear multistep methods \cite[Chapter III.7]{Hairer_book_I} a Taylor expansion reveals that
\begin{equation}
e^{(k,h)}_{n+1}  = C_k(\bm{c})\, h_n^{k+1}\, y^{(k+1)}(t_{n+1}) + \mathcal{O}\left(h_n^{k+2}\right),
\label{eqn:localerr}
\end{equation}
where $C_k(\bm{c})$ is the method error coefficient as a function of the stepsize ratios $\bm{c} = [c_i]_{i=-1}^{k}$ \eqref{eqn:abscissae}, derived from residuals on the $k+1$ order conditions. In the {\sc Limm} case different trees in the B-series expansion contribute differently to the $\mathcal{O}(h_n^{k+1})$ error term in \eqref{eqn:localerr}, as can be seen from \eqref{eqn:LIMW-order-conditions-traditional}--\eqref{eqn:LIMW-order-conditions-mu}. In order to build practical error estimators we will use an expansion of the form \eqref{eqn:localerr} with our error constant computed from the two residuals \eqref{eqn:lerr_res}:
\begin{equation}
\label{eqn:localerr-constant}
C_k(\bm{c}) = \frac{1}{(k+1)!}\max \left( \left|\rho_{k+1,a}(\bm{c})\right|,\,\left|\rho_{k+1,a}(\bm{c}) + \rho_{k+1,b}(\bm{c})\right| \right),
\end{equation}
where the sum of the residuals is used because, for {\sc Limm}, \eqref{eqn:lerr_res_b} always appears on trees alongside \eqref{eqn:lerr_res_a}, as described in the observations below \eqref{eqn:LIMM_gen-bseries-4}.
We use past solutions for $y_{n+1},\dots,y_{n-k}$ to approximate the $(k+1)$-st time derivative of the solution via divided differences:
\begin{equation}
\delta^{k+1} y\left[t_{n+1},\dots,t_{n-k}\right] = \frac{y^{(k+1)}(\xi)}{(k+1)!}, \quad \xi \in [t_{n-k},t_{n+1}].
\label{eqn:mvtdd}
\end{equation}
Putting together equations \eqref{eqn:localerr}, \eqref{eqn:localerr-constant}, and \eqref{eqn:mvtdd} we arrive at the following estimate for the local error.
\begin{definition}[Local error estimate]
\label{def:localerr}
For a $k$-step order $k$ {\sc Limm} method a practical local truncation error estimate is:
\begin{equation}
\label{eqn:error-estimator}
e^{(k,h)}_{n+1} = (k+1)!\, C_k(\bm{c})\, h_n^{k+1}\, \delta^{k+1} y\left[t_{n+1},\dots,t_{n-k}\right],
\end{equation}
using method error coefficient $C_k(\bm{c})$ \eqref{eqn:localerr-constant}, the current stepsize $h_n$, and the $(k+1)$-st order divided difference of the solution at $t_{n+1},\dots,t_{n-k}$. Due to  \eqref{eqn:localerr-constant}, \eqref{eqn:error-estimator} may slightly overestimate the local errors in the asymptotic regime.
\end{definition}
This error estimate is convenient for our purpose, as we can compute estimates for methods of order $k-1$ and $k+1$ by simply using their error coefficients and changing the order of the divided difference. Note that the divided difference of order $k+1$ requires $y_{n+1}$ and a history of $k+1$ values of $y_n,\dots,y_{n-k}$, whereas the $k$-step {\sc Limm} method only requires a history of $k$ values of $y$ in order to compute $y_{n+1}$. Divided differences for the higher order methods require storing additional past solution values.

Instead of building divided differences at each step from the history of solutions, it is more efficient to maintain a history of the divided differences and build new ones using the following formulas (similar to updating the Nordsieck vector in a variety of multistep codes \cite{byrne1975, brown1989, radhakrishnan1993}). First, after computing $y_{n+1}$ via the {\sc Limm} method, we produce the $(k+1)$-st divided difference as follows:
\begin{equation}
\label{eqn:ddifkp1}
\delta^{k+1} y\left[t_{n+1},\dots,t_{n-k}\right] = \frac{y_{n+1} - y_n}{\prod_{j=0}^{k} \sum_{l=0}^{j} (t_{n+1-l} - t_{n-l})} - \sum_{i=1}^{k} \frac{\delta^{i} y\left[t_{n},\dots,t_{n-i}\right]}{\prod_{j=i}^{k} \sum_{l=0}^{j} (t_{n+1-l} - t_{n-l})}.
\end{equation}
Then, through direct application of the recursive definition, one computes the updated values of order $k$ and $k+2$ divided differences.
\iflong
\begin{subequations}
\begin{equation}
\label{eqn:ddifk}
\delta^{k} y\left[t_{n+1},\dots,t_{n-k+1}\right] = \delta^{k} y\left[t_{n},\dots,t_{n-k}\right] + \sum_{j=0}^{k} h_{n-j}\, \delta^{k+1} y\left[t_{n+1},\dots,t_{n-k}\right],
\end{equation}
\begin{equation}
\label{eqn:ddifkp2}
\delta^{k+2} y\left[t_{n+1},\dots,t_{n-k-1}\right] = \frac{\delta^{k+1} y\left[t_{n+1},\dots,t_{n-k}\right] - \delta^{k+1} y\left[t_{n},\dots,t_{n-k-1}\right]}{\sum_{j=0}^{k+1} h_{n-j}}.
\end{equation}
Together, this requires saving divided differences of orders $0,\dots,k+2$ at each timestep (with the order 0 coming from the direct use of $y_n$ in \eqref{eqn:ddifkp1}, and order $k+2$ needed in the event we choose to use a higher order method for the next step). Order $1,\dots,k-1$ differences can be updated via repeated application of \eqref{eqn:ddifk}:
\begin{equation}
\delta^{i-1} y\left[t_{n+1},\dots,t_{n-i+2}\right] = \delta^{i-1} y\left[t_{n},\dots,t_{n-i+1}\right] + \sum_{j=0}^{i-1} h_{n-j} \delta^{i} y\left[t_{n+1},\dots,t_{n-i+1}\right].
\end{equation}
\end{subequations}
\fi

To avoid the redundant saving of both the history of solution values and of divided differences, one can reformulate the {\sc Limm} method \eqref{eqn:LIMM_definition} in terms of divided differences:
\begin{equation}
\begin{split}
&\left(\Id - h_n\, \mu_{-1}(\bm{c})\, \J_n\right) y_{n+1}= -\sum_{i=0}^{k-1} h_n^i\, \widehat{\alpha}_i(\bm{c})\, \delta^i y\left[t_n,\dots,t_{n-i}\right] \\ 
&\qquad + h\, \J_n\, \sum_{i=0}^{k-1} h_n^i \, \widehat{\mu}_i(\bm{c})\, \delta^i y\left[t_n,\dots,t_{n-i}\right]
+ h_n\, \sum_{i=0}^{k-1} h_n^i\, \widehat{\beta}_i(\bm{c}) \, \delta^i f\left[t_n,\dots,t_{n-i}\right],
\end{split}
\end{equation}
where $\delta^i f\left[t_n,\dots,t_{n-i}\right]$ are divided differences constructed from evaluations of the right-hand side function $f(t_n, y_n)$ evaluated at $(t_n, y_n)$, and $\widehat{\alpha}_i$, $\widehat{\mu}_i$, $\widehat{\beta}_i$ are transformed coefficients (as functions of the timestep ratios). The transformed coefficients can be computed from the original method coefficients as
\iflong
\begin{align}
\widehat{\alpha}_i(\bm{c}) & = (-1)^i \sum_{j=i}^{k-1} \alpha_j(\bm{c}) \prod_{l=1}^{i} \sum_{m=l}^{j} (c_m - c_{m-1}), \\
\widehat{\beta}_i(\bm{c}) & = (-1)^i \sum_{j=i}^{k-1} \beta_j(\bm{c}) \prod_{l=1}^{i} \sum_{m=l}^{j} (c_m - c_{m-1}), \\
\widehat{\mu}_i(\bm{c}) & = (-1)^i \sum_{j=i}^{k-1} \mu_j(\bm{c}) \prod_{l=1}^{i} \sum_{m=l}^{j} (c_m - c_{m-1}).
\end{align}
\fi

\begin{remark}[The first step]
The first step of integration requires additional care. Even starting with a one-step order one method requires an order two divided difference for error estimation (and a difference of order three for estimation of the order two method's error). We initialize our divided differences as follows:
\begin{align*}
\delta^0 y[t_0] & = y_0, \quad
\delta^1 y\left[t_0\right] \approx f_0.
\end{align*}
Then, following the computation of $y_1$, we approximate the next divided differences
\begin{align*}
\delta^2 y\left[t_1, t_0\right] & \approx \left((y_1 - y_0)/h_0  - f_0\right)/\left(2 h_0\right), \quad
\delta^3 y\left[t_1, t_0\right] \approx \left(\delta^2 y\left[t_1, t_0\right] - \J_0 f_0\right)/\left(3 h_0\right).
\end{align*}
\end{remark}
%

\begin{remark}[Stability for variable stepsize methods]
Evaluating the stability properties of variable stepsize multistep methods is challenging (see \cite{Calvo1987, Calvo1990, Calvo1993, Guglielmi2001} for analysis of variable stepsize BDF methods). 
Specifically, explicit formulas for the bounds $\omega_{\rm min}$ and $\omega_{\rm max}$ in Theorem \ref{thm:convergence} that ensure stability for a general multistep methods are not available. 
Since an in-depth analysis is outside the scope of the current work, we follow the best practices of other variable stepsize multistep implementations \cite[Section III.5]{Hairer_book_I}, by ensuring the base fixed stepsize method (from which the variable stepsize one is derived) is stable, and limiting stepsize increases to only occur after $k$ successful steps at the current stepsize.
\iflong 
Specifically, we allow step changes occur after only $k+1$ successful steps have been taken with the current stepsize.
\fi
\end{remark}

\section{Numerical Results}
\label{sec:results}

In this section we present numerical results for {\sc Limm} methods including both convergence and performance experiments. First, we demonstrate the convergence order of the fixed stepsize {\sc Limm} and {\sc Limm-w} methods, then we examine the performance of the self-starting variable stepsize and variable order implementation of the {\sc Limm} methods which was detailed in section \ref{sec:impdetails}.

The convergence tests use fixed stepsize implementations of {\sc Limm} and {\sc Limm-w}, with the built-in Matlab time integrator \texttt{ode15s} used as a starter method.  For the performance tests, our primary point of comparison is the BDF methods \cite[Section III.1]{Hairer_book_I}, one of the most widely used families of implicit multistep methods. We use a self-starting variable stepsize and order implementation of BDF built in the same framework as the {\sc Limm} implementation, using the same error controller and also the error estimator from definition \ref{def:localerr}. For comparison we include results for SDIRK4a \cite[Section IV.6]{Hairer_book_II} (a 5-stage, fourth order singly diagonally implicit Runge-Kutta method) and RODAS4 \cite[Section VI.4]{Hairer_book_II} (a 6-stage, fourth order Rosenbrock method) as examples of high-order single step implicit methods. The variable stepsize {\sc Limm-(w)} and BDF implementations, along with the SDIRK4a and RODAS4 methods can be found in our MatlODE package \cite{matlode}, in the \texttt{limm\_dev} branch. We also provide comparisons against the Matlab built-in \texttt{ode15s} in both BDF and NDF modes.

Each time integrator is tested by sweeping through the same series of tolerances from $10^{-2}, 10^{-3}, ..., 10^{-10}$, with equal relative and absolute tolerances. When GMRES is used for linear system solves, it uses a tolerance  that is one tenth of the integrator tolerance. All tests are performed in Matlab (version 2019a) on a workstation with dual Intel Xeon E5-2650 v3 processors with 20 cores (40 threads) and 128GB of memory. Test problems are drawn from the ODE Test Problems package \cite{roberts2019otp}, which contains Matlab implementations of a variety of ODEs and PDEs suitable for testing time integrators. Reference solutions are obtained by applying the Matlab built-in \texttt{ode15s} integrator when full Jacobians are available, or an implementation of RODAS4 using GMRES when only Jacobian-vector products are available; both are applied with the tightest possible tolerance of $100 $ times the machine precision.

\subsection{Convergence: Lorenz-96 ODE}

To test the convergence order of the fixed stepsize {\sc Limm} and {\sc Limm-w} methods, we use the Lorenz-96 problem \cite{lorenz96}, a system of nonlinear ODEs defined as
\begin{equation}
\label{eqn:Lorenz}
\frac{\text{d}x_i}{\text{d}t} = \left(x_{i+1} - x_{i-2}\right)x_{i-1} - x_i + F(t), \quad i = 1,\dots,N,
\end{equation}
with periodic boundary conditions $x_{N+1} = x_1$, $x_0 = x_N$, $x_{-1} = x_{N-1}$ and forcing function $F(t)= 8 + 4 \cos \left(3 \pi t\right)$. We select $N = 40$, and timespan $t \in \left[0, 0.5\right]$.

Figure \ref{fig:lorenz96conv} shows the convergence results for {\sc Limm} orders 1-5 as solid lines, and for {\sc Limm-w} orders 1-5 as dashed lines. The legend contains the slopes of the linear interpolants for each test, demonstrating that each method achieves its theoretical order. {\sc Limm-w} shows slightly smaller total errors.

\begin{figure}[h]
\centering
\includegraphics[trim=5 5 50 5, clip, width=0.7\textwidth]{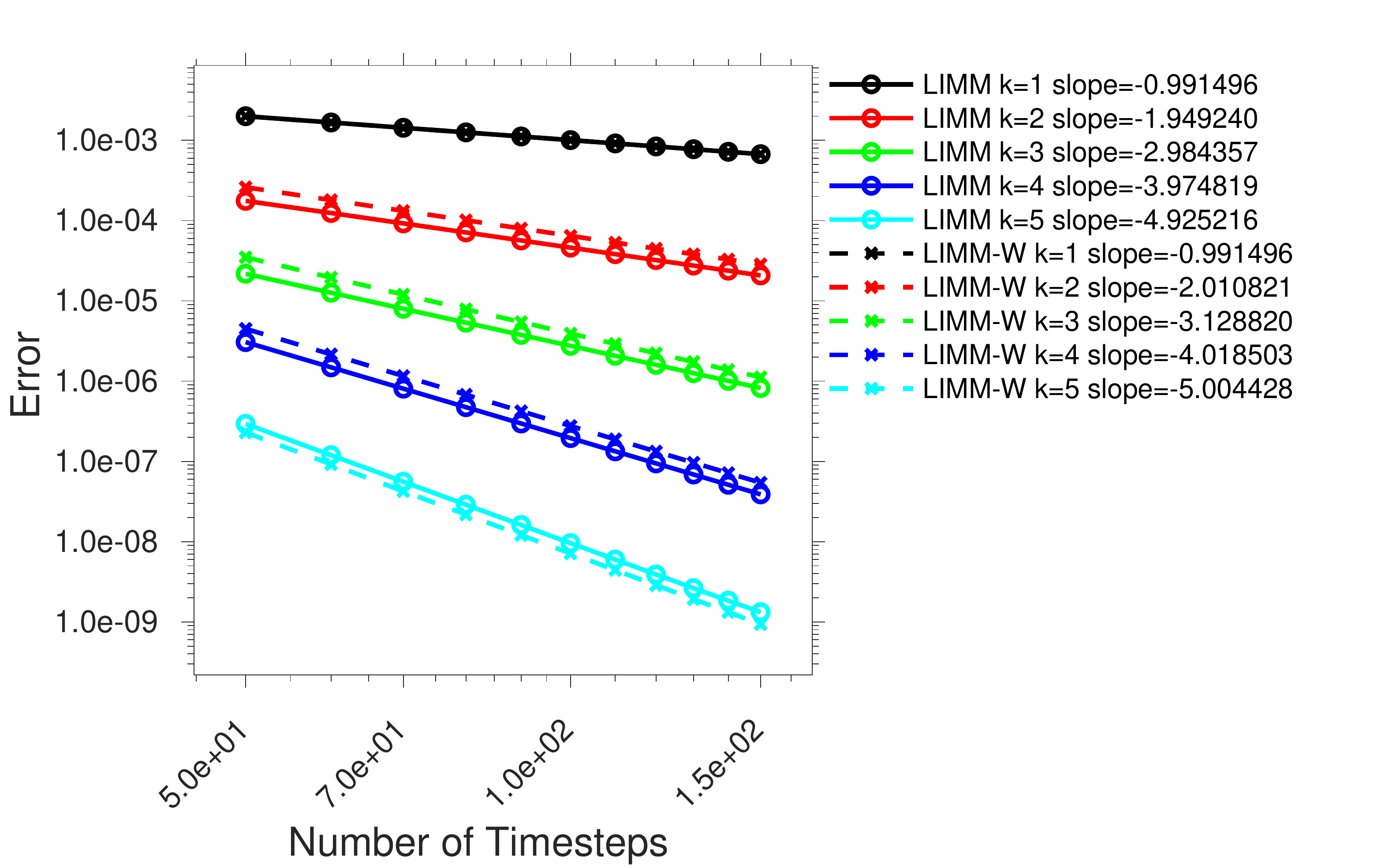}
\caption{Fixed stepsize results for the Lorenz-96 problem \eqref{eqn:Lorenz} showing the orders of convergence for {\sc Limm-w} (dashed lines) and {\sc Limm} (solid lines) methods for orders 1-5.}
\label{fig:lorenz96conv}
\end{figure}

\subsection{Performance: Gray--Scott reaction-diffusion model}

We run performance comparisons using the Gray--Scott reaction-diffusion model, which describes a two species chemical reactions with retirement \cite{Gray1983, Gray1984}:
\iflong
\begin{align*}
U + 2V & \rightarrow 3V \\
V & \rightarrow P.
\end{align*}
\fi
%
%
\begin{equation}
\label{eqn:GS}
\frac{\partial u}{\partial t}  = \varepsilon_1 \Delta u - u v^2 + F(1-u),\qquad
\frac{\partial v}{\partial t}  = \varepsilon_2 \Delta v + u v^2 - (F + k)v, 
\end{equation}
where $\varepsilon_1= 0.2$ and $\varepsilon_2= 0.1$ are diffusion rates, and $F= 0.04$ and $k= 0.06$ are reaction rates. We use a second order finite difference spatial discretization with periodic boundary conditions on a $128 \times 128$ 2D grid, which brings \eqref{eqn:GS} to ODE form \eqref{eqn:ode} with $N = 2\times 128 \times 128$. The simulation time interval is $t \in [0, 2]$.

Results on the Gray--Scott model are given in Figure \ref{fig:grayscottresults} for our implementations of {\sc Limm}, {\sc Limm-w}, BDF, SDIRK4a, RODAS4, as well as for the \texttt{ode15s} implementations of BDF and NDF \cite{shampine1997}. Because the full Jacobian is available, all methods use a direct $LU$-factorization, that is reused for methods requiring multiple linear solves per step. The error vs. timestep count results in Figure \ref{fig:grayscottsteps} reveal that all five  multistep integrators  achieve similar levels of error with similar numbers of timesteps, thus none of these multistep methods or implementations have a clear stability or error advantage over the others in this test. The one-step methods SDIRK4a and RODAS4 do show better errors with fewer timesteps, likely as a result of their greater flexibility in choosing timestep sizes.  The CPU time comparison in Figure \ref{fig:grayscotttimes} reveals the advantage of the built-in \texttt{ode15s} implementations, as they benefit from a massive amount of software optimizations. Comparing the results for our implementation of multistep integrators, we can see a small performance advantage for {\sc Limm} and {\sc Limm-w} over the similar BDF implementation, due to the need for only one linear system solve per timestep. Given a similar level of software optimization, we might expect the {\sc Limm} methods to hold the same advantage over the \texttt{ode15s} BDF implementation. The SDIRK4a and RODAS4 also show very good performance, due to the lower timestep count and reuse one $LU$-factorization per timestep.

\begin{figure}[h]
\centering
\subfloat[Variable stepsize convergence results.]{
\includegraphics[trim=5 5 50 5, clip, width=0.47\textwidth]{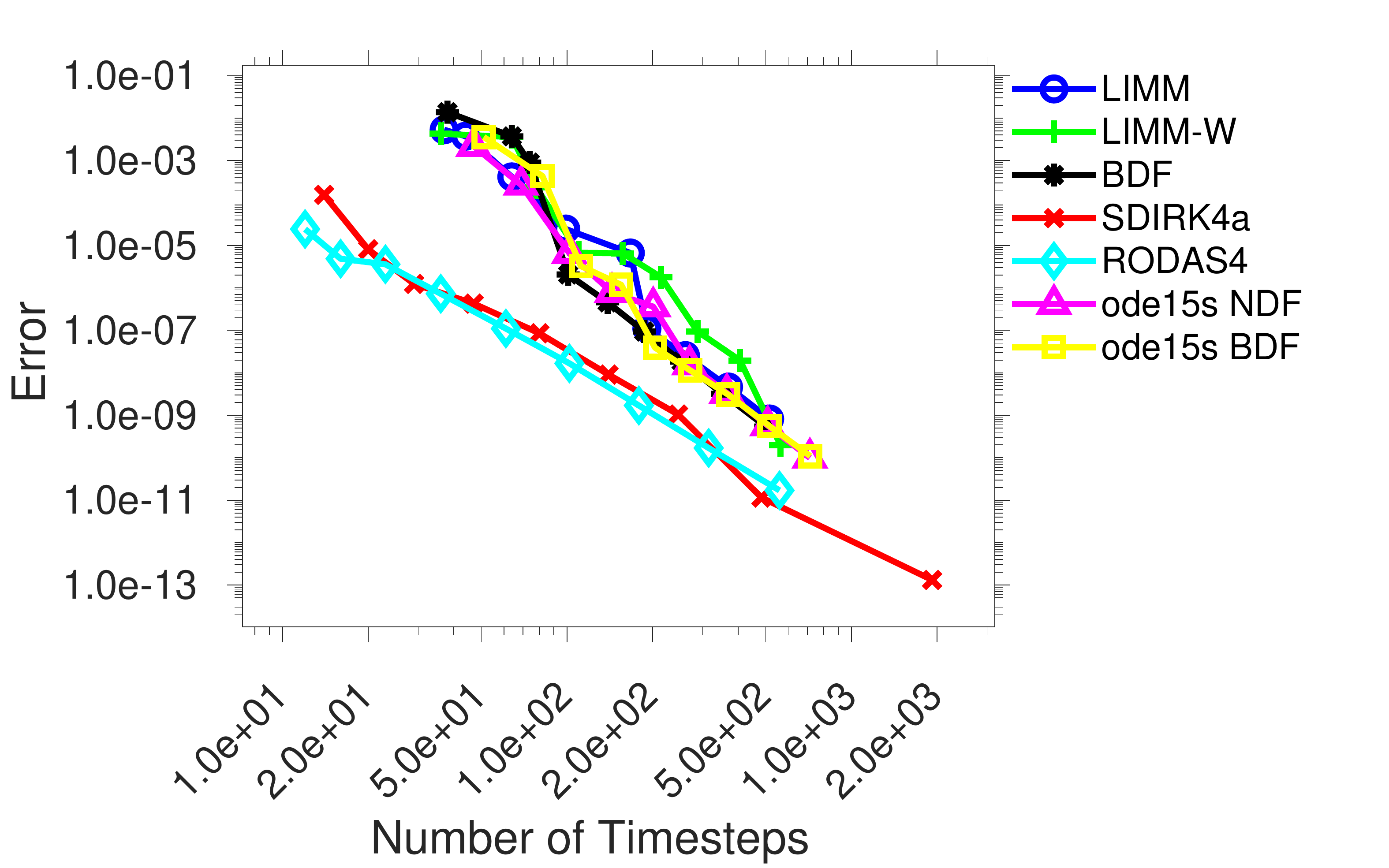}
\label{fig:grayscottsteps}
}
\subfloat[Work-precision diagram.]{
\includegraphics[trim=5 5 50 5, clip, width=0.47\textwidth]{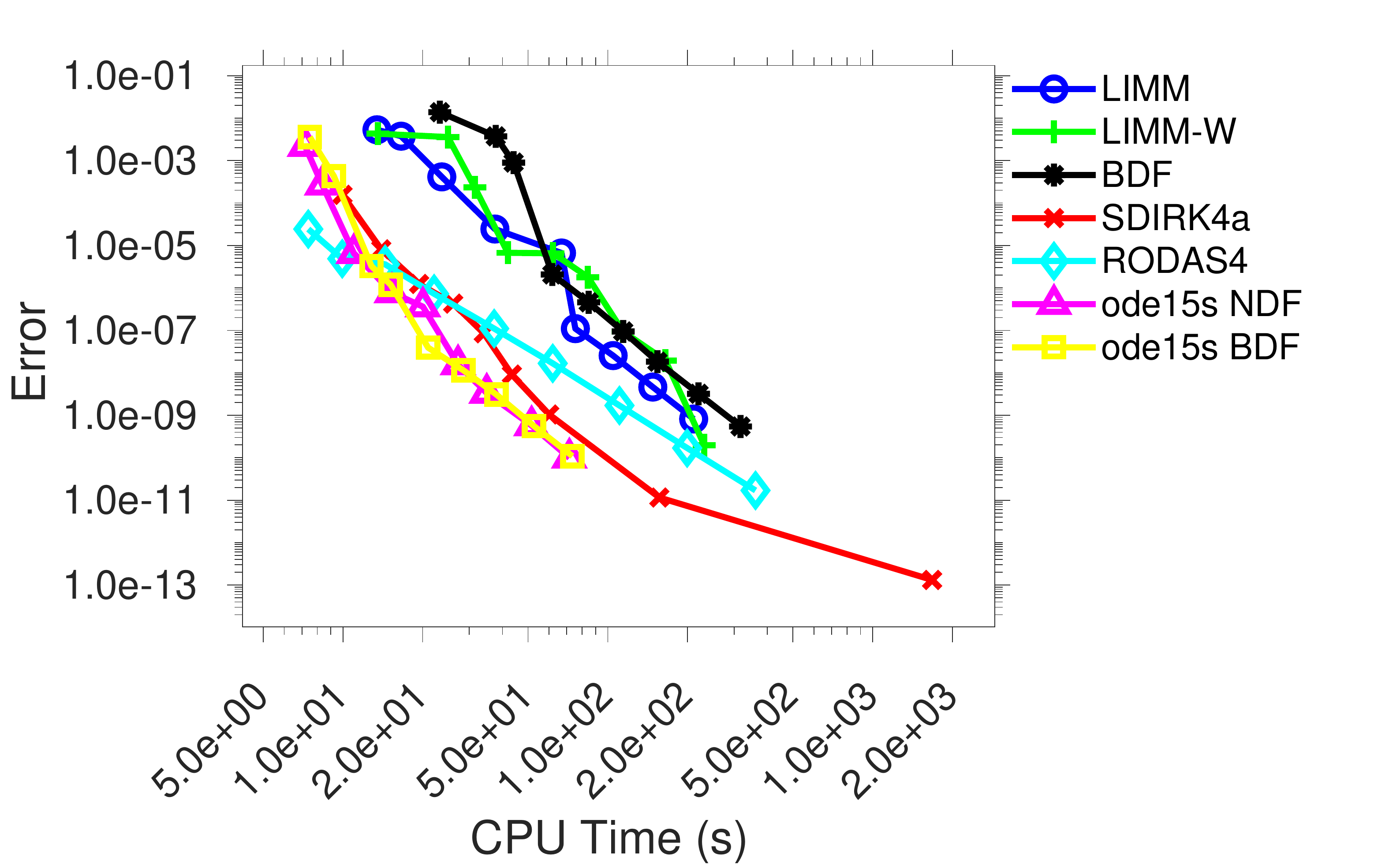}
\label{fig:grayscotttimes}
}
\caption{Performance results for different integrators applied to the 2D Gray-Scott reaction-diffusion problem \eqref{eqn:GS}.}
\label{fig:grayscottresults}
\end{figure}

\subsection{Performance: quasi-geostrophic model}

Next, we compare method performance using the 1.5-layer quasi-geostrophic (QG) model \cite{Sakov2007}, which provides a simplified representation of ocean dynamics:
\begin{align}
\label{eqn:QG}
  \displaystyle\frac{\partial q}{\partial t} & = -\psi_x - \varepsilon J\left(\psi,q\right) - A\Delta^3\psi + 2\pi \text{sin}\left(2\pi y\right), \qquad
  q = \Delta \psi - F\psi,
\end{align}
where $J\left(\psi,q\right) = \psi_x q_y - \psi_y q_x$, and $F = 1600$, $\varepsilon = 10^{-5}$ and $A = 10^{-5}$ are constants.  The implementation discretizes the system in terms of the stream function $\psi$, on the spatial domain $(x,y) \in [0,1]^2$, using second order central finite differences and homogeneous Dirichlet boundary conditions on a $127 \times 127$ grid. Integration is performed over the time span $t \in \left[0, 0.01\right]$. Due to the required solution of a Helmholtz equation, a portion of the QG Jacobian is dense; only Jacobian-vector products are available and all integrators use GMRES as linear solver, except for \texttt{ode15s} which cannot use iterative linear solvers. Instead, both BDF and NDF results for \texttt{ode15s} make use of a numerical Jacobian approximation.

Figure \ref{fig:qgresults} shows errors versus timestep count and CPU times for {\sc Limm}, {\sc Limm-w}, BDF, SDIRK4a, RODAS4, and \texttt{ode15s} applied to the QG model. First, we notice the jagged behavior of both the {\sc Limm} and BDF methods, especially at looser tolerances; this is an artifact of the variable order methods, where methods of lower order are preferred for their better stability. When the tolerances are tightened, the graphs smoothen, as the algorithm prefers the better asymptotic errors of the higher order methods. The {\sc Limm} and {\sc Limm-w} methods are particularly sensitive to this, possibly due to the use of GMRES, as inexact linear solves could degrade stability or lead to {\sc Limm} order reduction. Despite this, Figure \ref{fig:qgtimes} shows that {\sc Limm} retains its CPU time advantage over BDF, and also matches or outperforms RODAS4 and SDIRK4a for similar levels of error, as the small stability advantages that BDF, SDIRK and Rosenbrock methods have cannot overcome their larger cost per timestep. The built-in \texttt{ode15s} methods both provide very poor performance for this test case, as they do not support the use of direct Jacobian-vector products and are forced to build numerical Jacobian approximations to use with direct solves.

\begin{figure}[h]
\centering
\subfloat[Variable stepsize convergence results.]{
\includegraphics[trim=5 5 50 5, clip, width=0.47\textwidth]{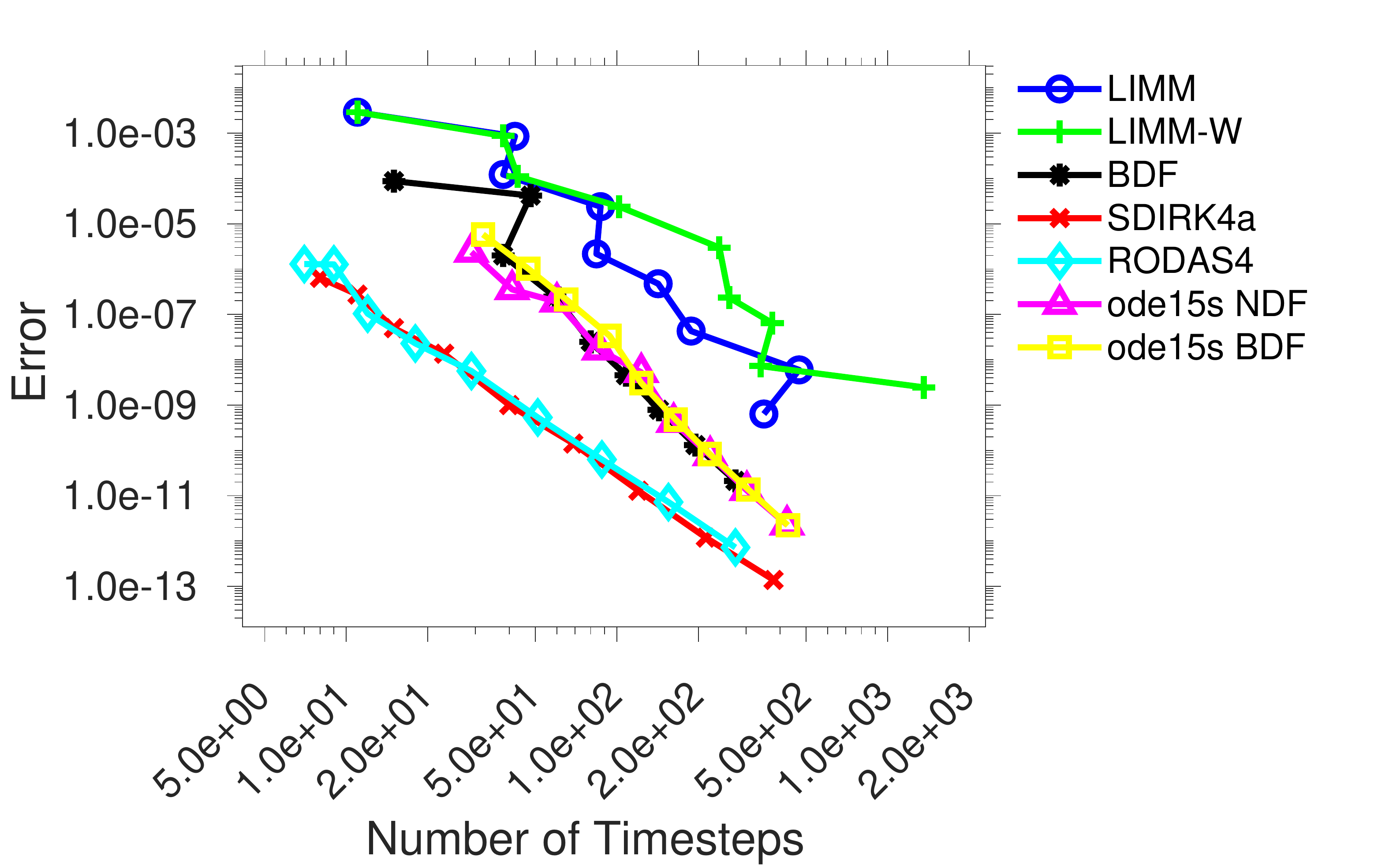}
\label{fig:qgsteps}
}
\subfloat[Work-precision diagram.]{
\includegraphics[trim=5 5 50 5, clip, width=0.47\textwidth]{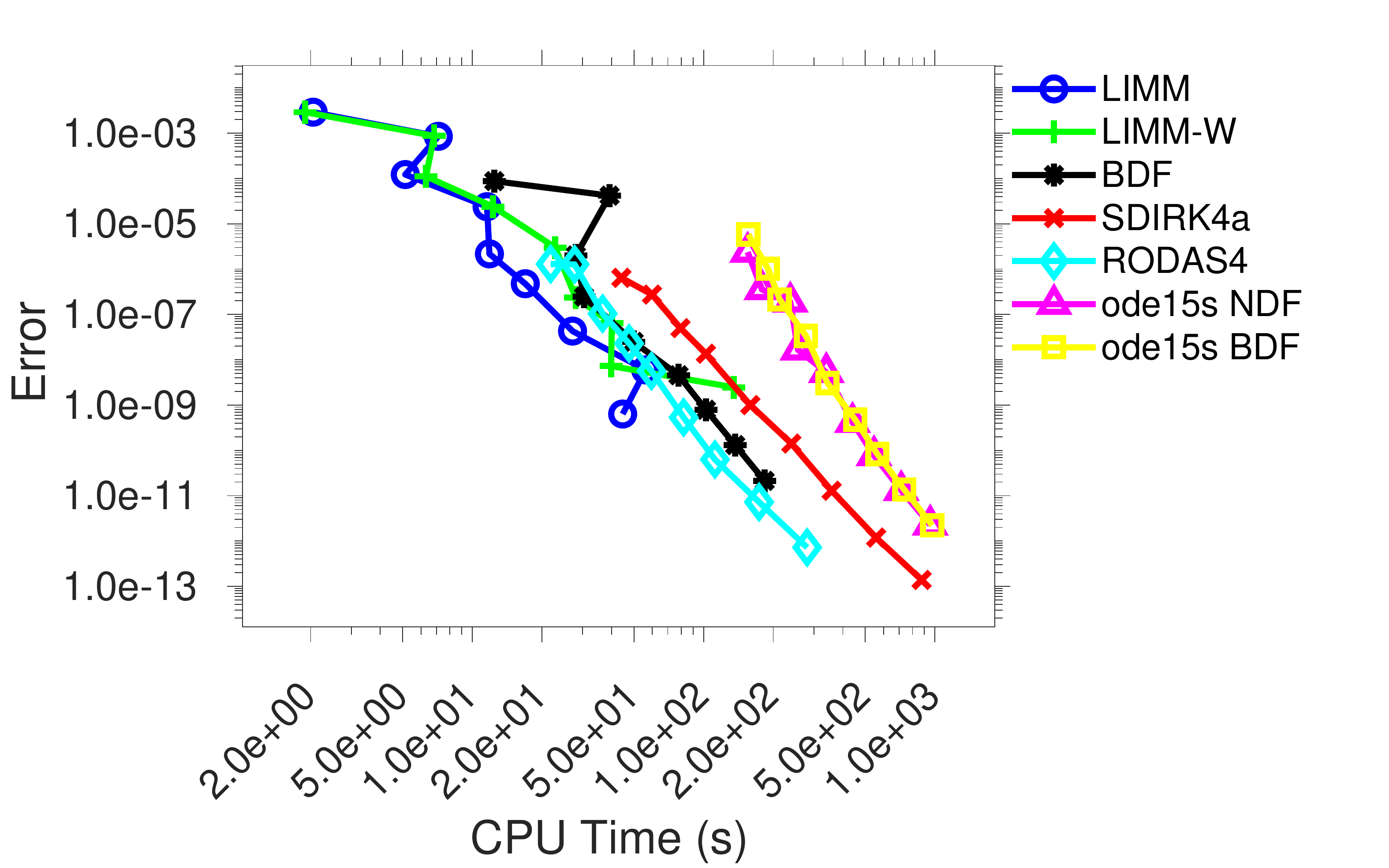}
\label{fig:qgtimes}
}
\caption{Performance results  for different integrators applied to the 1.5 layer QG model \eqref{eqn:QG}.}
\label{fig:qgresults}
\end{figure}

\section{Conclusions}
\label{sec:conclusions}

Classical implicit linear multistep methods require the solution of a nonlinear system of equations at each timestep. This work develops the {\sc Limm} class of linearly implicit multistep methods, which only require the solution of one linear system per timestep. Order conditions for variable stepsize {\sc Limm} methods of arbitrary order are constructed via Butcher trees and B-series operations. Order conditions for {\sc Limm-w} schemes that can use arbitrary approximations of the Jacobian matrix are developed by direct series expansion. A {\sc Limm} method has twice as many coefficients as a traditional linear multistep method (with the same number of steps), and this additional freedom enables the construction of methods with excellent accuracy and stability properties. We discuss the optimal design of  {\sc Limm} methods, and develop a set of $k$-step order $k$ {\sc Limm} methods for $k = 1,\dots,5$. We also discuss the details of a self-starting variable stepsize and variable order implementation of these {\sc Limm} methods. Numerical experiments demonstrate the convergence at the theoretical orders for the fixed stepsize {\sc Limm} methods. Moreover, our variable stepsize {\sc Limm} implementation outperforms a comparable BDF implementation for nonlinear problems.

\iflong
In the future, we intend to use the family of {\sc Limm-w} methods to reuse the same Jacobian for multiple timesteps. Additionally, extension of {\sc Limm} to use a Krylov subspace approximation of the Jacobian, as in Rosenbrock-Krylov or exponential-Krylov methods is underway, which will directly account for errors coming from Krylov-based iterative linear solvers in the order conditions of the time integrator.
\fi

\ifnosup
\appendix
\section{Variable stepsize {\sc Limm} coefficients}
\label{sec:limmvarcoeff}

Tables \ref{tbl:limmvcoef1} -- \ref{tbl:limmvcoef5b} contain the variable stepsize coefficient expressions for {\sc Limm} methods of order 1-5, parameterized by the $c_i$'s defined in \eqref{eqn:abscissae}, and the constant coeffcients.

\begin{table}[H]
\centering
\caption{{\sc Limm} 1-step order 1 coefficients.}
\label{tbl:limmvcoef1}
{\footnotesize
\renewcommand{\arraystretch}{2}
$\begin{array}{c|cc}
 i & -1 & 0 \\
\hline
 \alpha_i & 1 & -1 \\
 \beta_i & 0 & 1 \\
 \mu_i & 1 & -1 \\
\end{array}
$
}
\end{table}

\begin{table}[H]
\centering
\caption{{\sc Limm} 2-step order 2 variable stepsize coefficients.}
\label{tbl:limmvcoef2}
{\footnotesize
\renewcommand{\arraystretch}{2}
$\begin{array}{c|ccc}
 i & -1 & 0 & 1 \\
\hline
 \alpha_i & 1 & -\frac{4}{3} & \frac{1}{3} \\
 \beta_i & 0 & \frac{2}{3} & -\beta_0+\left(\alpha_0+1\right) c_1+1 \\
 \mu_i & \frac{1}{2} \left(1-\left(\alpha_0+1\right) c_1^2\right) & \frac{1}{2} \left(-2 \beta_0+\left(\alpha_0+1\right) c_1^2+2 \left(\alpha_0+1\right) c_1+1\right) & \beta_0+\left(\alpha_0+1\right) \left(-c_1\right)-1 \\
\end{array}
$
}
\end{table}

\begin{table}[H]
\centering
\caption{{\sc Limm} 3-step order 3 variable stepsize coefficients.}
\label{tbl:limmvcoef3}
{\tiny
\renewcommand{\arraystretch}{2}
\begin{turn}{-90}
$\begin{array}{c|l}
\hline
 \alpha_{-1} & 1 \\
 \beta_{-1} & 0 \\
 \mu_{-1} & \left[\alpha_1 c_1^3-3 c_1 \left(\left(\alpha_0+\alpha_1+1\right) c_2^2-1\right)+2 \left(\alpha_0+\alpha_1+1\right) c_2^3+2\right]/\left(6 \left(c_1+1\right)\right) \\
\hline
 \alpha_0 & -67569925/40220258 \\
 \beta_0 & 6/11 \\
 \mu_0 & -\left[-3 c_1 \left(-2 \beta_0+\left(\alpha_0+\alpha_1+1\right) c_2^2+2 \left(\alpha_0+\alpha_1+1\right) c_2+1\right)+\alpha_1 c_1^3+3 \alpha_1 c_1^2+2 \left(\alpha_0+\alpha_1+1\right) c_2^3+3 \left(\alpha_0+\alpha_1+1\right) c_2^2-1\right]/\left(6 c_1\right) \\
\hline
 \alpha_1 & 77233903/99562899 \\
 \beta_1 & \left[-\alpha_1 c_1^3+3 \alpha_1 c_2^2 c_1-2 \left(\alpha_0+\alpha_1+1\right) c_2^3+3 \left(\beta_0-1\right) c_2^2+1\right]/\left(3 \left(c_1^2-c_2^2\right)\right) \\
 \mu_1 & \left[3 c_1^2 \left(c_2^2 \left(\alpha_0-2 \beta_0+3\right)+2 \left(\alpha_0+\alpha_1+1\right) c_2^3-1\right)+c_1 \left(4 \left(\alpha_0+\alpha_1+1\right) c_2^3-6 \left(\beta_0-1\right) c_2^2-2\right)-2 \left(\alpha_0+\alpha_1+1\right) c_2^5 \right. \\ & \left. -3 \left(\alpha_0+\alpha_1+1\right) c_2^4-4 \alpha_1 c_1^3 c_2^2-\alpha_1 c_1^4+c_2^2\right]/\left(6 c_1 \left(c_1+1\right) \left(c_1-c_2\right) \left(c_1+c_2\right)\right) \\
\hline
 \alpha_2 & -383355371802341/4004445485007942 \\
 \beta_2 & -\left[3 c_1^2 \left(\beta_0+\left(\alpha_0+\alpha_1+1\right) \left(-c_2\right)-1\right)+2 \alpha_1 c_1^3+\left(\alpha_0+\alpha_1+1\right) c_2^3+1\right]/\left(3 \left(c_1^2-c_2^2\right)\right) \\
 \mu_2 & \left[3 c_1^2 \left(\beta_0+\left(\alpha_0+\alpha_1+1\right) \left(-c_2\right)-1\right)+2 \alpha_1 c_1^3+\left(\alpha_0+\alpha_1+1\right) c_2^3+1\right]/\left(3 \left(c_1^2-c_2^2\right)\right) \\
\end{array}
$
\end{turn}
}
\end{table}

\begin{table}[H]
\centering
\caption{{\sc Limm} 4-step order 4 variable stepsize coefficients.}
\label{tbl:limmvcoef4}
{\tiny
\renewcommand{\arraystretch}{2}
\begin{turn}{-90}
$\begin{array}{c|l}
\hline
 \alpha_{-1} & 1 \\
 \beta_{-1} & 0 \\
 \mu_{-1} & \left[-\alpha_1 c_1^4+2 \alpha_1 c_2 c_1^3+2 c_1 \left(\alpha_2 c_2^3-3 c_2 \left(\left(\alpha_0+\alpha_1+\alpha_2+1\right) c_3^2-1\right)+2 \left(\alpha_0+\alpha_1+\alpha_2+1\right) c_3^3+2\right)-\alpha_2 c_2^4-3 \left(\alpha_0+\alpha_1+\alpha_2+1\right) c_3^4+4 c_2 \left(\left(\alpha_0+\alpha_1+\alpha_2+1\right) c_3^3+1\right)+3\right]/\left(12 \left(c_1+1\right) \left(c_2+1\right)\right) \\
\hline
 \alpha_0 & -60010656/28439311 \\
 \beta_0 & 12/25 \\
 \mu_0 & \left[-2 c_1 \left(-3 c_2 \left(-2 \beta_0+\left(\alpha_0+\alpha_1+\alpha_2+1\right) c_3^2+2 \left(\alpha_0+\alpha_1+\alpha_2+1\right) c_3+1\right)+\alpha_2 c_2^3+3 \alpha_2 c_2^2+2 \left(\alpha_0+\alpha_1+\alpha_2+1\right) c_3^3+3 \left(\alpha_0+\alpha_1+\alpha_2+1\right) c_3^2-1\right)+\alpha_1 c_1^4-2 \alpha_1 \left(c_2-1\right) c_1^3-6 \alpha_1 c_2 c_1^2 \right. \\ & \left. +3 \alpha_0 c_3^4 +4 \alpha_0 c_3^3+3 \alpha_1 c_3^4+4 \alpha_1 c_3^3+\alpha_2 c_2^4+3 \alpha_2 c_3^4+2 \alpha_2 c_2^3+4 \alpha_2 c_3^3-2 c_2 \left(2 \left(\alpha_0+\alpha_1+\alpha_2+1\right) c_3^3+3 \left(\alpha_0+\alpha_1+\alpha_2+1\right) c_3^2-1\right)+3 c_3^4+4 c_3^3+1\right]/\left(12 c_1 c_2\right) \\
\hline
 \alpha_1 & 71006953/40099309 \\
 \beta_1 & \left[4 c_2^2 \left(\alpha_1 c_1^3-3 \alpha_1 c_3^2 c_1+2 \left(\alpha_0+\alpha_1+\alpha_2+1\right) c_3^3-3 \left(\beta_0-1\right) c_3^2-1\right)+\alpha_2 c_2^5+\alpha_2 c_3 c_2^4-8 \alpha_2 c_3^2 c_2^3-c_2 \left(3 \alpha_1 c_1^4-4 \alpha_1 c_3 c_1^3+\left(\alpha_0+\alpha_1+\alpha_2+1\right) c_3^4+4 c_3+3\right) \right. \\ & \left. -c_3 \left(3 \alpha_1 c_1^4-4 \alpha_1 c_3 c_1^3 +\left(\alpha_0+\alpha_1+\alpha_2+1\right) c_3^4+4 c_3+3\right)\right]/\left(12 \left(c_1-c_2\right) \left(c_1-c_3\right) \left(c_2 c_3+c_1 \left(c_2+c_3\right)\right)\right) \\
 \mu_1 & -\left[-2 c_1^2 \left(-3 c_2^2 \left(c_3^2 \left(\alpha_0+\alpha_2-2 \beta_0+3\right)+2 \left(\alpha_0+\alpha_1+\alpha_2+1\right) c_3^3-1\right)+\alpha_2 c_2^4+\alpha_2 c_3 \left(4 c_3+1\right) c_2^3-c_2 \left(\left(\alpha_0+\alpha_1+\alpha_2+1\right) c_3^3-3 c_3-2\right)+2 c_3 \left(c_3+1\right) \left(\left(\alpha_0+\alpha_1+\alpha_2+1\right) c_3^3+1\right)\right) \right. \\ & \left. +c_1 \left(4 c_2^2 \left(2 \left(\alpha_0+\alpha_1+\alpha_2+1\right) c_3^3-3 \left(\beta_0-1\right) c_3^2-1\right)+\alpha_2 c_2^5+\alpha_2 c_3 \left(c_3+1\right) c_2^4-6 \alpha_2 c_3^2 c_2^3-c_2 \left(4 \left(\alpha_0+\alpha_1+\alpha_2+1\right) c_3^5+7 \left(\alpha_0+\alpha_1+\alpha_2+1\right) c_3^4-2 c_3^2+4 c_3+3\right) \right.\right. \\ & \left.\left. +3 c_3 \left(c_3+1\right) \left(\left(\alpha_0+\alpha_1+\alpha_2+1\right) c_3^4-1\right)\right) +\alpha_1 \left(c_3^2+c_3+c_2\right) c_1^5 -\alpha_1 c_2 \left(2 c_2-\left(c_3-2\right) c_3\right) c_1^4-2 \alpha_1 c_2 \left(4 c_2-1\right) c_3^2 c_1^3 \right. \\ & \left. +c_2 c_3^2\left(\alpha_2 c_2^4+2 \alpha_2 c_2^3-2 c_2 \left(2 \left(\alpha_0+\alpha_1+\alpha_2+1\right) c_3^3+3 \left(\alpha_0+\alpha_1+\alpha_2+1\right) c_3^2-1\right)+3 \left(\alpha_0+\alpha_1+\alpha_2+1\right) c_3^4+4 \left(\alpha_0+\alpha_1+\alpha_2+1\right) c_3^3+1\right)\right] \\ & /\left(12 c_1 \left(c_1+1\right) \left(c_1-c_2\right) \left(c_1-c_3\right) \left(c_2 c_3+c_1 \left(c_2+c_3\right)\right)\right) \\
\hline
 \alpha_2 & -345107661/454781887 \\
 \beta_2 & \left[-4 c_1^2 \left(\alpha_2 c_2^3-3 \alpha_2 c_3^2 c_2+2 \left(\alpha_0+\alpha_1+\alpha_2+1\right) c_3^3-3 \left(\beta_0-1\right) c_3^2-1\right)-\alpha_1 c_1^5-\alpha_1 c_3 c_1^4+8 \alpha_1 c_3^2 c_1^3+c_1 \left(3 \alpha_2 c_2^4-4 \alpha_2 c_3 c_2^3+\left(\alpha_0+\alpha_1+\alpha_2+1\right) c_3^4+4 c_3+3\right) \right. \\ & \left. +c_3 \left(3 \alpha_2 c_2^4-4 \alpha_2 c_3 c_2^3+\left(\alpha_0+\alpha_1+\alpha_2+1\right) c_3^4+4 c_3+3\right)\right]/\left(12 \left(c_1-c_2\right) \left(c_2-c_3\right) \left(c_2 c_3+c_1 \left(c_2+c_3\right)\right)\right) \\
 \mu_2 & \left[-2 c_1^2 \left(-3 c_2^2 \left(c_3^2 \left(\alpha_0+\alpha_1-2 \beta_0+3\right)+2 \left(\alpha_0+\alpha_1+\alpha_2+1\right) c_3^3-1\right)-2 c_2 \left(2 \left(\alpha_0+\alpha_1+\alpha_2+1\right) c_3^3-3 \left(\beta_0-1\right) c_3^2-1\right)+\alpha_2 c_2^4+4 \alpha_2 c_3^2 c_2^3 \right.\right. \\ & \left.\left. +c_3^2 \left(2 \left(\alpha_0+\alpha_1+\alpha_2+1\right) c_3^3+3 \left(\alpha_0+\alpha_1+\alpha_2+1\right) c_3^2-1\right)\right) +\alpha_1 \left(c_3^2+c_2\right) c_1^5+\alpha_1 \left(-2 c_2^2+c_3 \left(c_3+1\right) c_2+2 c_3^2\right) c_1^4-2 \alpha_1 c_2 c_3 \left(3 c_3+c_2 \left(4 c_3+1\right)\right) c_1^3 \right. \\ & \left. +c_1 \left(3 \left(\alpha_0+\alpha_1+\alpha_2+1\right) c_3^6+4 \left(\alpha_0+\alpha_1+\alpha_2+1\right) c_3^5+2 \alpha_2 c_2^3 c_3^2+\alpha_2 c_2^4 \left(c_3-2\right) c_3+\alpha_2 c_2^5 +2 c_2^2 \left(\left(\alpha_0+\alpha_1+\alpha_2+1\right) c_3^3-3 c_3-2\right)-c_2 \left(4 \left(\alpha_0+\alpha_1+\alpha_2+1\right) c_3^5+7 \left(\alpha_0+\alpha_1+\alpha_2+1\right) c_3^4 \right.\right.\right. \\ & \left.\left.\left. -2 c_3^2+4 c_3+3\right)+c_3^2\right) +c_2 c_3 \left(c_3+1\right) \left(\alpha_2 c_2^4-4 c_2 \left(\left(\alpha_0+\alpha_1+\alpha_2+1\right) c_3^3+1\right)+3 \left(\alpha_0+\alpha_1+\alpha_2+1\right) c_3^4-3\right)\right]/\left(12 \left(c_1-c_2\right) c_2 \left(c_2+1\right) \left(c_2-c_3\right) \left(c_2 c_3+c_1 \left(c_2+c_3\right)\right)\right) \\
\hline
 \alpha_3 & 50927106883029008210353/518631772039236867838813 \\
 \beta_3 & \left[4 c_1^2 \left(3 c_2^2 \left(\beta_0+\left(\alpha_0+\alpha_1+\alpha_2+1\right) \left(-c_3\right)-1\right)+2 \alpha_2 c_2^3+\left(\alpha_0+\alpha_1+\alpha_2+1\right) c_3^3+1\right)-\alpha_1 c_1^5-\alpha_1 c_2 c_1^4+8 \alpha_1 c_2^2 c_1^3+c_1 \left(-\alpha_2 c_2^4+4 c_2 \left(\left(\alpha_0+\alpha_1+\alpha_2+1\right) c_3^3+1\right)-3 \left(\alpha_0+\alpha_1+\alpha_2+1\right) c_3^4+3\right) \right. \\ & \left. +c_2 \left(-\alpha_2 c_2^4+4 c_2 \left(\left(\alpha_0+\alpha_1+\alpha_2+1\right) c_3^3+1\right)-3 \left(\alpha_0+\alpha_1+\alpha_2+1\right) c_3^4+3\right)\right]/\left(12 \left(c_1-c_3\right) \left(c_3-c_2\right) \left(c_2 c_3+c_1 \left(c_2+c_3\right)\right)\right) \\
 \mu_3 & \left[-4 c_1^2 \left(3 c_2^2 \left(\beta_0+\left(\alpha_0+\alpha_1+\alpha_2+1\right) \left(-c_3\right)-1\right)+2 \alpha_2 c_2^3+\left(\alpha_0+\alpha_1+\alpha_2+1\right) c_3^3+1\right)+\alpha_1 c_1^5+\alpha_1 c_2 c_1^4-8 \alpha_1 c_2^2 c_1^3+c_1 \left(\alpha_2 c_2^4-4 c_2 \left(\left(\alpha_0+\alpha_1+\alpha_2+1\right) c_3^3+1\right)+3 \left(\alpha_0+\alpha_1+\alpha_2+1\right) c_3^4-3\right) \right. \\ & \left. +c_2 \left(\alpha_2 c_2^4-4 c_2 \left(\left(\alpha_0+\alpha_1+\alpha_2+1\right) c_3^3+1\right)+3 \left(\alpha_0+\alpha_1+\alpha_2+1\right) c_3^4-3\right)\right]/\left(12 \left(c_1-c_3\right) \left(c_3-c_2\right) \left(c_2 c_3+c_1 \left(c_2+c_3\right)\right)\right) \\
\end{array}
$
\end{turn}
}
\end{table}

\begin{table}[H]
\centering
\caption{{\sc Limm} 5-step order 5 variable stepsize coefficients (part 1).}
\label{tbl:limmvcoef5a}
{\fontsize{3pt}{4pt}\selectfont
\renewcommand{\arraystretch}{2}
\begin{turn}{-90}
$\begin{array}{c|l}
\hline
 \alpha_{-1} & 1 \\
 \beta_{-1} & 0 \\
 \mu_{-1} & \left[3 \alpha_1 c_1^5-5 \alpha_1 \left(c_2+c_3\right) c_1^4+10 \alpha_1 c_2 c_3 c_1^3-5 c_1 \left(\alpha_2 c_2^4-2 \alpha_2 c_3 c_2^3-2 c_2 \left(\alpha_3 c_3^3-3 c_3 \left(\left(\alpha_0+\alpha_1+\alpha_2+\alpha_3+1\right) c_4^2-1\right)+2 \left(\alpha_0+\alpha_1+\alpha_2+\alpha_3+1\right) c_4^3+2\right)+\alpha_3 c_3^4+3 \left(\alpha_0+\alpha_1+\alpha_2+\alpha_3+1\right) c_4^4 \right.\right. \\ & \left.\left. -4 c_3 \left(\left(\alpha_0+\alpha_1+\alpha_2+\alpha_3+1\right) c_4^3+1\right)-3\right)+3 \alpha_2 c_2^5-5 \alpha_2 c_2^4 c_3-5 c_2 \left(\alpha_3 c_3^4-4 c_3 \left(\left(\alpha_0+\alpha_1+\alpha_2+\alpha_3+1\right) c_4^3+1\right)+3 \left(\alpha_0+\alpha_1+\alpha_2+\alpha_3+1\right) c_4^4-3\right) \right. \\ & \left. +3 \left(\alpha_3 c_3^5-5 c_3 \left(\left(\alpha_0+\alpha_1+\alpha_2+\alpha_3+1\right) c_4^4-1\right)+4 \left(\alpha_0+\alpha_1+\alpha_2+\alpha_3+1\right) c_4^5+4\right)\right]/\left(60 \left(c_1+1\right) \left(c_2+1\right) \left(c_3+1\right)\right) \\
\hline
 \alpha_0 & -104367911/41202283 \\
 \beta_0 & 60/137 \\
 \mu_0 & \left[-3 \alpha_1 c_1^5+5 \left(c_2+c_3-1\right) \alpha_1 c_1^4-10 \left(c_2 \left(c_3-1\right)-c_3\right) \alpha_1 c_1^3-30 c_2 c_3 \alpha_1 c_1^2+5 \left(\alpha_2 c_2^4-2 \left(c_3-1\right) \alpha_2 c_2^3-6 c_3 \alpha_2 c_2^2-2 \left(\alpha_3 c_3^3+3 \alpha_3 c_3^2-3 \left(\left(\alpha_0+\alpha_1+\alpha_2+\alpha_3+1\right) c_4^2 \right.\right.\right.\right. \\ & \left.\left.\left.\left. +2 \left(\alpha_0+\alpha_1+\alpha_2+\alpha_3+1\right) c_4-2 \beta_0+1\right) c_3+2 c_4^3 \left(\alpha_0+\alpha_1+\alpha_2+\alpha_3+1\right)+3 c_4^2 \left(\alpha_0+\alpha_1+\alpha_2+\alpha_3+1\right)-1\right) c_2+3 c_4^4+4 c_4^3+3 c_4^4 \alpha_0+4 c_4^3 \alpha_0+3 c_4^4 \alpha_1+4 c_4^3 \alpha_1+3 c_4^4 \alpha_2+4 c_4^3 \alpha_2+c_3^4 \alpha_3 \right.\right. \\ & \left.\left. +3 c_4^4 \alpha_3+2 c_3^3 \alpha_3+4 c_4^3 \alpha_3-2 c_3 \left(2 \left(\alpha_0+\alpha_1+\alpha_2+\alpha_3+1\right) c_4^3+3 \left(\alpha_0+\alpha_1+\alpha_2+\alpha_3+1\right) c_4^2-1\right)+1\right) c_1-12 c_4^5+15 c_3 c_4^4-15 c_4^4+20 c_3 c_4^3+5 c_3-12 c_4^5 \alpha_0+15 c_3 c_4^4 \alpha_0-15 c_4^4 \alpha_0+20 c_3 c_4^3 \alpha_0 \right. \\ & \left. -12 c_4^5 \alpha_1+15 c_3 c_4^4 \alpha_1-15 c_4^4 \alpha_1+20 c_3 c_4^3 \alpha_1-3 c_2^5 \alpha_2-12 c_4^5 \alpha_2+15 c_3 c_4^4 \alpha_2-15 c_4^4 \alpha_2+20 c_3 c_4^3 \alpha_2+5 c_2^4 \left(c_3-1\right) \alpha_2+10 c_2^3 c_3 \alpha_2-3 c_3^5 \alpha_3-12 c_4^5 \alpha_3-5 c_3^4 \alpha_3+15 c_3 c_4^4 \alpha_3-15 c_4^4 \alpha_3+20 c_3 c_4^3 \alpha_3 \right. \\ & \left.+5 c_2 \left(\alpha_3 c_3^4+2 \alpha_3 c_3^3-2 \left(2 \left(\alpha_0+\alpha_1+\alpha_2+\alpha_3+1\right) c_4^3+3 \left(\alpha_0+\alpha_1+\alpha_2+\alpha_3+1\right) c_4^2-1\right) c_3+3 c_4^4 \left(\alpha_0+\alpha_1+\alpha_2+\alpha_3+1\right)+4 c_4^3 \left(\alpha_0+\alpha_1+\alpha_2+\alpha_3+1\right)+1\right)+3\right]/\left(60 c_1 c_2 c_3\right) \\
\hline
 \alpha_1 & 59680231/21017185 \\
 \beta_1 & \left[-5 c_2^2 \left(4 c_3^2 \left(\alpha_1 c_1^3-3 \alpha_1 c_4^2 c_1+2 \left(\alpha_0+\alpha_1+\alpha_2+\alpha_3+1\right) c_4^3-3 \left(\beta_0-1\right) c_4^2-1\right)+\alpha_3 c_3^5+\alpha_3 c_4 c_3^4-8 \alpha_3 c_4^2 c_3^3-c_3 \left(3 \alpha_1 c_1^4-4 \alpha_1 c_4 c_1^3+\left(\alpha_0+\alpha_1+\alpha_2+\alpha_3+1\right) c_4^4+4 c_4+3\right) \right.\right. \\ & \left.\left. -c_4 \left(3 \alpha_1 c_1^4-4 \alpha_1 c_4 c_1^3+\left(\alpha_0+\alpha_1+\alpha_2+\alpha_3+1\right) c_4^4+4 c_4+3\right)\right)+3 \alpha_2 \left(c_3+c_4\right) c_2^6-\alpha_2 \left(5 c_3^2+2 c_4 c_3+5 c_4^2\right) c_2^5-5 \alpha_2 c_3 c_4 \left(c_3+c_4\right) c_2^4+40 \alpha_2 c_3^2 c_4^2 c_2^3+\left(c_3+c_4\right) c_2 \left(-12 \alpha_1 c_1^5 \right.\right. \\ & \left.\left. +15 \alpha_1 \left(c_3+c_4\right) c_1^4-20 \alpha_1 c_3 c_4 c_1^3+3 \alpha_3 c_3^5-5 \alpha_3 c_3^4 c_4-3 \left(\alpha_0+\alpha_1+\alpha_2+\alpha_3+1\right) c_4^5+5 c_3 \left(\left(\alpha_0+\alpha_1+\alpha_2+\alpha_3+1\right) c_4^4+4 c_4+3\right)+15 c_4+12\right)+c_3 c_4 \left(-12 \alpha_1 c_1^5+15 \alpha_1 \left(c_3+c_4\right) c_1^4 \right.\right. \\ & \left.\left. -20 \alpha_1 c_3 c_4 c_1^3+3 \alpha_3 c_3^5-5 \alpha_3 c_3^4 c_4-3 \left(\alpha_0+\alpha_1+\alpha_2+\alpha_3+1\right) c_4^5+5 c_3 \left(\left(\alpha_0+\alpha_1+\alpha_2+\alpha_3+1\right) c_4^4+4 c_4+3\right)+15 c_4+12\right)\right] \\ & /\left(60 \left(c_1-c_2\right) \left(c_1-c_3\right) \left(c_1-c_4\right) \left(c_2 c_3 c_4+c_1 \left(c_3 c_4+c_2 \left(c_3+c_4\right)\right)\right)\right) \\
 \mu_1 & \left[-3 \left(c_3 c_4 \left(c_4+1\right)+c_2 \left(c_4^2+c_4+c_3\right)\right) \alpha_1 c_1^6+\left(5 \left(c_4^2+c_4+c_3\right) c_2^2+c_3 \left(5 c_3+2 c_4 \left(c_4+5\right)\right) c_2+5 c_3^2 c_4 \left(c_4+1\right)\right) \alpha_1 c_1^5-5 c_2 c_3 \left(c_4 \left(c_4-c_3 \left(c_4-2\right)\right)+c_2 \left(2 c_3-\left(c_4-2\right) c_4\right)\right) \alpha_1 c_1^4 \right. \\ & \left. +10 c_2 c_3 \left(c_2 \left(1-4 c_3\right)+c_3\right) c_4^2 \alpha_1 c_1^3+5 \left(\left(c_4^2+c_4+c_3\right) \alpha_2 c_2^5+c_3 \left(\left(c_4-1\right) c_4-2 c_3\right) \alpha_2 c_2^4-2 c_3^2 c_4 \left(4 c_4+1\right) \alpha_2 c_2^3-2 \left(\alpha_3 c_3^4+c_4 \left(4 c_4+1\right) \alpha_3 c_3^3-3 \left(2 \left(\alpha_0+\alpha_1+\alpha_2+\alpha_3+1\right) c_4^3 \right.\right.\right.\right. \\ & \left.\left.\left.\left. +\left(\alpha_0+\alpha_2+\alpha_3-2 \beta_0+3\right) c_4^2-1\right) c_3^2-\left(\left(\alpha_0+\alpha_1+\alpha_2+\alpha_3+1\right) c_4^3-3 c_4-2\right) c_3+2 c_4 \left(c_4+1\right) \left(\left(\alpha_0+\alpha_1+\alpha_2+\alpha_3+1\right) c_4^3+1\right)\right) c_2^2+\left(\alpha_3 c_3^5+\left(c_4-1\right) c_4 \alpha_3 c_3^4 \right.\right.\right. \\ & \left.\left.\left. +2 \left(\left(\alpha_0+\alpha_1+\alpha_2+\alpha_3+1\right) c_4^3-3 c_4-2\right) c_3^2-\left(4 \left(\alpha_0+\alpha_1+\alpha_2+\alpha_3+1\right) c_4^5+5 \left(\alpha_0+\alpha_1+\alpha_2+\alpha_3+1\right) c_4^4+4 c_4^2+8 c_4+3\right) c_3+3 c_4 \left(c_4+1\right) \left(c_4^4 \left(\alpha_0+\alpha_1+\alpha_2+\alpha_3+1\right)-1\right)\right) c_2 \right.\right. \\ & \left.\left. +c_3 c_4 \left(c_4+1\right) \left(\alpha_3 c_3^4-4 \left(\left(\alpha_0+\alpha_1+\alpha_2+\alpha_3+1\right) c_4^3+1\right) c_3+3 c_4^4 \left(\alpha_0+\alpha_1+\alpha_2+\alpha_3+1\right)-3\right)\right) c_1^2-\left(3 \left(c_4^2+c_4+c_3\right) \alpha_2 c_2^6-c_3 \left(5 c_3+2 c_4 \left(c_4+1\right)\right) \alpha_2 c_2^5-5 c_3 c_4 \left(2 c_4+c_3 \left(c_4+1\right)\right) \alpha_2 c_2^4 \right.\right. \\ & \left.\left. +30 c_3^2 c_4^2 \alpha_2 c_2^3-5 \left(\alpha_3 c_3^5+c_4 \left(c_4+1\right) \alpha_3 c_3^4-6 c_4^2 \alpha_3 c_3^3+4 \left(2 \left(\alpha_0+\alpha_1+\alpha_2+\alpha_3+1\right) c_4^3-3 \left(\beta_0-1\right) c_4^2-1\right) c_3^2-\left(4 \left(\alpha_0+\alpha_1+\alpha_2+\alpha_3+1\right) c_4^5+7 \left(\alpha_0+\alpha_1+\alpha_2+\alpha_3+1\right) c_4^4-2 c_4^2+4 c_4 \right.\right.\right.\right. \\ & \left.\left.\left.\left. +3\right) c_3+3 c_4 \left(c_4+1\right) \left(c_4^4 \left(\alpha_0+\alpha_1+\alpha_2+\alpha_3+1\right)-1\right)\right) c_2^2+\left(3 \alpha_3 c_3^6-2 c_4 \left(c_4+1\right) \alpha_3 c_3^5-10 c_4^2 \alpha_3 c_3^4+5 \left(4 \left(\alpha_0+\alpha_1+\alpha_2+\alpha_3+1\right) c_4^5+7 \left(\alpha_0+\alpha_1+\alpha_2+\alpha_3+1\right) c_4^4-2 c_4^2+4 c_4+3\right) c_3^2 \right.\right.\right. \\ & \left.\left.\left. -2 \left(15 \left(\alpha_0+\alpha_1+\alpha_2+\alpha_3+1\right) c_4^6+19 \left(\alpha_0+\alpha_1+\alpha_2+\alpha_3+1\right) c_4^5-5 c_4^2-15 c_4-6\right) c_3+12 c_4 \left(c_4+1\right) \left(\left(\alpha_0+\alpha_1+\alpha_2+\alpha_3+1\right) c_4^5+1\right)\right) c_2+3 c_3 c_4 \left(c_4+1\right) \left(\alpha_3 c_3^5 \right.\right.\right. \\ & \left.\left.\left. -5 \left(c_4^4 \left(\alpha_0+\alpha_1+\alpha_2+\alpha_3+1\right)-1\right) c_3+4 c_4^5 \left(\alpha_0+\alpha_1+\alpha_2+\alpha_3+1\right)+4\right)\right) c_1+c_2 c_3 c_4^2 \left(-3 \alpha_2 c_2^5+5 \left(c_3-1\right) \alpha_2 c_2^4+10 c_3 \alpha_2 c_2^3+5 \left(\alpha_3 c_3^4+2 \alpha_3 c_3^3-2 \left(2 \left(\alpha_0+\alpha_1+\alpha_2+\alpha_3+1\right) c_4^3 \right.\right.\right.\right. \\ & \left.\left.\left.\left. +3 \left(\alpha_0+\alpha_1+\alpha_2+\alpha_3+1\right) c_4^2-1\right) c_3+3 c_4^4 \left(\alpha_0+\alpha_1+\alpha_2+\alpha_3+1\right)+4 c_4^3 \left(\alpha_0+\alpha_1+\alpha_2+\alpha_3+1\right)+1\right) c_2-3 c_3^5 \alpha_3-5 c_3^4 \alpha_3+5 c_3 \left(3 \left(\alpha_0+\alpha_1+\alpha_2+\alpha_3+1\right) c_4^4 \right.\right.\right. \\ & \left.\left.\left. +4 \left(\alpha_0+\alpha_1+\alpha_2+\alpha_3+1\right) c_4^3+1\right)-3 \left(4 \left(\alpha_0+\alpha_1+\alpha_2+\alpha_3+1\right) c_4^5+5 \left(\alpha_0+\alpha_1+\alpha_2+\alpha_3+1\right) c_4^4-1\right)\right)\right]/\left(60 c_1 \left(c_1+1\right) \left(c_1-c_2\right) \left(c_1-c_3\right) \left(c_1-c_4\right) \left(c_2 c_3 c_4+c_1 \left(c_3 c_4+c_2 \left(c_3+c_4\right)\right)\right)\right) \\
\hline
 \alpha_2 & -97736124/57440479 \\
 \beta_2 & -\left[-5 c_1^2 \left(4 c_3^2 \left(\alpha_2 c_2^3-3 \alpha_2 c_4^2 c_2+2 \left(\alpha_0+\alpha_1+\alpha_2+\alpha_3+1\right) c_4^3-3 \left(\beta_0-1\right) c_4^2-1\right)+\alpha_3 c_3^5+\alpha_3 c_4 c_3^4-8 \alpha_3 c_4^2 c_3^3-c_3 \left(3 \alpha_2 c_2^4-4 \alpha_2 c_4 c_2^3+\left(\alpha_0+\alpha_1+\alpha_2+\alpha_3+1\right) c_4^4+4 c_4+3\right) \right.\right. \\ & \left.\left. -c_4 \left(3 \alpha_2 c_2^4-4 \alpha_2 c_4 c_2^3+\left(\alpha_0+\alpha_1+\alpha_2+\alpha_3+1\right) c_4^4+4 c_4+3\right)\right)+3 \alpha_1 \left(c_3+c_4\right) c_1^6-\alpha_1 \left(5 c_3^2+2 c_4 c_3+5 c_4^2\right) c_1^5-5 \alpha_1 c_3 c_4 \left(c_3+c_4\right) c_1^4+40 \alpha_1 c_3^2 c_4^2 c_1^3+\left(c_3+c_4\right) c_1 \left(-12 \alpha_2 c_2^5 \right.\right. \\ & \left.\left. +15 \alpha_2 \left(c_3+c_4\right) c_2^4-20 \alpha_2 c_3 c_4 c_2^3+3 \alpha_3 c_3^5 -5 \alpha_3 c_3^4 c_4-3 \left(\alpha_0+\alpha_1+\alpha_2+\alpha_3+1\right) c_4^5+5 c_3 \left(\left(\alpha_0+\alpha_1+\alpha_2+\alpha_3+1\right) c_4^4+4 c_4+3\right)+15 c_4+12\right)+c_3 c_4 \left(-12 \alpha_2 c_2^5+15 \alpha_2 \left(c_3+c_4\right) c_2^4 \right.\right. \\ & \left.\left. -20 \alpha_2 c_3 c_4 c_2^3+3 \alpha_3 c_3^5-5 \alpha_3 c_3^4 c_4-3 \left(\alpha_0+\alpha_1+\alpha_2+\alpha_3+1\right) c_4^5+5 c_3 \left(\left(\alpha_0+\alpha_1+\alpha_2+\alpha_3+1\right) c_4^4+4 c_4+3\right)+15 c_4+12\right)\right] \\ & /\left(60 \left(c_1-c_2\right) \left(c_2-c_3\right) \left(c_2-c_4\right) \left(c_2 c_3 c_4+c_1 \left(c_3 c_4+c_2 \left(c_3+c_4\right)\right)\right)\right) \\
 \mu_2 & \left[3 \left(c_3 c_4^2+c_2 \left(c_4^2+c_4+c_3\right)\right) \alpha_1 c_1^6-\left(5 \left(c_4^2+c_4+c_3\right) c_2^2+c_3 \left(5 c_3+2 c_4 \left(c_4+1\right)\right) c_2+5 \left(c_3-1\right) c_3 c_4^2\right) \alpha_1 c_1^5+5 c_3 \left(\left(-c_4^2+c_4+2 c_3\right) c_2^2-c_4 \left(2 c_4+c_3 \left(c_4+1\right)\right) c_2-2 c_3 c_4^2\right) \alpha_1 c_1^4 \right. \\ & \left. +10 c_2 c_3^2 c_4 \left(3 c_4+c_2 \left(4 c_4+1\right)\right) \alpha_1 c_1^3-5 \left(\left(c_4^2+c_4+c_3\right) \alpha_2 c_2^5+c_3 \left(\left(c_4-2\right) c_4-2 c_3\right) \alpha_2 c_2^4+2 \left(1-4 c_3\right) c_3 c_4^2 \alpha_2 c_2^3-2 \left(\alpha_3 c_3^4+c_4 \left(4 c_4+1\right) \alpha_3 c_3^3-3 \left(2 \left(\alpha_0+\alpha_1+\alpha_2+\alpha_3+1\right) c_4^3 \right.\right.\right.\right. \\ & \left.\left.\left.\left. +\left(\alpha_0+\alpha_1+\alpha_3-2 \beta_0+3\right) c_4^2-1\right) c_3^2-\left(\left(\alpha_0+\alpha_1+\alpha_2+\alpha_3+1\right) c_4^3-3 c_4-2\right) c_3+2 c_4 \left(c_4+1\right) \left(\left(\alpha_0+\alpha_1+\alpha_2+\alpha_3+1\right) c_4^3+1\right)\right) c_2^2+\left(\alpha_3 c_3^5+c_4 \left(c_4+1\right) \alpha_3 c_3^4-6 c_4^2 \alpha_3 c_3^3 \right.\right.\right. \\ & \left.\left.\left. +4 \left(2 \left(\alpha_0+\alpha_1+\alpha_2+\alpha_3+1\right) c_4^3-3 \left(\beta_0-1\right) c_4^2-1\right) c_3^2-\left(4 \left(\alpha_0+\alpha_1+\alpha_2+\alpha_3+1\right) c_4^5+7 \left(\alpha_0+\alpha_1+\alpha_2+\alpha_3+1\right) c_4^4-2 c_4^2+4 c_4+3\right) c_3+3 c_4 \left(c_4+1\right) \left(c_4^4 \left(\alpha_0+\alpha_1+\alpha_2+\alpha_3+1\right)-1\right)\right) c_2 \right.\right. \\ & \left.\left. +c_3 c_4^2 \left(\alpha_3 c_3^4+2 \alpha_3 c_3^3-2 \left(2 \left(\alpha_0+\alpha_1+\alpha_2+\alpha_3+1\right) c_4^3+3 \left(\alpha_0+\alpha_1+\alpha_2+\alpha_3+1\right) c_4^2-1\right) c_3+3 c_4^4 \left(\alpha_0+\alpha_1+\alpha_2+\alpha_3+1\right)+4 c_4^3 \left(\alpha_0+\alpha_1+\alpha_2+\alpha_3+1\right)+1\right)\right) c_1^2+\left(3 \left(c_4^2+c_4+c_3\right) \alpha_2 c_2^6 \right.\right. \\ & \left.\left. -c_3 \left(5 c_3+2 c_4 \left(c_4+5\right)\right) \alpha_2 c_2^5+5 c_3 c_4 \left(c_4-c_3 \left(c_4-2\right)\right) \alpha_2 c_2^4-10 c_3^2 c_4^2 \alpha_2 c_2^3-5 \left(\alpha_3 c_3^5+\left(c_4-1\right) c_4 \alpha_3 c_3^4+2 \left(\left(\alpha_0+\alpha_1+\alpha_2+\alpha_3+1\right) c_4^3-3 c_4-2\right) c_3^2-\left(4 \left(\alpha_0+\alpha_1+\alpha_2+\alpha_3+1\right) c_4^5 \right.\right.\right.\right. \\ & \left.\left.\left.\left. +5 \left(\alpha_0+\alpha_1+\alpha_2+\alpha_3+1\right) c_4^4+4 c_4^2+8 c_4+3\right) c_3+3 c_4 \left(c_4+1\right) \left(c_4^4 \left(\alpha_0+\alpha_1+\alpha_2+\alpha_3+1\right)-1\right)\right) c_2^2+\left(3 \alpha_3 c_3^6-2 c_4 \left(c_4+1\right) \alpha_3 c_3^5-10 c_4^2 \alpha_3 c_3^4+5 \left(4 \left(\alpha_0+\alpha_1+\alpha_2+\alpha_3+1\right) c_4^5 \right.\right.\right.\right. \\ & \left.\left.\left.\left. +7 \left(\alpha_0+\alpha_1+\alpha_2+\alpha_3+1\right) c_4^4-2 c_4^2+4 c_4+3\right) c_3^2-2 \left(15 \left(\alpha_0+\alpha_1+\alpha_2+\alpha_3+1\right) c_4^6+19 \left(\alpha_0+\alpha_1+\alpha_2+\alpha_3+1\right) c_4^5-5 c_4^2-15 c_4-6\right) c_3+12 c_4 \left(c_4+1\right) \left(\left(\alpha_0+\alpha_1+\alpha_2+\alpha_3+1\right) c_4^5+1\right)\right) c_2 \right.\right. \\ & \left.\left. +c_3 c_4^2 \left(3 \alpha_3 c_3^5+5 \alpha_3 c_3^4-5 \left(3 \left(\alpha_0+\alpha_1+\alpha_2+\alpha_3+1\right) c_4^4+4 \left(\alpha_0+\alpha_1+\alpha_2+\alpha_3+1\right) c_4^3+1\right) c_3+3 \left(4 \left(\alpha_0+\alpha_1+\alpha_2+\alpha_3+1\right) c_4^5+5 \left(\alpha_0+\alpha_1+\alpha_2+\alpha_3+1\right) c_4^4-1\right)\right)\right) c_1 \right. \\ & \left. +c_2 c_3 c_4 \left(c_4+1\right) \left(3 \alpha_2 c_2^5-5 c_3 \alpha_2 c_2^4-5 \left(\alpha_3 c_3^4-4 \left(\left(\alpha_0+\alpha_1+\alpha_2+\alpha_3+1\right) c_4^3+1\right) c_3+3 c_4^4 \left(\alpha_0+\alpha_1+\alpha_2+\alpha_3+1\right)-3\right) c_2+3 \left(\alpha_3 c_3^5-5 \left(c_4^4 \left(\alpha_0+\alpha_1+\alpha_2+\alpha_3+1\right)-1\right) c_3 \right.\right.\right. \\ & \left.\left.\left. +4 c_4^5 \left(\alpha_0+\alpha_1+\alpha_2+\alpha_3+1\right)+4\right)\right)\right]/\left(60 \left(c_1-c_2\right) c_2 \left(c_2+1\right) \left(c_2-c_3\right) \left(c_2-c_4\right) \left(c_2 c_3 c_4+c_1 \left(c_3 c_4+c_2 \left(c_3+c_4\right)\right)\right)\right) \\
\end{array}
$
\end{turn}
}
\end{table}

\begin{table}[H]
\centering
\caption{{\sc Limm} 5-step order 5 variable stepsize coefficients (part 2).}
\label{tbl:limmvcoef5b}
{\fontsize{3pt}{4pt}\selectfont
\renewcommand{\arraystretch}{2}
\begin{turn}{-90}
$\begin{array}{c|l}
\hline
 \alpha_4 & -188732392210474496577705869057/1979785468648998861857945444345 \\
 \beta_4 & -\left[-5 c_1^2 \left(-4 c_2^2 \left(3 c_3^2 \left(\beta_0+\left(\alpha_0+\alpha_1+\alpha_2+\alpha_3+1\right) \left(-c_4\right)-1\right)+2 \alpha_3 c_3^3+\left(\alpha_0+\alpha_1+\alpha_2+\alpha_3+1\right) c_4^3+1\right)+\alpha_2 c_2^5+\alpha_2 c_3 c_2^4-8 \alpha_2 c_3^2 c_2^3+c_2 \left(\alpha_3 c_3^4-4 c_3 \left(\left(\alpha_0+\alpha_1+\alpha_2+\alpha_3+1\right) c_4^3+1\right) \right.\right.\right. \\ & \left.\left.\left. +3 \left(\alpha_0+\alpha_1+\alpha_2+\alpha_3+1\right) c_4^4-3\right)+c_3 \left(\alpha_3 c_3^4-4 c_3 \left(\left(\alpha_0+\alpha_1+\alpha_2+\alpha_3+1\right) c_4^3+1\right)+3 \left(\alpha_0+\alpha_1+\alpha_2+\alpha_3+1\right) c_4^4-3\right)\right)+3 \alpha_1 \left(c_2+c_3\right) c_1^6-\alpha_1 \left(5 c_2^2+2 c_3 c_2+5 c_3^2\right) c_1^5 \right. \\ & \left. -5 \alpha_1 c_2 c_3 \left(c_2+c_3\right) c_1^4+40 \alpha_1 c_2^2 c_3^2 c_1^3+\left(c_2+c_3\right) c_1 \left(3 \alpha_2 c_2^5-5 \alpha_2 c_3 c_2^4-5 c_2 \left(\alpha_3 c_3^4-4 c_3 \left(\left(\alpha_0+\alpha_1+\alpha_2+\alpha_3+1\right) c_4^3+1\right)+3 \left(\alpha_0+\alpha_1+\alpha_2+\alpha_3+1\right) c_4^4-3\right) \right.\right. \\ & \left.\left. +3 \left(\alpha_3 c_3^5-5 c_3 \left(\left(\alpha_0+\alpha_1+\alpha_2+\alpha_3+1\right) c_4^4-1\right)+4 \left(\alpha_0+\alpha_1+\alpha_2+\alpha_3+1\right) c_4^5+4\right)\right)+c_2 c_3 \left(3 \alpha_2 c_2^5-5 \alpha_2 c_3 c_2^4-5 c_2 \left(\alpha_3 c_3^4-4 c_3 \left(\left(\alpha_0+\alpha_1+\alpha_2+\alpha_3+1\right) c_4^3+1\right) \right.\right.\right. \\ & \left.\left.\left. +3 \left(\alpha_0+\alpha_1+\alpha_2+\alpha_3+1\right) c_4^4-3\right)+3 \left(\alpha_3 c_3^5-5 c_3 \left(\left(\alpha_0+\alpha_1+\alpha_2+\alpha_3+1\right) c_4^4-1\right)+4 \left(\alpha_0+\alpha_1+\alpha_2+\alpha_3+1\right) c_4^5+4\right)\right)\right] \\ & /\left(60 \left(c_2-c_4\right) \left(c_4-c_1\right) \left(c_4-c_3\right) \left(c_2 c_3 c_4+c_1 \left(c_3 c_4+c_2 \left(c_3+c_4\right)\right)\right)\right) \\
 \mu_4 & \left[-5 c_1^2 \left(-4 c_2^2 \left(3 c_3^2 \left(\beta_0+\left(\alpha_0+\alpha_1+\alpha_2+\alpha_3+1\right) \left(-c_4\right)-1\right)+2 \alpha_3 c_3^3+\left(\alpha_0+\alpha_1+\alpha_2+\alpha_3+1\right) c_4^3+1\right)+\alpha_2 c_2^5+\alpha_2 c_3 c_2^4-8 \alpha_2 c_3^2 c_2^3+c_2 \left(\alpha_3 c_3^4-4 c_3 \left(\left(\alpha_0+\alpha_1+\alpha_2+\alpha_3+1\right) c_4^3+1\right) \right.\right.\right. \\ & \left.\left.\left. +3 \left(\alpha_0+\alpha_1+\alpha_2+\alpha_3+1\right) c_4^4-3\right)+c_3 \left(\alpha_3 c_3^4-4 c_3 \left(\left(\alpha_0+\alpha_1+\alpha_2+\alpha_3+1\right) c_4^3+1\right)+3 \left(\alpha_0+\alpha_1+\alpha_2+\alpha_3+1\right) c_4^4-3\right)\right)+3 \alpha_1 \left(c_2+c_3\right) c_1^6-\alpha_1 \left(5 c_2^2+2 c_3 c_2+5 c_3^2\right) c_1^5 \right. \\ & \left. -5 \alpha_1 c_2 c_3 \left(c_2+c_3\right) c_1^4+40 \alpha_1 c_2^2 c_3^2 c_1^3+\left(c_2+c_3\right) c_1 \left(3 \alpha_2 c_2^5-5 \alpha_2 c_3 c_2^4-5 c_2 \left(\alpha_3 c_3^4-4 c_3 \left(\left(\alpha_0+\alpha_1+\alpha_2+\alpha_3+1\right) c_4^3+1\right)+3 \left(\alpha_0+\alpha_1+\alpha_2+\alpha_3+1\right) c_4^4-3\right) \right.\right. \\ & \left.\left. +3 \left(\alpha_3 c_3^5-5 c_3 \left(\left(\alpha_0+\alpha_1+\alpha_2+\alpha_3+1\right) c_4^4-1\right)+4 \left(\alpha_0+\alpha_1+\alpha_2+\alpha_3+1\right) c_4^5+4\right)\right)+c_2 c_3 \left(3 \alpha_2 c_2^5-5 \alpha_2 c_3 c_2^4-5 c_2 \left(\alpha_3 c_3^4-4 c_3 \left(\left(\alpha_0+\alpha_1+\alpha_2+\alpha_3+1\right) c_4^3+1\right) \right.\right.\right. \\ & \left.\left.\left. +3 \left(\alpha_0+\alpha_1+\alpha_2+\alpha_3+1\right) c_4^4-3\right)+3 \left(\alpha_3 c_3^5-5 c_3 \left(\left(\alpha_0+\alpha_1+\alpha_2+\alpha_3+1\right) c_4^4-1\right)+4 \left(\alpha_0+\alpha_1+\alpha_2+\alpha_3+1\right) c_4^5+4\right)\right)\right] \\ & /\left(60 \left(c_2-c_4\right) \left(c_4-c_1\right) \left(c_4-c_3\right) \left(c_2 c_3 c_4+c_1 \left(c_3 c_4+c_2 \left(c_3+c_4\right)\right)\right)\right) \\
\end{array}
$
\end{turn}
}
\end{table}

\section{Variable stepsize {\sc Limm-w} coefficients}
\label{sec:limmwvarcoeff}

Tables \ref{tbl:limmwvcoef1} -- \ref{tbl:limmwvcoef5} contain the variable stepsize coefficient expressions for {\sc Limm-w} methods of order 1-5, parameterized by the $c_i$'s defined in \eqref{eqn:abscissae}, and the constant coeffcients.

\begin{table}[H]
\centering
\caption{{\sc Limm-w} 1-step order 1 coefficients.}
\label{tbl:limmwvcoef1}
{\footnotesize
\renewcommand{\arraystretch}{2}
$\begin{array}{c|cc}
 i & -1 & 0 \\
\hline
 \alpha_i & 1 & -1 \\
 \beta_i & 0 & 1 \\
 \mu_i & 1 & -1 \\
\end{array}
$
}
\end{table}

\begin{table}[H]
\centering
\caption{{\sc Limm-w} 2-step order 2 variable stepsize coefficients.}
\label{tbl:limmwvcoef2}
{\footnotesize
\renewcommand{\arraystretch}{2}
$\begin{array}{c|ccc}
 i & -1 & 0 & 1 \\
\hline
 \alpha_i & 1 & -\frac{146619050}{133414177} & \frac{13204873}{133414177} \\
 \beta_i & 0 & \frac{1}{2} \left(\left(\alpha_0+1\right) c_1+\frac{1}{c_1}+2\right) & \frac{\left(\alpha_0+1\right) c_1^2-1}{2 c_1} \\
 \mu_i & \frac{1}{2} \left(1-\left(\alpha_0+1\right) c_1^2\right) & \frac{\left(c_1+1\right) \left(\left(\alpha_0+1\right) c_1^2-1\right)}{2 c_1} & \frac{1}{2 c_1}-\frac{1}{2} \left(\alpha_0+1\right) c_1 \\
\end{array}
$
}
\end{table}

\begin{table}[H]
\centering
\caption{{\sc Limm-w} 3-step order 3 variable stepsize coefficients.}
\label{tbl:limmwvcoef3}
{\footnotesize
\renewcommand{\arraystretch}{2}
$\begin{array}{c|c}
\hline
 \alpha_{-1} & 1 \\
 \beta_{-1} & 0 \\
 \mu_{-1} & \left[\alpha_1 c_1^3-3 c_1 \left(\left(\alpha_0+\alpha_1+1\right) c_2^2-1\right)+2 \left(\alpha_0+\alpha_1+1\right) c_2^3+2\right]/\left(6 \left(c_1+1\right)\right) \\
\hline
 \alpha_0 & -192592391/118869921 \\
 \beta_0 & \left[\alpha_1 c_1^3-3 \alpha_1 c_2 c_1^2+3 c_1 \left(\left(\alpha_0+\alpha_1+1\right) c_2^2+2 c_2+1\right)-\left(\alpha_0+\alpha_1+1\right) c_2^3+3 c_2+2\right]/\left(6 c_1 c_2\right) \\
 \mu_0 & -\left[\left(c_2+1\right) \left(\alpha_1 c_1^3-3 c_1 \left(\left(\alpha_0+\alpha_1+1\right) c_2^2-1\right)+2 \left(\alpha_0+\alpha_1+1\right) c_2^3+2\right)\right]/\left(6 c_1 c_2\right) \\
\hline
 \alpha_1 & 41981416/61945353 \\
 \beta_1 & \left[-2 \alpha_1 c_1^3+3 \alpha_1 c_2 c_1^2-\left(\alpha_0+\alpha_1+1\right) c_2^3+3 c_2+2\right]/\left(6 c_1 \left(c_1-c_2\right)\right) \\
 \mu_1 & -\left[\left(c_2+1\right) \left(\alpha_1 c_1^3-3 c_1 \left(\left(\alpha_0+\alpha_1+1\right) c_2^2-1\right)+2 \left(\alpha_0+\alpha_1+1\right) c_2^3+2\right)\right] \\ & /\left(6 c_1 \left(c_1+1\right) \left(c_1-c_2\right)\right) \\
\hline
 \alpha_2 & -5229175002546/90906657005273 \\
 \beta_2 & -\left[\alpha_1 c_1^3-3 c_1 \left(\left(\alpha_0+\alpha_1+1\right) c_2^2-1\right)+2 \left(\alpha_0+\alpha_1+1\right) c_2^3+2\right]/\left(6 \left(c_1-c_2\right) c_2\right) \\
 \mu_2 & \left[\alpha_1 c_1^3-3 c_1 \left(\left(\alpha_0+\alpha_1+1\right) c_2^2-1\right)+2 \left(\alpha_0+\alpha_1+1\right) c_2^3+2\right]/\left(6 \left(c_1-c_2\right) c_2\right) \\
\end{array}
$
}
\end{table}

\begin{table}[H]
\centering
\caption{{\sc Limm-w} 4-step order 4 variable stepsize coefficients.}
\label{tbl:limmwvcoef4}
{\tiny
\renewcommand{\arraystretch}{2}
\begin{turn}{-90}
$\begin{array}{c|l}
\hline
 \alpha_{-1} & 1 \\
 \beta_{-1} & 0 \\
 \mu_{-1} & \left[-\alpha_1 c_1^4+2 \alpha_1 c_2 c_1^3+2 c_1 \left(\alpha_2 c_2^3-3 c_2 \left(\left(\alpha_0+\alpha_1+\alpha_2+1\right) c_3^2-1\right)+2 \left(\alpha_0+\alpha_1+\alpha_2+1\right) c_3^3+2\right)-\alpha_2 c_2^4-3 \left(\alpha_0+\alpha_1+\alpha_2+1\right) c_3^4+4 c_2 \left(\left(\alpha_0+\alpha_1+\alpha_2+1\right) c_3^3+1\right)+3\right]/\left(12 \left(c_1+1\right) \left(c_2+1\right)\right) \\
\hline
 \alpha_0 & -68547635/35752838 \\
 \beta_0 & \left[-\alpha_1 c_1^4+2 \alpha_1 \left(c_2+c_3\right) c_1^3-6 \alpha_1 c_2 c_3 c_1^2+2 c_1 \left(\alpha_2 c_2^3-3 \alpha_2 c_3 c_2^2+3 c_2 \left(\left(\alpha_0+\alpha_1+\alpha_2+1\right) c_3^2+2 c_3+1\right)-\left(\alpha_0+\alpha_1+\alpha_2+1\right) c_3^3+3 c_3+2\right)+\alpha_0 c_3^4+\alpha_1 c_3^4-\alpha_2 c_2^4+\alpha_2 c_3^4+2 \alpha_2 c_2^3 c_3 \right. \\ & \left. -2 c_2 \left(\left(\alpha_0+\alpha_1+\alpha_2+1\right) c_3^3-3 c_3-2\right)+c_3^4+4 c_3+3\right]/\left(12 c_1 c_2 c_3\right) \\
 \mu_0 & \left[\left(c_3+1\right) \left(\alpha_1 c_1^4-2 \alpha_1 c_2 c_1^3-2 c_1 \left(\alpha_2 c_2^3-3 c_2 \left(\left(\alpha_0+\alpha_1+\alpha_2+1\right) c_3^2-1\right)+2 \left(\alpha_0+\alpha_1+\alpha_2+1\right) c_3^3+2\right)+\alpha_2 c_2^4+3 \left(\alpha_0+\alpha_1+\alpha_2+1\right) c_3^4-4 c_2 \left(\left(\alpha_0+\alpha_1+\alpha_2+1\right) c_3^3+1\right)-3\right)\right]/\left(12 c_1 c_2 c_3\right) \\
\hline
 \alpha_1 & 332147775/246829693 \\
 \beta_1 & -\left[3 \alpha_1 c_1^4-4 \alpha_1 \left(c_2+c_3\right) c_1^3+6 \alpha_1 c_2 c_3 c_1^2+\alpha_0 c_3^4+\alpha_1 c_3^4-\alpha_2 c_2^4+\alpha_2 c_3^4+2 \alpha_2 c_2^3 c_3-2 c_2 \left(\left(\alpha_0+\alpha_1+\alpha_2+1\right) c_3^3-3 c_3-2\right)+c_3^4+4 c_3+3\right]/\left(12 c_1 \left(c_1-c_2\right) \left(c_1-c_3\right)\right) \\
 \mu_1 & -\left[\left(c_3+1\right) \left(\alpha_1 c_1^4-2 \alpha_1 c_2 c_1^3-2 c_1 \left(\alpha_2 c_2^3-3 c_2 \left(\left(\alpha_0+\alpha_1+\alpha_2+1\right) c_3^2-1\right)+2 \left(\alpha_0+\alpha_1+\alpha_2+1\right) c_3^3+2\right)+\alpha_2 c_2^4+3 \left(\alpha_0+\alpha_1+\alpha_2+1\right) c_3^4-4 c_2 \left(\left(\alpha_0+\alpha_1+\alpha_2+1\right) c_3^3+1\right)-3\right)\right] \\ & /\left(12 c_1 \left(c_1+1\right) \left(c_1-c_2\right) \left(c_1-c_3\right)\right) \\
\hline
 \alpha_2 & -120323842/247754257 \\
 \beta_2 & -\left[-\alpha_1 c_1^4+2 \alpha_1 c_3 c_1^3-2 c_1 \left(2 \alpha_2 c_2^3-3 \alpha_2 c_3 c_2^2+\left(\alpha_0+\alpha_1+\alpha_2+1\right) c_3^3-3 c_3-2\right)+\alpha_0 c_3^4+\alpha_1 c_3^4+3 \alpha_2 c_2^4+\alpha_2 c_3^4-4 \alpha_2 c_2^3 c_3+c_3^4+4 c_3+3\right]/\left(12 c_2 \left(c_2-c_1\right) \left(c_2-c_3\right)\right) \\
 \mu_2 & \left[\left(c_3+1\right) \left(\alpha_1 c_1^4-2 \alpha_1 c_2 c_1^3-2 c_1 \left(\alpha_2 c_2^3-3 c_2 \left(\left(\alpha_0+\alpha_1+\alpha_2+1\right) c_3^2-1\right)+2 \left(\alpha_0+\alpha_1+\alpha_2+1\right) c_3^3+2\right)+\alpha_2 c_2^4+3 \left(\alpha_0+\alpha_1+\alpha_2+1\right) c_3^4-4 c_2 \left(\left(\alpha_0+\alpha_1+\alpha_2+1\right) c_3^3+1\right)-3\right)\right] \\ & /\left(12 \left(c_1-c_2\right) c_2 \left(c_2+1\right) \left(c_2-c_3\right)\right) \\
\hline
 \alpha_3 & 11382486133370227314625/198763375884603824550058 \\
 \beta_3 & \left[-\alpha_1 c_1^4+2 \alpha_1 c_2 c_1^3+2 c_1 \left(\alpha_2 c_2^3-3 c_2 \left(\left(\alpha_0+\alpha_1+\alpha_2+1\right) c_3^2-1\right)+2 \left(\alpha_0+\alpha_1+\alpha_2+1\right) c_3^3+2\right)-\alpha_2 c_2^4-3 \left(\alpha_0+\alpha_1+\alpha_2+1\right) c_3^4+4 c_2 \left(\left(\alpha_0+\alpha_1+\alpha_2+1\right) c_3^3+1\right)+3\right] \\ & /\left(12 \left(c_1-c_3\right) c_3 \left(c_3-c_2\right)\right) \\
 \mu_3 & \left[\alpha_1 c_1^4-2 \alpha_1 c_2 c_1^3-2 c_1 \left(\alpha_2 c_2^3-3 c_2 \left(\left(\alpha_0+\alpha_1+\alpha_2+1\right) c_3^2-1\right)+2 \left(\alpha_0+\alpha_1+\alpha_2+1\right) c_3^3+2\right)+\alpha_2 c_2^4+3 \left(\alpha_0+\alpha_1+\alpha_2+1\right) c_3^4-4 c_2 \left(\left(\alpha_0+\alpha_1+\alpha_2+1\right) c_3^3+1\right)-3\right] \\ & /\left(12 \left(c_1-c_3\right) c_3 \left(c_3-c_2\right)\right) \\
\end{array}
$
\end{turn}
}
\end{table}

\begin{table}[H]
\centering
\caption{{\sc Limm-w} 5-step order 5 variable stepsize coefficients.}
\label{tbl:limmwvcoef5}
{\fontsize{3pt}{4pt}\selectfont
\renewcommand{\arraystretch}{2}
\begin{turn}{-90}
$\begin{array}{c|l}
\hline
 \alpha_{-1} & 1 \\
 \beta_{-1} & 0 \\
 \mu_{-1} & \left[3 \alpha_1 c_1^5-5 \alpha_1 \left(c_2+c_3\right) c_1^4+10 \alpha_1 c_2 c_3 c_1^3-5 c_1 \left(\alpha_2 c_2^4-2 \alpha_2 c_3 c_2^3-2 c_2 \left(\alpha_3 c_3^3-3 c_3 \left(\left(\alpha_0+\alpha_1+\alpha_2+\alpha_3+1\right) c_4^2-1\right)+2 \left(\alpha_0+\alpha_1+\alpha_2+\alpha_3+1\right) c_4^3+2\right)+\alpha_3 c_3^4+3 \left(\alpha_0+\alpha_1+\alpha_2+\alpha_3+1\right) c_4^4 \right.\right. \\ & \left.\left. -4 c_3 \left(\left(\alpha_0+\alpha_1+\alpha_2+\alpha_3+1\right) c_4^3+1\right)-3\right)+3 \alpha_2 c_2^5-5 \alpha_2 c_2^4 c_3-5 c_2 \left(\alpha_3 c_3^4-4 c_3 \left(\left(\alpha_0+\alpha_1+\alpha_2+\alpha_3+1\right) c_4^3+1\right)+3 \left(\alpha_0+\alpha_1+\alpha_2+\alpha_3+1\right) c_4^4-3\right) \right. \\ & \left. +3 \left(\alpha_3 c_3^5-5 c_3 \left(\left(\alpha_0+\alpha_1+\alpha_2+\alpha_3+1\right) c_4^4-1\right)+4 \left(\alpha_0+\alpha_1+\alpha_2+\alpha_3+1\right) c_4^5+4\right)\right]/\left(60 \left(c_1+1\right) \left(c_2+1\right) \left(c_3+1\right)\right) \\
\hline
 \alpha_0 & -170476503/75237041 \\
 \beta_0 & \left[3 \alpha_1 c_1^5-5 \alpha_1 \left(c_2+c_3+c_4\right) c_1^4+10 \alpha_1 \left(c_3 c_4+c_2 \left(c_3+c_4\right)\right) c_1^3-30 \alpha_1 c_2 c_3 c_4 c_1^2-5 c_1 \left(\alpha_2 c_2^4-2 \alpha_2 \left(c_3+c_4\right) c_2^3+6 \alpha_2 c_3 c_4 c_2^2-2 c_2 \left(\alpha_3 c_3^3-3 \alpha_3 c_4 c_3^2+3 c_3 \left(\left(\alpha_0+\alpha_1+\alpha_2+\alpha_3+1\right) c_4^2+2 c_4+1\right) \right.\right.\right. \\ & \left.\left.\left. -\left(\alpha_0+\alpha_1+\alpha_2+\alpha_3+1\right) c_4^3+3 c_4+2\right)-\alpha_0 c_4^4-\alpha_1 c_4^4-\alpha_2 c_4^4+\alpha_3 c_3^4-\alpha_3 c_4^4-2 \alpha_3 c_3^3 c_4+2 c_3 \left(\left(\alpha_0+\alpha_1+\alpha_2+\alpha_3+1\right) c_4^3-3 c_4-2\right)-c_4^4-4 c_4-3\right)-3 \alpha_0 c_4^5+5 \alpha_0 c_3 c_4^4-3 \alpha_1 c_4^5 \right. \\ & \left. +5 \alpha_1 c_3 c_4^4+3 \alpha_2 c_2^5-3 \alpha_2 c_4^5+5 \alpha_2 c_3 c_4^4+10 \alpha_2 c_2^3 c_3 c_4-5 \alpha_2 c_2^4 \left(c_3+c_4\right)+3 \alpha_3 c_3^5-3 \alpha_3 c_4^5+5 \alpha_3 c_3 c_4^4-5 \alpha_3 c_3^4 c_4-5 c_2 \left(\alpha_3 c_3^4-2 \alpha_3 c_4 c_3^3+2 c_3 \left(\left(\alpha_0+\alpha_1+\alpha_2+\alpha_3+1\right) c_4^3-3 c_4-2\right) \right.\right. \\ & \left.\left. -\left(\alpha_0+\alpha_1+\alpha_2+\alpha_3+1\right) c_4^4-4 c_4-3\right)-3 c_4^5+5 c_3 c_4^4+15 c_3+20 c_3 c_4+15 c_4+12\right]/\left(60 c_1 c_2 c_3 c_4\right) \\
 \mu_0 & -\left[\left(c_4+1\right) \left(3 \alpha_1 c_1^5-5 \alpha_1 \left(c_2+c_3\right) c_1^4+10 \alpha_1 c_2 c_3 c_1^3-5 c_1 \left(\alpha_2 c_2^4-2 \alpha_2 c_3 c_2^3-2 c_2 \left(\alpha_3 c_3^3-3 c_3 \left(\left(\alpha_0+\alpha_1+\alpha_2+\alpha_3+1\right) c_4^2-1\right)+2 \left(\alpha_0+\alpha_1+\alpha_2+\alpha_3+1\right) c_4^3+2\right)+\alpha_3 c_3^4 \right.\right.\right. \\ & \left.\left.\left. +3 \left(\alpha_0+\alpha_1+\alpha_2+\alpha_3+1\right) c_4^4-4 c_3 \left(\left(\alpha_0+\alpha_1+\alpha_2+\alpha_3+1\right) c_4^3+1\right)-3\right)+3 \alpha_2 c_2^5-5 \alpha_2 c_2^4 c_3-5 c_2 \left(\alpha_3 c_3^4-4 c_3 \left(\left(\alpha_0+\alpha_1+\alpha_2+\alpha_3+1\right) c_4^3+1\right)+3 \left(\alpha_0+\alpha_1+\alpha_2+\alpha_3+1\right) c_4^4-3\right) \right.\right. \\ & \left.\left. +3 \left(\alpha_3 c_3^5-5 c_3 \left(\left(\alpha_0+\alpha_1+\alpha_2+\alpha_3+1\right) c_4^4-1\right)+4 \left(\alpha_0+\alpha_1+\alpha_2+\alpha_3+1\right) c_4^5+4\right)\right)\right]/\left(60 c_1 c_2 c_3 c_4\right) \\
\hline
 \alpha_1 & 124149029/52265116 \\
 \beta_1 & \left[-12 \alpha_1 c_1^5+15 \alpha_1 \left(c_2+c_3+c_4\right) c_1^4-20 \alpha_1 \left(c_3 c_4+c_2 \left(c_3+c_4\right)\right) c_1^3+30 \alpha_1 c_2 c_3 c_4 c_1^2-3 \alpha_0 c_4^5+5 \alpha_0 c_3 c_4^4-3 \alpha_1 c_4^5+5 \alpha_1 c_3 c_4^4+3 \alpha_2 c_2^5-3 \alpha_2 c_4^5+5 \alpha_2 c_3 c_4^4+10 \alpha_2 c_2^3 c_3 c_4-5 \alpha_2 c_2^4 \left(c_3+c_4\right) \right. \\ & \left. +3 \alpha_3 c_3^5-3 \alpha_3 c_4^5+5 \alpha_3 c_3 c_4^4-5 \alpha_3 c_3^4 c_4-5 c_2 \left(\alpha_3 c_3^4-2 \alpha_3 c_4 c_3^3+2 c_3 \left(\left(\alpha_0+\alpha_1+\alpha_2+\alpha_3+1\right) c_4^3-3 c_4-2\right)-\left(\alpha_0+\alpha_1+\alpha_2+\alpha_3+1\right) c_4^4-4 c_4-3\right)-3 c_4^5+5 c_3 c_4^4+15 c_3+20 c_3 c_4 \right. \\ & \left. +15 c_4+12\right]/\left(60 c_1 \left(c_1-c_2\right) \left(c_1-c_3\right) \left(c_1-c_4\right)\right) \\
 \mu_1 & -\left[\left(c_4+1\right) \left(3 \alpha_1 c_1^5-5 \alpha_1 \left(c_2+c_3\right) c_1^4+10 \alpha_1 c_2 c_3 c_1^3-5 c_1 \left(\alpha_2 c_2^4-2 \alpha_2 c_3 c_2^3-2 c_2 \left(\alpha_3 c_3^3-3 c_3 \left(\left(\alpha_0+\alpha_1+\alpha_2+\alpha_3+1\right) c_4^2-1\right)+2 \left(\alpha_0+\alpha_1+\alpha_2+\alpha_3+1\right) c_4^3+2\right)+\alpha_3 c_3^4 \right.\right.\right. \\ & \left.\left.\left. +3 \left(\alpha_0+\alpha_1+\alpha_2+\alpha_3+1\right) c_4^4-4 c_3 \left(\left(\alpha_0+\alpha_1+\alpha_2+\alpha_3+1\right) c_4^3+1\right)-3\right)+3 \alpha_2 c_2^5-5 \alpha_2 c_2^4 c_3-5 c_2 \left(\alpha_3 c_3^4-4 c_3 \left(\left(\alpha_0+\alpha_1+\alpha_2+\alpha_3+1\right) c_4^3+1\right)+3 \left(\alpha_0+\alpha_1+\alpha_2+\alpha_3+1\right) c_4^4-3\right) \right.\right. \\ & \left.\left. +3 \left(\alpha_3 c_3^5-5 c_3 \left(\left(\alpha_0+\alpha_1+\alpha_2+\alpha_3+1\right) c_4^4-1\right)+4 \left(\alpha_0+\alpha_1+\alpha_2+\alpha_3+1\right) c_4^5+4\right)\right)\right]/\left(60 c_1 \left(c_1+1\right) \left(c_1-c_2\right) \left(c_1-c_3\right) \left(c_1-c_4\right)\right) \\
\hline
 \alpha_2 & -53697673/39342191 \\
 \beta_2 & \left[3 \alpha_1 c_1^5-5 \alpha_1 \left(c_3+c_4\right) c_1^4+10 \alpha_1 c_3 c_4 c_1^3+5 c_1 \left(3 \alpha_2 c_2^4-4 \alpha_2 \left(c_3+c_4\right) c_2^3+6 \alpha_2 c_3 c_4 c_2^2+\alpha_0 c_4^4+\alpha_1 c_4^4+\alpha_2 c_4^4-\alpha_3 c_3^4+\alpha_3 c_4^4+2 \alpha_3 c_3^3 c_4-2 c_3 \left(\left(\alpha_0+\alpha_1+\alpha_2+\alpha_3+1\right) c_4^3-3 c_4-2\right)+c_4^4 \right.\right. \\ & \left.\left. +4 c_4+3\right)-3 \alpha_0 c_4^5+5 \alpha_0 c_3 c_4^4-3 \alpha_1 c_4^5+5 \alpha_1 c_3 c_4^4-12 \alpha_2 c_2^5-3 \alpha_2 c_4^5+5 \alpha_2 c_3 c_4^4-20 \alpha_2 c_2^3 c_3 c_4+15 \alpha_2 c_2^4 \left(c_3+c_4\right)+3 \alpha_3 c_3^5-3 \alpha_3 c_4^5+5 \alpha_3 c_3 c_4^4-5 \alpha_3 c_3^4 c_4-3 c_4^5+5 c_3 c_4^4+15 c_3+20 c_3 c_4 \right. \\ & \left. +15 c_4+12\right]/\left(60 c_2 \left(c_2-c_1\right) \left(c_2-c_3\right) \left(c_2-c_4\right)\right) \\
 \mu_2 & -\left[\left(c_4+1\right) \left(3 \alpha_1 c_1^5-5 \alpha_1 \left(c_2+c_3\right) c_1^4+10 \alpha_1 c_2 c_3 c_1^3-5 c_1 \left(\alpha_2 c_2^4-2 \alpha_2 c_3 c_2^3-2 c_2 \left(\alpha_3 c_3^3-3 c_3 \left(\left(\alpha_0+\alpha_1+\alpha_2+\alpha_3+1\right) c_4^2-1\right)+2 \left(\alpha_0+\alpha_1+\alpha_2+\alpha_3+1\right) c_4^3+2\right)+\alpha_3 c_3^4 \right.\right.\right. \\ & \left.\left.\left. +3 \left(\alpha_0+\alpha_1+\alpha_2+\alpha_3+1\right) c_4^4-4 c_3 \left(\left(\alpha_0+\alpha_1+\alpha_2+\alpha_3+1\right) c_4^3+1\right)-3\right)+3 \alpha_2 c_2^5-5 \alpha_2 c_2^4 c_3-5 c_2 \left(\alpha_3 c_3^4-4 c_3 \left(\left(\alpha_0+\alpha_1+\alpha_2+\alpha_3+1\right) c_4^3+1\right)+3 \left(\alpha_0+\alpha_1+\alpha_2+\alpha_3+1\right) c_4^4-3\right) \right.\right. \\ & \left.\left. +3 \left(\alpha_3 c_3^5-5 c_3 \left(\left(\alpha_0+\alpha_1+\alpha_2+\alpha_3+1\right) c_4^4-1\right)+4 \left(\alpha_0+\alpha_1+\alpha_2+\alpha_3+1\right) c_4^5+4\right)\right)\right]/\left(60 c_2 \left(c_2+1\right) \left(c_2-c_1\right) \left(c_2-c_3\right) \left(c_2-c_4\right)\right) \\
\hline
 \alpha_3 & 67073128/206463953 \\
 \beta_3 & \left[3 \alpha_1 c_1^5-5 \alpha_1 \left(c_2+c_4\right) c_1^4+10 \alpha_1 c_2 c_4 c_1^3-5 c_1 \left(\alpha_2 c_2^4-2 \alpha_2 c_4 c_2^3+2 c_2 \left(2 \alpha_3 c_3^3-3 \alpha_3 c_4 c_3^2+\left(\alpha_0+\alpha_1+\alpha_2+\alpha_3+1\right) c_4^3-3 c_4-2\right)-\alpha_0 c_4^4-\alpha_1 c_4^4-\alpha_2 c_4^4-3 \alpha_3 c_3^4-\alpha_3 c_4^4+4 \alpha_3 c_3^3 c_4-c_4^4 \right.\right. \\ & \left.\left. -4 c_4-3\right)+3 \alpha_2 c_2^5-5 \alpha_2 c_2^4 c_4+5 c_2 \left(3 \alpha_3 c_3^4-4 \alpha_3 c_4 c_3^3+\left(\alpha_0+\alpha_1+\alpha_2+\alpha_3+1\right) c_4^4+4 c_4+3\right)-3 \left(4 \alpha_3 c_3^5-5 \alpha_3 c_4 c_3^4+\left(\alpha_0+\alpha_1+\alpha_2+\alpha_3+1\right) c_4^5-5 c_4-4\right)\right] \\ & /\left(60 c_3 \left(c_3-c_1\right) \left(c_3-c_2\right) \left(c_3-c_4\right)\right) \\
 \mu_3 & -\left[\left(c_4+1\right) \left(3 \alpha_1 c_1^5-5 \alpha_1 \left(c_2+c_3\right) c_1^4+10 \alpha_1 c_2 c_3 c_1^3-5 c_1 \left(\alpha_2 c_2^4-2 \alpha_2 c_3 c_2^3-2 c_2 \left(\alpha_3 c_3^3-3 c_3 \left(\left(\alpha_0+\alpha_1+\alpha_2+\alpha_3+1\right) c_4^2-1\right)+2 \left(\alpha_0+\alpha_1+\alpha_2+\alpha_3+1\right) c_4^3+2\right)+\alpha_3 c_3^4 \right.\right.\right. \\ & \left.\left.\left. +3 \left(\alpha_0+\alpha_1+\alpha_2+\alpha_3+1\right) c_4^4-4 c_3 \left(\left(\alpha_0+\alpha_1+\alpha_2+\alpha_3+1\right) c_4^3+1\right)-3\right)+3 \alpha_2 c_2^5-5 \alpha_2 c_2^4 c_3-5 c_2 \left(\alpha_3 c_3^4-4 c_3 \left(\left(\alpha_0+\alpha_1+\alpha_2+\alpha_3+1\right) c_4^3+1\right)+3 \left(\alpha_0+\alpha_1+\alpha_2+\alpha_3+1\right) c_4^4-3\right) \right.\right. \\ & \left.\left. +3 \left(\alpha_3 c_3^5-5 c_3 \left(\left(\alpha_0+\alpha_1+\alpha_2+\alpha_3+1\right) c_4^4-1\right)+4 \left(\alpha_0+\alpha_1+\alpha_2+\alpha_3+1\right) c_4^5+4\right)\right)\right]/\left(60 c_3 \left(c_3+1\right) \left(c_3-c_1\right) \left(c_3-c_2\right) \left(c_3-c_4\right)\right) \\
\hline
 \alpha_4 & -2219582774479398588921363466455/31940845355796541711865631316388 \\
 \beta_4 & \left[3 \alpha_1 c_1^5-5 \alpha_1 \left(c_2+c_3\right) c_1^4+10 \alpha_1 c_2 c_3 c_1^3-5 c_1 \left(\alpha_2 c_2^4-2 \alpha_2 c_3 c_2^3-2 c_2 \left(\alpha_3 c_3^3-3 c_3 \left(\left(\alpha_0+\alpha_1+\alpha_2+\alpha_3+1\right) c_4^2-1\right)+2 \left(\alpha_0+\alpha_1+\alpha_2+\alpha_3+1\right) c_4^3+2\right)+\alpha_3 c_3^4+3 \left(\alpha_0+\alpha_1+\alpha_2+\alpha_3+1\right) c_4^4 \right.\right. \\ & \left.\left. -4 c_3 \left(\left(\alpha_0+\alpha_1+\alpha_2+\alpha_3+1\right) c_4^3+1\right)-3\right)+3 \alpha_2 c_2^5-5 \alpha_2 c_2^4 c_3-5 c_2 \left(\alpha_3 c_3^4-4 c_3 \left(\left(\alpha_0+\alpha_1+\alpha_2+\alpha_3+1\right) c_4^3+1\right)+3 \left(\alpha_0+\alpha_1+\alpha_2+\alpha_3+1\right) c_4^4-3\right) \right. \\ & \left. +3 \left(\alpha_3 c_3^5-5 c_3 \left(\left(\alpha_0+\alpha_1+\alpha_2+\alpha_3+1\right) c_4^4-1\right)+4 \left(\alpha_0+\alpha_1+\alpha_2+\alpha_3+1\right) c_4^5+4\right)\right]/\left(60 c_4 \left(c_4-c_1\right) \left(c_4-c_2\right) \left(c_4-c_3\right)\right) \\
 \mu_4 & \left[3 \alpha_1 c_1^5-5 \alpha_1 \left(c_2+c_3\right) c_1^4+10 \alpha_1 c_2 c_3 c_1^3-5 c_1 \left(\alpha_2 c_2^4-2 \alpha_2 c_3 c_2^3-2 c_2 \left(\alpha_3 c_3^3-3 c_3 \left(\left(\alpha_0+\alpha_1+\alpha_2+\alpha_3+1\right) c_4^2-1\right)+2 \left(\alpha_0+\alpha_1+\alpha_2+\alpha_3+1\right) c_4^3+2\right)+\alpha_3 c_3^4+3 \left(\alpha_0+\alpha_1+\alpha_2+\alpha_3+1\right) c_4^4 \right.\right. \\ & \left.\left. -4 c_3 \left(\left(\alpha_0+\alpha_1+\alpha_2+\alpha_3+1\right) c_4^3+1\right)-3\right)+3 \alpha_2 c_2^5-5 \alpha_2 c_2^4 c_3-5 c_2 \left(\alpha_3 c_3^4-4 c_3 \left(\left(\alpha_0+\alpha_1+\alpha_2+\alpha_3+1\right) c_4^3+1\right)+3 \left(\alpha_0+\alpha_1+\alpha_2+\alpha_3+1\right) c_4^4-3\right) \right. \\ & \left. +3 \left(\alpha_3 c_3^5-5 c_3 \left(\left(\alpha_0+\alpha_1+\alpha_2+\alpha_3+1\right) c_4^4-1\right)+4 \left(\alpha_0+\alpha_1+\alpha_2+\alpha_3+1\right) c_4^5+4\right)\right]/\left(60 \left(c_1-c_4\right) c_4 \left(c_4-c_2\right) \left(c_4-c_3\right)\right) \\
\end{array}
$
\end{turn}
}
\end{table}

\fi 

\bibliographystyle{siamplain}
\bibliography{Bib/ode_general,Bib/ode_imex,Bib/ode_krylov,Bib/sandu,Bib/ode_exponential,limm_bibliography}

\end{document}